\documentclass[reqno]{amsart}
\usepackage[T1]{fontenc}
\usepackage[utf8]{inputenc}
\usepackage{lmodern}
\usepackage{amsmath,amssymb,mathtools,mathrsfs}
\usepackage{microtype,enumitem,needspace}
\usepackage{tikz}
\usetikzlibrary{arrows.meta}

\usepackage[colorlinks=true,linkcolor=orange,
 citecolor=orange,urlcolor=red!55!black]{hyperref}

\usepackage[backend=biber,style=trad-abbrv,maxnames=99,maxalphanames=9, isbn=false, giveninits=true, doi=false, url=true]{biblatex}
\renewbibmacro{in:}{}
\allowdisplaybreaks[1]
\numberwithin{equation}{section}
\newtheorem{theorem}{Theorem}[section]
\newtheorem{proposition}[theorem]{Proposition}
\newtheorem{lemma}[theorem]{Lemma}
\newtheorem{corollary}[theorem]{Corollary}
\theoremstyle{definition}
\newtheorem{definition}[theorem]{Definition}
\newtheorem{remark}[theorem]{Remark}
\newtheorem*{conjecture}{Conjecture}
\setlist[enumerate,1]{label=\textup{(\arabic*)},ref=\arabic*,leftmargin=2em,itemsep=.2\baselineskip}
\DeclareMathOperator{\tr}{tr}
\DeclareMathOperator{\Null}{Null}

\DeclareMathOperator{\Pole}{Pole}

\DeclareMathOperator{\vol}{vol}
\newcommand{\ddc}{dd^c}
\newcommand{\cC}{\mathcal C}
\newcommand{\cI}{\mathcal I}
\newcommand{\cO}{\mathcal O}
\newcommand{\cab}{c_{\alpha,\beta}}
\newcommand{\coc}{c_{\omega,\chi}}
\newcommand{\cld}{c_{L,D}}
\newcommand{\R}{\mathbb R}

\title[The J-equation at the birational minimal slope]{The $J$-equation at the birational minimal slope}
\author{Junbang Liu}
\address{Department of Mathematics, The Hong Kong University of Science and Technology, Clear Water Bay, Kowloon, Hong Kong}
\email{junbangliu@ust.hk}
\date{}
\hypersetup{pdftitle={The J-equation at the birational minimal slope},pdfauthor={Junbang Liu}}
\begin{document}
\begin{abstract}
We prove the existence, uniqueness, and partial regularity outside a proper analytic subset for the K\"ahler current solving the $J$-equation at the birational minimal slope. This confirms Datar--Mete--Song's conjecture 1.5 in \cite{DMS26}. We introduce an analytic threshold, and prove its equivalence to the birational threshold introduced by Datar--Mete--Song. One of the key tools is the approximation of subsolution by Bergman's kernels, which is motivated by the work of Demailly on the approximation of plurisubharmonic functions with analytic singularities. As an application of the Bergman kernel approximation, we combine the results of Fang--Ma \cite{FM26} to give an analytic characterization of the $J$-null locus of a semistable pair $(\alpha,\beta)$. This removes the technical assumptions in \cite{L26b}.
\end{abstract}
\maketitle
\setcounter{tocdepth}{1}
{
  \hypersetup{linkcolor=black}
  \tableofcontents
}

\section{Introduction}\label{sec:intro}

Let $X$ be a connected compact K\"ahler manifold of dimension $n\ge1$, let $\alpha,\beta$ be K\"ahler classes, and fix a K\"ahler form $\omega\in\beta$. The $J$-equation seeks a K\"ahler form $\eta\in\alpha$ such that
\begin{equation}\label{eq:j-equation}
 n\omega\wedge\eta^{n-1}=c\eta^n,
 \qquad c=\cab=n{\alpha^{n-1}\beta}/{\alpha^n}.
\end{equation}
Equivalently, $\tr_\eta\omega=c$. Donaldson introduced the equation
through a moment-map construction \cite{D99}, and X.~X.~Chen through
the study of the Mabuchi energy \cite{C00,C04}. Its solutions are
critical points of the $J$-functional, whose lower bounds and
coercivity enter the study of constant scalar curvature K\"ahler
metrics. There is an extensive literature on the $J$-equation,
the $J$-flow, and related inverse Hessian equations, including
smooth solvability, convergence, and degeneration; see
\cite{W04,W06,SW08,FLM11,FL12,FL13,SW13,FLSW14,LS15,CS17,S18,
SD20,C21,DP21,S20,To23,FM24,GS24,M26a,M26b,F26,FZ26,
L26a,L26b} and the references therein.

For $n\ge2$, the analytic criterion is the existence of a K\"ahler
form $\eta_0\in\alpha$ satisfying the smooth cone condition
\[
 c\eta_0^{n-1}-(n-1)\omega\wedge\eta_0^{n-2}>0;
\]
see Weinkove \cite{W04,W06} and Song--Weinkove \cite{SW08}.
Lejmi--Sz\'ekelyhidi \cite{LS15} conjectured it is equivalent to the numerical intersection inequalities
\begin{equation}\label{eq:j-stability}
 (c\alpha^p-p\beta\alpha^{p-1})[V]>0
 \quad\text{for every irreducible }V\subsetneq X,\quad
 1\le p=\dim V<n.
\end{equation}
The numerical criterion was established in the toric case by
Collins--Sz\'ekelyhidi \cite{CS17}, under uniform numerical
positivity by G.~Chen \cite{C21}, and in its strict form by
Datar--Pingali \cite{DP21} in the projective case and Song
\cite{S20} in the K\"ahler setting. Replacing $>$ by $\ge$
in \eqref{eq:j-stability} is the so called
$J$-semistability, or $J$-nefness. Subvarieties where
equality holds are called $J$-null. The semistable $J$-equation was first studied by Fang-Lai-Song-Weinkove \cite{FLSW14} on K\'ahler surfaces, where they established the convergence of the $J$-flow under a smooth boundary cone condition. Murakami proved the convergence under $J$-semistability alone and obtained the existence and uniqueness of admissible weak solution \cite{M26a,M26b}. Recently, the author \cite{L26a} proved that in dimension $3$, there was an analytic characterization of the $J$-null subvarieties, and showed the partial regularity of the weak solution obtained by Murakami outside the $J$-null locus. Higher dimensional case was proved under a smooth boundary cone condition and $J$-bigness assumption in \cite{L26b}. The $J$-bigness assumption was removed by the work of Fu-Zhang \cite{FZ26}, and they also proved a uniform $L^\infty$-estimate under the smooth cone boundary condition. 
The study of the resulting
degeneration also leads to stability thresholds, optimal
destabilizing subvarieties, and finite numerical tests. These
questions have been developed by Sj\"ostr\"om Dyrefelt \cite{SD20},
Khalid--Sj\"ostr\"om Dyrefelt \cite{KD24,KD26}, and
Sivaram--Sj\"ostr\"om Dyrefelt \cite{SSD26}; finiteness of null
subvarieties in the semistable case and optimal destabilizers
in the unstable case, in all dimensions, is established in
\cite{L26c}. Very recently, Fang-Ma \cite{FM26} proved that in the semistable case, the weak solution obtained by Murakami~\cite{M26b} is smooth outside the $J$-null locus. Moreover, they showed that $X\setminus \text{Null}_J$ is the largest locus of $C^2$-regularity of the weak solution.

The failure of \eqref{eq:j-stability} rules
out a smooth solution, but it leaves open the existence of a
canonical singular solution. Datar--Mete--Song \cite[Sections~1.1
and~1.2]{DMS26} formulated this problem by analogy with the complex
Monge--Amp\`ere equation in a big class. To recall that analogy,
let $\gamma$ first be a K\"ahler class. Yau's theorem \cite{Y78}
gives a unique K\"ahler form $S\in\gamma$ with
\[
 S^n=\frac{\gamma^n}{\beta^n}\,\omega^n.
\]
The numerical characterization of the K\"ahler cone
\cite{DP04} describes precisely when a smooth positive
representative exists. If $\gamma$ is merely big, a smooth
K\"ahler representative may no longer exist. Nevertheless,
Boucksom--Eyssidieux--Guedj--Zeriahi \cite{BEGZ10} construct a
unique positive current $S\in\gamma$ of full non-pluripolar mass satisfying
\[
 \langle S^n\rangle
 =\frac{\vol(\gamma)}{\beta^n}\,\omega^n,\qquad
 \int_X\langle S^n\rangle=\vol(\gamma).
\]
The brackets denote the non-pluripolar product.
The normalization is now governed by the volume of the class,
which has the birational description \cite{B02}
\[
 \vol(\gamma)
 =\sup_{\substack{\pi:Y\to X,\ D\ge0\\
                  \pi^*\gamma-[D]\ {\rm big\ and\ nef}}}
          (\pi^*\gamma-[D])^n.
\]
Here $\pi$ ranges over modifications from smooth compact
K\"ahler manifolds and $D$ over effective real divisors.
Thus the normalization of the weak equation records the positive
classes remaining after divisorial singularities are removed.

For the $J$-equation, $\alpha$ and $\beta$ remain K\"ahler,
but their relative numerical inequalities may fail. If a K\"ahler
current $T\in\alpha$ solves the scalar equation at a constant
$a>0$, then integration gives
\[
 a=\frac{n\int_X\omega\wedge\langle T^{n-1}\rangle}
          {\int_X\langle T^n\rangle}.
\]
These non-pluripolar masses need not equal the corresponding
cohomological intersections, so $a$ need not equal $c$.
On a modification, removing an effective divisor from
$\pi^*\alpha$ gives a residual class and its $J$-slope.
This motivates the birational minimal slope of
Datar--Mete--Song \cite[Definition~1.3]{DMS26}:
\begin{equation}\label{eq:minimal-slope}
 \zeta=\inf_{\pi,D}c_L,\qquad
 c_L=n\frac{L^{n-1}\pi^*\beta}{L^n},\qquad
 L=\pi^*\alpha-[D].
\end{equation}
The infimum is over modifications $\pi:Y\to X$ from smooth
compact complex manifolds and effective real divisors $D$ for
which $L$ is big and nef. Lemma~\ref{lem:kahler-domination}
permits restriction to K\"ahler modifications and K\"ahler
residual classes. The reciprocal $1/\zeta$ is the analogue
of the volume in the preceding Monge--Amp\`ere problem.
Datar-Mete-Song made the following two conjectures \cite{DMS26}

\begin{conjecture}[Datar--Mete--Song {\cite[Conjecture~1.5]{DMS26}}]
For every K\"ahler form $\omega\in\beta$, there is a unique
K\"ahler current $T\in\alpha$ such that
\[
 n\omega\wedge\langle T^{n-1}\rangle=\zeta\langle T^n\rangle.
\]
\end{conjecture}

\begin{conjecture}[Datar--Mete--Song {\cite[Conjecture~2.14]{DMS26}}]
The pair $(\alpha,\beta)$ is numerically $J$-semistable
if and only if $\zeta=c$.
\end{conjecture}

Datar--Mete--Song proved both statements on surfaces
\cite[Theorem~1.4]{DMS26}. Fu \cite{F26} established the second
conjecture in arbitrary dimension, and Murakami \cite{M26b}
constructed weak solutions in the numerically semistable case.
Our main theorem confirms Conjecture~1.5 in every dimension
and proves smoothness outside a proper analytic subset, without
assuming $J$-semistability.

\begin{theorem}\label{thm:main}
There is a unique K\"ahler current $T\in\alpha$ satisfying
\begin{equation}\label{eq:main-equation}
 n\omega\wedge\langle T^{n-1}\rangle=\zeta\langle T^n\rangle,
 \qquad 0<\zeta\le c.
\end{equation}
Moreover, $T\ge\omega/\zeta$, and there is a proper closed analytic subset $Z\subset X$ such that
\[
 T|_{X\setminus Z}\text{ is a smooth K\"ahler form},\qquad
 \tr_T\omega=\zeta\quad\text{on }X\setminus Z.
\] 
\end{theorem}

All products in \eqref{eq:main-equation} are non-pluripolar. We write $T_{\rm ac}$ for the absolutely continuous part of the coefficient measures with respect to smooth volume. The conjecture imposes no additional cone condition. We prove that every scalar K\"ahler-current solution at a positive constant $a$ satisfies
\begin{equation}\label{eq:admissibility}
 a\langle T^p\rangle-p\omega\wedge\langle T^{p-1}\rangle\ge0
 \quad(1\le p<n),
 \qquad \tr_{T_{\rm ac}}\omega=a\quad\text{a.e.}
\end{equation}
Thus admissibility, in the sense of \eqref{eq:admissibility}, follows from the scalar equation.
Corollary~\ref{cor:semistable} also recovers
the implication $J$-semistability $\Longrightarrow\zeta=c$
by a trace comparison. In the unstable case, Fu's theorem gives
$\zeta<c$. Theorem~\ref{thm:main}
therefore treats both the boundary and unstable cases of the
equation. Datar--Mete--Song further predict that the canonical
current arises as a limit of the $J$-flow and describe associated
bubbling phenomena \cite[Section~1.2]{DMS26}. We prove this
convergence for every smooth initial metric and obtain the regularity
assertion of Theorem~\ref{thm:main} from the local smooth convergence.

\begin{theorem}\label{thm:flow}Fix K\"ahler forms $\kappa\in \alpha, \omega\in \beta$. 
For every smooth $\varphi_0$ with $\kappa+\ddc\varphi_0>0$, the solution of
\begin{equation}\label{eq:flow}
 \partial_t\varphi=c-\tr_{\varphi}\omega,\qquad \varphi(0)=\varphi_0,
\end{equation}
exists for all $t\ge0$. The analytic set $Z$ in
Theorem~\ref{thm:main} can be chosen using only
$(X,\alpha,\beta,\kappa,\omega)$ so that
\begin{equation}\label{eq:flow-convergence}
 \begin{aligned}
\kappa_{\varphi(t)}:= \kappa+\ddc\varphi(t)&\rightharpoonup T
       &&\text{on }X,\\
 \kappa+\ddc\varphi(t)&\longrightarrow T
       &&\text{in }C^\infty_{\rm loc}(X\setminus Z).
 \end{aligned}
\end{equation}
For
\[
 b(t)=\frac1{\beta^n}\int_X\varphi(t)\,\omega^n,
\]
the potentials $\varphi(t)-b(t)$ converge in $L^1(X)$ and
$C^\infty_{\rm loc}(X\setminus Z)$ to the potential of $T$ with
zero $\omega^n$-mean. Moreover,
\[
 \frac{b(t)}t,\ b'(t)\longrightarrow c-\zeta,\qquad
 \tr_{\varphi(t)}\omega\longrightarrow\zeta
       \quad\text{in }L^p(X,\omega^n),\quad 1\le p<\infty.
\]

\end{theorem}
Here we use the $\tr_\varphi\omega$ to denote the trace of $\omega$ with respect to $\kappa_\varphi$. 

An interesting corollary in the proof of the convergence is that, when the pair $(\alpha,\beta)$ is semistable, the trace is indeed converging in $L^\infty(X)$ to $c$, even along the singular set $Z$; see Corollary~\ref{cor:uniform-trace}. Similar phenomenon also appears in the study of K\"ahler Ricci flow with semi-ample canonical line bundle. The manifold there admits a Calabi-Yau fibration (with possible singular fibers) over its canonical model. The scalar curvature converges uniformly even along the singular fiber \cite{J20,ZZ26}.

The regularization of trace-constrained currents is also of
independent interest. The following theorem refines Bergman
regularization \cite{D92,D12} by retaining the inverse-trace
bound up to an arbitrarily small error. Here
\emph{analytic singularity type} means that local potentials
have the form
\[
 b\log\Bigl(\sum_{\ell=1}^N|f_\ell|^2\Bigr)+O(1),
 \qquad b>0,\quad f_\ell\text{ holomorphic}.
\]

\begin{theorem}\label{thm:regularization}
Let $T\in\alpha$ be a closed positive current with
$T_{\rm ac}>0$ and $\tr_{T_{\rm ac}}\omega\le a$
almost everywhere, where $a>0$.
For every $\varepsilon>0$, there is a K\"ahler current $T_\varepsilon\in\alpha$,
with analytic singularity type, such that
\[
 \tr_{(T_\varepsilon)_{\rm ac}}\omega\le a+\varepsilon
 \quad\text{almost everywhere}.
\]
They can be chosen with $T_\varepsilon\rightharpoonup T$
as $\varepsilon\downarrow0$.
Each $T_\varepsilon$ is smooth outside a proper analytic set.
There is a finite composition of blowups with smooth centers
$\pi_\varepsilon:Y_\varepsilon\to X$ for which
\[
 \pi_\varepsilon^*T_\varepsilon
   =\sigma_\varepsilon+[D_\varepsilon],\qquad
 D_\varepsilon\ge0,\qquad
 \sigma_\varepsilon\ge
       \frac{\pi_\varepsilon^*\omega}{a+\varepsilon},
\]
Here $D_\varepsilon$ is an effective real divisor and
$\sigma_\varepsilon$ is a smooth closed semipositive form.
\end{theorem}

This is Theorem~\ref{thm:bergman}.
The remainder in the local logarithmic expression is bounded
on $X$; smoothness of the remainder is asserted after
principalization. The same construction preserves the
$(n-1)$-dimensional inverse-trace bound
(Lemma~\ref{lem:partial-trace}), which
is used to produce a strict subsolution for the regularity
argument.

Inspired by recent work of Fang--Ma \cite{FM26}, we give an application of the above Bergman kernel approximation. We show that under the assumption of $J$-semistability alone, the $J$-null subvariety can be characterized by strict subsolution with analytic singularity type with smooth remainder after principalization. \begin{theorem}
There exists $B\in\mathcal R_\omega$ with $\Pole(B)=\Null_J(\alpha,\beta)$.
Consequently,
\begin{equation}
 E_{\rm res}(\omega)=\Null_J(\alpha,\beta).
\end{equation}
Here $\mathcal{R}_\omega$ is the set of K\"ahler current $S$, satisfying $P_S(\omega)\le c-\varepsilon$ for some $\varepsilon>0$, having analytic singularity type with smooth remainder after principalization. And $E_{res}(\omega)=\cap_{S\in \mathcal{R}_\omega}\Pole(S)$.
\end{theorem}

The thoerem removes the technical assumptions on both the smooth boudary cone condition, and $J$-bigness in \cite{L26b}, but the conclusion is a little bit weaker than that. In the case of \cite{L26b}, the K\"ahler current can be choosen to have analytic singularity and smooth remainder on $X$ itself. As a consequence of the theorem, the $J$-flow converges in $C^\infty_{\rm loc}(X\setminus \Null_J(\alpha,\beta))$ when $\zeta=c=:\cab$.

Finally, we remark that the ideas here may also be extended to a larger family of geometric PDEs, for example, the LYZ(Leung--Yau--Zaslov) equations.  

\subsection*{Outline of the proof}
The starting point is the weakly closed convex family
\[
 \mathcal C_a=\{T\in\alpha:T\ge0,\ dT=0,
       \ T_{\rm ac}>0,\ \tr_{T_{\rm ac}}\omega\le a\text{ a.e.}\}.
\]
Defining $a_*=\inf\{a>0:\mathcal C_a\ne\varnothing\}$,
we first identify $a_*$ with the birational invariant $\zeta$. We call $a_*$ the analytic threshold. 
One of the key tools is the Lamari-type characterization for the inverse trace inequality(Theorem~\ref{thm:duality})
\[
\cC_a\neq \varnothing\quad \Leftrightarrow\quad \int_X\bigl(\tr\sqrt{\chi^{-1}\omega}\bigr)^2\chi^n
 \le na\int_X\alpha\wedge\chi^{n-1},\quad \forall \text{ Gauduchon metric }\chi.
\]
This is an analogy of the Lamari's characterization of a class containing a K\"ahler current $T\geq a\omega$, see \cite{L99,Tosatti16}. 
Comparison with Monge--Amp\`ere equations on modifications gives
$a_*\le\zeta$. For the opposite inequality, we use the principalization of the Bergman approximation(Theorem~\ref{thm:regularization}) as an admissible modification model.  It converts any
$T\in\mathcal C_a$ into smooth semipositive residual forms
on modifications. A further controlled perturbation makes the
residual classes K\"ahler, with slopes at most $a+o(1)$.
Hence $\zeta\le a_*$, and weak compactness gives
$\mathcal C_\zeta\ne\varnothing$.

The same construction gives residual K\"ahler
classes $L_j$ whose slopes tend to $\zeta$ and whose
stability threshold $\Gamma$ (introduced by Sjöström Dyrefelt \cite{SD20}) on modifications tends uniformly to zero.
Smooth $J$-equations on these modifications produce balanced
metrics that detect every positive trace defect. The resulting
integral estimate forces
$\tr_{T_{\rm ac}}\omega=\zeta$ almost everywhere for every
$T\in\mathcal C_\zeta$. Strict convexity identifies their
absolutely continuous parts, and a regularized-maximum argument
then identifies the currents themselves. Thus
$\mathcal C_\zeta=\{T\}$. Local Dirichlet replacement upgrades
this almost-everywhere identity to the scalar non-pluripolar
equation. 

For higher regularity, we evolve any smooth initial metric by the
$J$-flow. Comparison with a potential of $T$ and convexity of the
inverse trace first implies 
\[
 \frac{\varphi(t)}t\longrightarrow c-\zeta\quad\text{in }L^1(X).
\]
The differential inequality
\[
 \partial_t\bigl(\tr_{\varphi}\omega\bigr)
 \ge-\frac1{2t}\tr_{\varphi}\omega
\]
converts time-averaged convergence into convergence of the trace at
every time. Weak closure then puts every current limit in
$\mathcal C_\zeta=\{T\}$. This step uses no regularity of $T$.

To obtain local smooth convergence, we construct
a strict subsolution with one analytic singularity set. We apply
Chen's mass-concentration method \cite{C21} to the residual
models $L_j$. The logarithmic construction of Demailly--P\u{a}un
\cite{DP04}, performed first on the fixed product $X\times X$, gives
a lower bound for the diagonal mass that is uniform in $j$. This
yields a current in $\alpha-e\beta$ for one $e>0$, independent of the
modification. Adding $e\omega$ and applying partial-trace
regularization gives $B=\kappa+\ddc\psi\in\alpha$, smooth outside a
proper analytic set $Z$, with
\[
 B\ge b_0\omega,\qquad
 \max_{\dim_{\mathbb C}H=n-1}
      \tr_{B|_H}(\omega|_H)\le\zeta-\eta
 \quad\text{on }X\setminus Z
\]
for some $b_0,\eta>0$.

The parabolic weighted trace estimate \cite{W06,SW08} and an ABP
estimate \cite{S18} then give
\[
 \psi-C\le\varphi(t)-b(t)\le C,\qquad
 \tr_\omega(\kappa+\ddc\varphi(t))\le Ce^{-N\psi}.
\]
Together with the uniform lower metric bound from the evolution of
the trace, these estimates give uniform parabolic ellipticity on
compact subsets of $X\setminus Z$. Standard Evans--Krylov estimates
and parabolic Schauder bootstrapping give local
smooth precompactness.
Since the weak limit is $T$, the entire flow converges locally
smoothly to $T$, proving its higher regularity.

Sections~\ref{sec:currents}--\ref{sec:bergman} develop the
trace inequalities and regularization.
Section~\ref{sec:threshold} identifies the two thresholds,
and Sections~\ref{sec:rigidity}--\ref{sec:scalar} prove
rigidity, the scalar equation, and uniqueness.
Section~\ref{sec:concentration} constructs the analytic strict
subsolution. Section~\ref{sec:flow}
establish convergence of $J$-flow and the local estimates.

\textbf{Acknowledgements}. The project was initiated during the author's participation in the 2024 SLMath workshop \emph{Geometry and Analysis of Special Structures on Manifolds}. The author would like to thank Jian Song for helpful discussions on the $J$-equation during the workshop. The author also thanks Hao Fang, Biao Ma, and Jinyang Wu for sharing their preprints \cite{FM26,FMW26}. The results in section~\ref{sec:null} was inspired by their work \cite{FM26}.  The author is grateful to Frederick Tsz-Ho Fong for helpful discussions, for bringing the work \cite{J20} to the author's attention, and for raising the question of uniform convergence of the trace along the $J$-flow. The research is partially supported by the Hong Kong General Research Fund \#16305625 of the Hong Kong Research Grants Council.

\textbf{Declaration on the use of AI}. The author used ChatGPT 5.6 and 6 during the preparation of this manuscript for brainstorming possible approaches and exploring related ideas and techniques. The main ideas of the analytic threshold, Bergman approximation, and Perron's method to uniqueness are due to the author. The author takes the full responsibility for the paper's content.
\section{Currents and trace conditions}\label{sec:currents}

Throughout, $dd^c=i\partial\bar\partial$, $\omega\in\beta$ and $\kappa\in\alpha$ are fixed K\"ahler forms, and $H_u=(u_{j\bar k})$ is the complex Hessian. Divisors and their integration currents are distinguished by brackets. Our normalization is
\begin{equation}\label{eq:poincare-lelong}
 dd^c\log|f|^2=2\pi[E]
\end{equation}
when $f$ locally defines an integral divisor $E$.

For positive Hermitian forms $A,B$, $\tr_A B=\tr(A^{-1}B)$. For $1\le q\le n$, let $P_{q,B}(A)$ be the sum of the largest $q$ eigenvalues of $A^{-1/2}BA^{-1/2}$. In a $B$-unitary frame,
\begin{equation}\label{eq:partial-traces}
 P_{q,B}(A)=\max_{Q^2=Q=Q^*,\ \operatorname{rank}Q=q}\tr(QA^{-1}),
 \qquad P_{n,B}(A)=\tr_A B.
\end{equation}
We write $P_B=P_{n-1,B}$. One also has (see \cite{C21})
\[
 P_B(A)=\max_{\dim_{\mathbb C}H=n-1}\tr_{A|_H}(B|_H).
\]
Each $P_{q,B}$ is convex, decreasing in matrix order, and homogeneous of degree $-1$. Moreover,
\[
 P_{q,B}(A)\le a\quad\Longrightarrow\quad A\ge B/a.
\]
The matrix cone $\mathfrak C_a(B)=\{A>0:\tr_A B\le a\}$ is closed and convex, and is preserved by addition of semipositive forms.

\begin{definition}\label{def:trace-family}
For $a>0$, set
\begin{equation}\label{eq:trace-family}
 \mathcal C_a=\{T\in\alpha:T\ge0,\ dT=0,\ T_{\rm ac}>0,
                          \ \tr_{T_{\rm ac}}\omega\le a\text{ a.e.}\},
 \qquad a_*=\inf\{a>0:\mathcal C_a\ne\varnothing\}.
\end{equation}
The absolutely continuous part is taken coefficientwise with respect to smooth volume. The singular coefficient matrix is positive.
\end{definition}

Fix a nonnegative radial $\varrho\in C_c^\infty(\mathbb C^n)$, supported in the unit ball and of integral one, and put $\varrho_r(z)=r^{-2n}\varrho(z/r)$. If $U=A\,d\lambda+U_s\ge0$, with $U_s\ge0$, then
\begin{equation}\label{eq:convolution}
 P_{q,B}(U*\varrho_r)
 \le P_{q,B}(A*\varrho_r)
 \le P_{q,B}(A)*\varrho_r.
\end{equation}
Consequently $P_{q,\omega}(U_{\rm ac})\le a$ almost everywhere is equivalent to $P_{q,B}(U*\varrho_r)\le a$ for every constant form $0<B\le\omega$ on the convolution domain. The converse follows from differentiation of coefficient measures and a countable collection of shrinking balls and frozen backgrounds. These conditions are weakly closed: convolution evaluations are continuous, and their matrices have lower bound $B/a$. When $U=dd^cu$, the convolution is applied to the entire local potential.
\subsection{Non-pluripolar products} 
We first recall the Bedford--Taylor monotone convergence theorem \cite[Theorem~2.1 and Corollary~2.2]{BT82}; see also \cite[Chapter~III, Theorem~3.7(a)]{D12}.
\begin{lemma}\label{lem:bt-convergence}
    Let $\Omega$ be a domain in $\mathbb C^n$, $1\leq p\leq n$, and, for each $0\leq i\leq p$, let $v_{i,j}$ be a decreasing sequence of locally bounded plurisubharmonic functions whose pointwise limit $v_i$ is locally bounded and plurisubharmonic. Then

\[
v_{0,j}dd^cv_{1,j}\wedge\cdots\wedge dd^cv_{p,j}
\longrightarrow v_0dd^cv_1\wedge\cdots\wedge dd^cv_p,
\]
weakly as currents with measure coefficients.
\end{lemma}

For a plurisubharmonic potential $u$, put $u_k=\max\{u,-k\}$, $E_k=\{u>-k\}$, and $T=dd^cu$. The non-pluripolar products satisfy

\begin{equation}\label{eq:np-locality}
\mathbf1_{E_k}\langle T^p\rangle
=\mathbf1_{E_k}(dd^cu_k)^p.
\end{equation}
They are the increasing limits of the right-hand sides, put no mass on pluripolar sets and are local in the plurifine topology. In particular $E_k$, and sets $\{u>b\}$ with $b$ smooth, are plurifinely open. In degree one,

\begin{equation}\label{eq:np-degree-one}
\langle dd^cu\rangle=\mathbf1_{\{u>-\infty\}}dd^cu.
\end{equation}
Non-pluripolar products depend only on the currents, not on chosen local potentials, and are symmetric and multilinear in positive closed $(1,1)$-currents. Thus a smooth form can be wedged outside the brackets. On a compact K\"ahler manifold the product of any finite collection of positive closed $(1,1)$-currents, in degrees at most $n$, is locally finite and defines a positive closed current. See \cite[Section 1.2, Definition 1.1, Propositions 1.4 and 1.6, and Theorem 1.8]{BEGZ10}.

Pluripolar sets have zero Lebesgue measure, so \eqref{eq:np-degree-one} gives

\begin{equation}\label{eq:ac-degree-one}
(\langle T\rangle)_{\mathrm{ac}}=T_{\mathrm{ac}}.
\end{equation}

We will also use the mixed Monge--Amp\`ere inequality \cite[Proposition 1.11]{BEGZ10} \begin{lemma}\label{lem:mixed-ma}
Let $T_1,\ldots,T_n$ be positive closed $(1,1)$-currents on a complex $n$-manifold, let $\mu$ be a positive measure, and let $f_1,\ldots,f_n\geq0$ be measurable functions such that

\[
\langle T_j^n\rangle\geq f_j\mu,\qquad 1\leq j\leq n.
\]
Then
\begin{equation}\label{eq:mixed-ma}
\langle T_1\wedge\cdots\wedge T_n\rangle
\geq(f_1\cdots f_n)^{1/n}\mu.
\end{equation}
\end{lemma}

\subsection{Trace inequalities and scalar solutions}

\begin{lemma}\label{lem:ac-lower}
Let $T=dd^cu\geq0$ locally on a coordinate ball, and suppose its non-pluripolar powers through degree $p$ are locally finite. Write $A=T_{\mathrm{ac}}$ for its measurable positive semidefinite coefficient matrix. If $\Theta$ is a smooth strongly positive $(n-p,n-p)$-form, then

\begin{equation}\label{eq:ac-lower}
\langle T^{p}\rangle\wedge\Theta
\geq A^p\wedge\Theta\qquad \text{as measures,}
\end{equation}
where the right-hand side means the pointwise algebraic wedge of the measurable matrix coefficients, integrated against coordinate volume.
\end{lemma}
\begin{proof}
    First suppose that $u$ is locally bounded. Consider the local convolution regularization $u_\delta$. Then $u_\delta$ decrease to $u$ as $\delta\downarrow0$, and $\ddc u_\delta\to A(x)$ for a.e. $x$. Therefore, for every nonnegative compactly supported smooth $\chi$, applying Fatou's lemma and Lemma~\ref{lem:bt-convergence}, one gets \[
    \int\chi A^p\wedge\Theta\leq \liminf_{\delta\downarrow 0}\int\chi(\ddc u_\delta)^p\wedge\Theta=\int\chi(\ddc u)^p\wedge\Theta.
    \]

    Now for arbitrary $u$, use $u_k=\max\{u,-k\}$. By \eqref{eq:np-locality} for degree $1,p$, and \eqref{eq:ac-degree-one}, the current $\ddc u_k$ has absolutely continuous matrix $A$ a.e. on $E_k=\{u>-k\}$, and its degree-$p$ product agrees with $\langle T^p\rangle$ on $E_k$. Hence by previous argument for bounded potentials, we get \[
    \mathbf1_{E_k}\langle T^p\rangle\wedge\Theta\geq \mathbf1_{E_k}A^p\wedge\Theta.
    \]
    The sets $E_k$ increase to the complement of the pluripolar pole set. Let $k\to \infty$, we get the desired~\eqref{eq:ac-lower}.
    
\end{proof}

Next, we show a reverse inequality in the top degree, using the mixed Monge--Amp\`ere inequality.
\begin{lemma}\label{lem:ac-upper}
On a coordinate ball fix $\theta=dd^c|z|^2$. If $T\geq\varepsilon\theta$ is a positive closed current with locally finite non-pluripolar top product, write $A=T_{\mathrm{ac}}$ and $f=\langle T^n\rangle_{\mathrm{ac}}/\theta^n$.
Then $f\leq\det A$ almost everywhere. 
\end{lemma}

\begin{proof}
    Let $B>0$ be a constant hermitian form. Apply \eqref{eq:mixed-ma} to $T,B,...,B$, and using $\langle T^n\rangle\geq\langle T^n\rangle_{\rm ac}=f\theta^n$, one gets \[
    \langle T\rangle\wedge B^{n-1}\geq f^{1/n}(\det B)^{(n-1)/n}\theta^n.
    \]
    Taking absolutely continuous densities and using \eqref{eq:ac-degree-one} gives \begin{equation}
        \label{eq:density-test}
\frac{\det B}{n}\tr(B^{-1}A)
\geq f^{1/n}(\det B)^{(n-1)/n}.
    \end{equation}
    Choose a countable dense family of positive definite Hermitian matrices, for example matrices with rational real and imaginary entries. Removing the union of their exceptional null sets makes \eqref{eq:density-test} simultaneous for that family. Continuity in $B$ extends it to every positive definite $B$ at each remaining point. Consequently,

\[
f(x)^{1/n}
\leq\inf_{B>0}\frac{(\det B)^{1/n}}n\tr(B^{-1}A(x))
=(\det A(x))^{1/n}.
\]
The last identity follows from arithmetic--geometric mean applied to $B^{-1/2}AB^{-1/2}$, with equality at $B=A$. The K\"ahler lower bound ensures $A>0$ a.e.
\end{proof}

We next pass from an almost-everywhere trace bound to inequalities
of non-pluripolar products. The proof combines the convolution
formulation of weak cone inequalities used by Chen
\cite[Definition~3.3]{C21} and Bedford--Taylar's monotone convergence. 
\begin{lemma}\label{lem:np-cone}
    Let $T=dd^cu\geq\varepsilon\theta$ on a coordinate ball, assume the non-pluripolar powers are locally finite, and let $A=T_{\mathrm{ac}}$. If $B>0$ is a constant Hermitian form, and

\[
\tr(A^{-1}B)\leq a\quad\text{a.e.},
\]
then
\begin{equation}\label{eq:np-cone}
a\langle T^p\rangle-pB\wedge\langle T^{p-1}\rangle\geq0,
\qquad 1\leq p\leq n.
\end{equation}
Here $\langle T^0\rangle=1$.
\end{lemma}

\begin{proof}
    Write the decomposition of $T$ as $T=A+T_s$, where the singular matrix-valued measure $T_s$ is positive. Therefore, for local convolution with radial kernel $\varrho_\delta$,\[
    \ddc u_\delta=A*\varrho_\delta+T_s*\varrho_\delta\geq A*\varrho_\delta.
    \]
    By monotonicity,  convexity of $H\mapsto\tr(H^{-1}B)$, and Jensen's inequality \begin{equation}\label{eq:convolved-trace}
\tr\bigl((dd^cu_\delta)^{-1}B\bigr)
\leq\tr\bigl((A*\varrho_\delta)^{-1}B\bigr)
\leq\int\varrho_\delta(y)\tr(A(x-y)^{-1}B)\,dy
\leq a. 
    \end{equation}
It follows that \[
a(\ddc u_\delta)^p-p(\ddc u_\delta)^{p-1}\wedge B\geq 0, \quad \text{ for all }1\le p\le n.
\]
The inequality is strict when $1\le p<n$. Hence, if $u$ is locally bounded, Lemma~\ref{lem:bt-convergence} gives the desired inequality~\eqref{eq:np-cone} by letting $\delta\downarrow 0$.

For arbitrary $u$, choose a quadratic potential $v$ such that $\ddc v=B$, and choose $\lambda\ge n/a$. Set \[
b_C=\lambda v-C, \qquad w_C=\max\{u,b_C\}.
\]
The form $\ddc b_C=\lambda B$ satisfies $\tr_{dd^cb_C}B=n/\lambda\le a$. Consider the following regularized maximum \[
w_{C,j}:=M_{\tau_j}(u_{\delta_j},b_C), \quad M_\tau(s,t)=\frac12(s+t+\sqrt{(s-t)^2+\tau^2}).
\]
Then as $\delta_j\downarrow 0, \tau_j\downarrow 0$, the function $w_{C,j}\downarrow w_{C}$. By convexity, we have $\tr_{dd^cw_{C,j}}B\le a$. Since $w_C$ is bounded, by previous discussion for bounded case, one gets, for $W=\ddc w_C$, \[
a\langle W^p\rangle-p\langle W^{p-1}\rangle\wedge B\geq 0, \quad \text{for all }1\leq p\leq n.
\]

On the plurifinely open set $\{u>b_C\}$, the potentials coincide $w_C=u$. By \eqref{eq:np-locality}, all the products appearing in \eqref{eq:np-cone} coincide there. Those sets increase to $\{u>-\infty\}$ as $C\to \infty$. Their complement is pluripolar and none of the terms in \eqref{eq:np-cone} charges it. Taking increasing restrictions proves \eqref{eq:np-cone} for $T$.  
    
\end{proof}

The preceding density comparisons determine all mixed densities
by applying them to $T+t\omega$ and comparing polynomial
coefficients. This consequence of
Lemmas~\ref{lem:ac-lower} and~\ref{lem:ac-upper}
is the link between the nonpluripolar-product equation and the pointwise trace.
\begin{lemma}\label{lem:mixed-density}
For any K\"ahler current $T$, any smooth K\"ahler form $\omega$, and $0\leq p\leq n$,
\begin{equation}\label{eq:mixed-density}
\big(\langle T^p\rangle\wedge\omega^{n-p}\big)_{\rm ac}
=(T_{\rm ac})^p\wedge\omega^{n-p}.
\end{equation}
Consequently every scalar solution at $a$ satisfies
$\tr((T_{\rm ac})^{-1}\omega)=a$ almost everywhere.
\end{lemma}
\begin{proof}
Lemma~\ref{lem:ac-lower} and Lemma~\ref{lem:ac-upper} give
$(\langle S^n\rangle)_{\rm ac}=(S_{\rm ac})^n$
for every K\"ahler current $S$. Apply this to $S=T+t\omega$ for countably many positive rational $t$. Non-pluripolar multilinearity and linearity of Lebesgue decomposition give, outside a common null set, the polynomial identity
\[
\sum_{p=0}^n\binom np t^{n-p}
\big(\langle T^p\rangle\wedge\omega^{n-p}\big)_{\rm ac}
=
\sum_{p=0}^n\binom np t^{n-p}
(T_{\rm ac})^p\wedge\omega^{n-p}.
\]
Its coefficients agree, proving \eqref{eq:mixed-density}. Taking absolutely continuous parts of the scalar equation and dividing by the strictly positive determinant gives the exact trace equality.
\end{proof}
Now we can show that the solution to the non-pluripolar product equation satisfies both the nonpluripolar cone condition automatically, and the cone condition introduced by Chen \cite{C21} in the sense of local convolution. 
\begin{theorem}[Automatic admissibility]\label{thm:automatic-cone}
    Let $X$ be a compact K\"ahler manifold of complex dimension $n\geq2$, let $\omega$ be a smooth K\"ahler form, and let $T$ be a K\"ahler current. If, for a constant $a>0$,

\begin{equation}\label{eq:scalar-equation}
n\langle T^{n-1}\rangle\wedge\omega
=a\langle T^n\rangle,
\end{equation}
then for every $1\leq p<n$,
\begin{equation}\label{eq:scalar-cone}
a\langle T^p\rangle-p\omega\wedge\langle T^{p-1}\rangle\geq0.
\end{equation}
Moreover, it satisfies the following inverse-trace inequality in the local convolution sense: for every coordinate neighborhood, every constant positive form $B\leq\omega$ there, every local potential $T=dd^cu$, and every radial convolution defined on a smaller neighborhood,

\begin{equation}\label{eq:scalar-density}
\tr_{dd^cu_\delta}B
=\tr\bigl((dd^cu_\delta)^{-1}B\bigr)\leq a.
\end{equation}
\end{theorem}
\begin{proof}
By Lemma~\ref{lem:mixed-density}, 
\begin{equation}\label{eq:pointwise-trace}
\tr(T_{\rm ac}^{-1}\omega)= a\quad\text{a.e.}
\end{equation}
For any constant $0<B\leq\omega$, we have $\tr(T_{\rm ac}^{-1}B)\leq a$. The positive singular-part decomposition and convexity argument \eqref{eq:convolved-trace} therefore prove \eqref{eq:scalar-density} for every local convolution.

Fix $0<\epsilon<1$. Around every point, continuity and strict positivity of $\omega$ allow a smaller coordinate ball and a constant positive form $B$ such that

\[
(1-\epsilon)\omega\leq B\leq\omega.
\]

Equation \eqref{eq:pointwise-trace} gives $\tr(T_{\rm ac}^{-1}B)\leq a$ there. Apply Lemma~\ref{lem:np-cone} and add the positive current $p(B-(1-\epsilon)\omega)\wedge\langle T^{p-1}\rangle$. We obtain

\[
a\langle T^p\rangle-p(1-\epsilon)\omega\wedge\langle T^{p-1}\rangle\geq0.
\]
 Positivity is local, so the same inequality holds on all of $X$. Letting $\epsilon\downarrow0$ proves \eqref{eq:scalar-cone}. 
    
\end{proof}

\subsection{Upper tests and comparison}
For local replacement we also need comparison with smooth
solutions. The passage to upper test functions is the usual
stability argument for viscosity subsolutions.
 The following convolution proof for the present trace constraint,
including potentials with poles, is essentially due to \cite{M26b}.
\begin{lemma}\label{lem:upper-tests}
Let $U$ be psh on a neighborhood of $\overline B_R$, and let $\omega$ be a smooth K\"ahler form there. Suppose
\begin{equation}\label{eq:weak-trace}
\tr_{(dd^cU)_{\rm ac}}\omega\leq a
\quad\hbox{almost everywhere}.
\end{equation}
If $q\in C^2(\overline{B_R})$ touches $U$ from above at $z_0\in B_R$, then
\begin{equation}\label{eq:upper-test}
dd^cq(z_0)>0,\qquad
\tr_{dd^cq(z_0)}\omega\leq a. 
\end{equation}
Moreover, if $v\in C^\infty(B_R)\cap C^0(\overline{B_R})$ solves $\tr_{\ddc v}\omega=a$, and $U\leq v$ on $\partial B_R$, then $U\leq v$ throughout $B_R$. \end{lemma}

\begin{proof}
Fix a constant positive form $B\le\omega$ on a small ball
about $z_0$. By the convolution argument in \eqref{eq:convolution},
$U_\delta=U*\varrho_\delta\downarrow U$ satisfies
\begin{equation}\label{eq:trace-comparison}
\tr_{dd^cU_\delta}B\le a,\qquad
dd^cU_\delta\ge B/a.
\end{equation}
For $\epsilon>0$, put $q_\epsilon=q+\epsilon|z-z_0|^2$,
and choose $r>0$ with $U\le q$ on
$\overline{B_r(z_0)}$. Let $z_\delta$ maximize
$U_\delta-q_\epsilon$ on this closed ball. Its maximum is
nonnegative, since $U_\delta(z_0)\ge U(z_0)$.
If $z_{\delta_k}\to z$, monotonicity gives
$U_{\delta_k}\le U_t$ for every fixed $t>0$ and all large
$k$. Letting $k\to\infty$, then $t\downarrow0$, yields
\[
0\le\limsup_{k\to\infty}
 (U_{\delta_k}-q_\epsilon)(z_{\delta_k})
\le U(z)-q_\epsilon(z)
\le-\epsilon|z-z_0|^2.
\]
Thus $z_\delta\to z_0$; in particular $z_\delta\in B_r(z_0)$
for small $\delta$. At these maxima,
\[
dd^cq_\epsilon(z_\delta)\ge dd^cU_\delta(z_\delta)\ge B/a,
\qquad
\tr_{dd^cq_\epsilon(z_\delta)}B\le a.
\]
Passing first $\delta\downarrow0$, then $\epsilon\downarrow0$,
gives the same bounds for $dd^cq(z_0)$.
Continuity of $\omega$ allows
$B=(1-\eta)\omega(z_0)$ on a sufficiently small ball.
Letting $\eta\downarrow0$ proves
\eqref{eq:upper-test}.

For comparison, the equation for $v$ gives
\begin{equation}\label{eq:trace-lower}
dd^cv\ge\omega/a.
\end{equation}
 If $U>v$ somewhere, choose $\epsilon>0$
so small that
\[
\epsilon\,dd^c|z|^2<\omega/a,\qquad
c:=\max_{\overline B_R}
 \{U+\epsilon(|z|^2-R^2)-v\}>0.
\]
The maximum occurs in $B_R$, since $U\le v$ on its
boundary. At a maximizing point,
$q=v-\epsilon(|z|^2-R^2)+c$ touches $U$ from above, and
\[
0<dd^cq=dd^cv-\epsilon\,dd^c|z|^2<dd^cv.
\]
Strict inverse monotonicity gives
\[
\tr_{dd^cq}\omega
>\tr_{dd^cv}\omega=a,
\]
contradicting \eqref{eq:upper-test}.
Taking $U$ to be a second smooth solution gives comparison
for ordered smooth boundary data.\qedhere
\end{proof}

\section{Bergman regularization}\label{sec:bergman}
In this section, we adapt Demailly's regularization \cite{D92} by Bergman kernels to approximate currents while preserving the 
full and partial inverse-trace bounds. The local estimates are obtained by H\"ormander's $L^2$-estimate for the $\bar\partial$ equations. Localization and principalization then allow
global gluing. The resulting approximation is used both to compare
the analytic and birational thresholds and to construct an analytic strict subsolution later.

We first recall some basic facts about the principalization of a coherent ideal sheaf and the local Bergman kernels. For a nonzero coherent
ideal $\mathcal I$ on a smooth compact complex manifold,
\cite[Theorem 1.10]{BM97} gives a finite composition of blowups
with smooth centers $\pi:Y\to X$ such that
\[
\mathcal I\cdot\mathcal O_Y=\mathcal O_Y(-E),
\]
where $E$ is effective and integral, and has simple normal crossings. This is a principalization of $\mathcal I$.
Its local meaning is the following: if $f_1,\ldots,f_N$
generate $\mathcal I$, then on a principalization chart
\[
\begin{gathered}
f_\ell\circ\pi=s\widehat f_\ell,\qquad
s=\gamma z_1^{q_1}\cdots z_k^{q_k},\\
\widehat f_\ell\text{ holomorphic},\qquad
\sum_{\ell=1}^N|\widehat f_\ell|^2>0,
\end{gathered}
\]
where $\gamma$ is holomorphic and nowhere zero and
$q_1,\ldots,q_k$ are nonnegative integers.
We say that a current $T$ has a smooth remainder after principalization if there is a coherent ideal $\cI$, and a principalization $\pi:Y\to X$, such that \begin{equation}
    \begin{aligned}
        \cI\cdot\cO_Y=\cO_Y(-E), \quad \pi^*T=\sigma+[D], \quad D=c_0E, (c_0>0),
    \end{aligned}
\end{equation}
where $\sigma $ is a smooth real closed form on $Y$. We say that a current $T$ has analytic singularity type if locally it has potential \[
c\log\left(\sum_{\ell=1}^N|f_{\ell}|^2\right)+O(1),\qquad c>0,\qquad f_\ell \text{ holomorphic}.
\]

We will use H\"ormander's estimate \cite{H65}, see also \cite[Chapter VIII, Theorem~6.1]{D12}, to bound the trace of the approximation. For convenience, we state a simple form here. 
\begin{proposition}\label{prop:hormander}
    Let $\Omega\subset\mathbb C^n$ be pseudoconvex, let $v\in C^\infty(\Omega)$ be strictly plurisubharmonic, and let $g=\sum g_jd\bar z_j$ be a distributionally $\bar\partial$-closed $(0,1)$-form with
\[
\int_\Omega \sum_{i,j}(H_v^{-1})_{ij}g_i\overline{g_j}\,e^{-v}d\lambda<\infty.
\]
There is $s$ with $\bar\partial s=g$ and
\begin{equation}\label{eq:hormander}
\int_\Omega |s|^2e^{-v}d\lambda
 \le\int_\Omega \sum_{i,j}(H_v^{-1})_{ij}g_i\overline{g_j}\,e^{-v}d\lambda.
\end{equation}
\end{proposition}

Let $\Omega\Subset\mathbb C^n$ be pseudoconvex and let
$\psi\in\operatorname{PSH}(\Omega)$, $\psi\not\equiv-\infty$.
Define
\[
\begin{aligned}
\mathscr H(\Omega,\psi)
 &=\left\{f\in\mathcal O(\Omega):
        \int_\Omega|f|^2e^{-\psi}d\lambda<\infty\right\},\\
\mathcal I(\psi)_x
 &=\{g\in\mathcal O_{\Omega,x}:
             |g|^2e^{-\psi}\text{ is locally integrable near }x\}.
\end{aligned}
\]
Nadel's coherence and local generation theorem \cite{Nadel}
asserts that $\mathcal I(\psi)$ is coherent and is generated by
any orthonormal basis $(\sigma_\ell)$ of
$\mathscr H(\Omega,\psi)$; see also \cite[Proposition~5.7 and its proof]{D00}. In particular, for every
$U\Subset\Omega$ there is a finite $N$ such that
\begin{equation}\label{eq:ideal-generators}
\mathcal I(\psi)|_U=(\sigma_1,\ldots,\sigma_N)|_U.
\end{equation}
Explicitly, if $x\in U$ and $g$ is holomorphic near $x$,
then
\[
|g|^2e^{-\psi}\in L^1_{\rm loc}\text{ near }x
\quad\Longleftrightarrow\quad
g=\sum_{\ell=1}^N a_\ell\sigma_\ell
\text{ near }x,
\]
with holomorphic coefficients $a_\ell$. 

\subsection{Local Bergman approximation}

Let $\Omega\Subset\mathbb C^n$ be pseudoconvex and
$u\in\operatorname{PSH}(\Omega)$, $u\not\equiv-\infty$.
For $m>0$, consider the weighted Hilbert space
\[
\begin{aligned}
\mathscr H_m
&=\left\{f\in\mathcal O(\Omega):
 \|f\|_m^2=\int_\Omega|f|^2e^{-mu}d\lambda<\infty\right\},
\\
\langle f,g\rangle_m&=\int_\Omega f\bar g\,e^{-mu}d\lambda.
\end{aligned}
\]
The inner product is linear in its first argument.
The submean inequality and local upper bounds for $u$ make
evaluation $f\mapsto f(w)$ a continuous linear functional on
$\mathscr H_m$. Its Riesz representative is denoted
$K_m(\,\cdot\,,w)$. For every $f\in\mathscr H_m$, one has the 
\emph{reproducing property}
\[
f(w)=\langle f,K_m(\,\cdot\,,w)\rangle_m
=\int_\Omega f(z)\overline{K_m(z,w)}e^{-mu(z)}d\lambda(z).
\]
In particular, the kernel recovers the value of every function
in the weighted space by integration against that function.
For any orthonormal basis $(\sigma_\ell)$,
\[
K_m(z,w)=\sum_{\ell\ge1}\sigma_\ell(z)\overline{\sigma_\ell(w)}.
\]
This expression is independent of the basis. It is holomorphic
in $z$ and antiholomorphic in $w$; the series and its
derivatives converge locally uniformly.

The reproducing property and Cauchy--Schwarz give
\[
K_m(w,w)=\|K_m(\,\cdot\,,w)\|_m^2
=\sup_{\|f\|_m\le1}|f(w)|^2.
\]
When $K_m(w,w)>0$, the supremum is attained by the
\emph{normalized kernel}
\[
k_w(z)=\frac{K_m(z,w)}{\sqrt{K_m(w,w)}},
\qquad
\|k_w\|_m=1,\qquad
k_w(w)=\sqrt{K_m(w,w)}.
\]
If $K_m(w,w)=0$, every function in $\mathscr H_m$ vanishes
at $w$. When the domain must be specified, we write
$K_{V,m}$ and $\|\cdot\|_{V,m}$. The same definitions and
identities hold with $d\lambda$ replaced by
$\rho\,d\lambda$.

Define
\[
b_m(z)=\frac1m\log K_m(z,z)
=\frac1m\log\sum_{\ell\ge1}|\sigma_\ell(z)|^2.
\]
\begin{proposition}[{Demailly \cite[Proposition~3.1]{D92}}]
   \[
u(z)-\frac{C_1}{m}
\le b_m(z)
\le\sup_{|\zeta-z|<r}u(\zeta)
 +\frac1m\log\frac{C_2}{r^{2n}},
\qquad
0<r<d(z,\partial\Omega).
\]
\end{proposition}
In particular,
\[
b_m\longrightarrow u
\quad\text{pointwise and in }L^1_{\rm loc},\qquad
\nu(u,z)-\frac{2n}{m}\le\nu(b_m,z)\le\nu(u,z).
\]
Here $\nu(u,z)$ is the Lelong number of $u$ at $z$.

If $0<c\le\rho\le C$ on $\Omega$, replacing $d\lambda$
by $\rho\,d\lambda$ places the new diagonal kernel between
$C^{-1}K_m(z,z)$ and $c^{-1}K_m(z,z)$, by the extremal
formula above. The corresponding potentials therefore differ
by $O(m^{-1})$.
The lower approximation bound also gives
\[
K_m\not\equiv0,\qquad \{K_m(z,z)=0\}\subset\{u=-\infty\}.
\]

We now combine Proposition~\ref{prop:hormander}
with the extremal formula for the Bergman metric, to give a proof that the local Bergman approximations preserve the inverse trace bounds. 
\begin{lemma}\label{lem:bergman-trace}
Let $\Omega\Subset\mathbb C^n$ be bounded pseudoconvex, and let $u$ be plurisubharmonic near $\overline\Omega$. Fix a constant positive form $B$, $a>0$, and $1\le q\le n$. Suppose
\begin{equation}\label{eq:bergman-full}
 P_{q,B}(H_{u*\varrho_r})\le a
 \quad\text{on }\Omega\text{ for all sufficiently small }r>0.
\end{equation}
For the kernel of $e^{-mu}d\lambda$, the potential $b_m=m^{-1}\log K_m(z,z)$ is smooth on $\{K_m(z,z)>0\}$, and
\begin{equation}\label{eq:bergman-partial}
 H_{b_m}>0,\qquad P_{q,B}(H_{b_m})\le a.
\end{equation}
For a smooth positive density $\rho\,d\lambda$ defined near $\overline\Omega$, the bound is $a/(1-C/m)$ for all sufficiently large $m$.
\end{lemma}
\begin{proof}
A constant linear change of coordinates reduces $B$ to $I$. Its constant Jacobian only multiplies the kernel by a constant. First assume $u$ is smooth; letting $\delta\downarrow0$ in the
hypothesis gives $P_{q,I}(H_u)\leq a$.  For $K_m(w,w)>0$, set
\[
 s_i(z)=(\bar z_i-\bar w_i)k_w(z).
\]
Reproduction gives $\langle s_i,f\rangle_m
=\langle k_w,(z_i-w_i)f\rangle_m=0, \forall f\in\mathscr H_m$, i.e. $s_i\perp\mathscr H_m$, so $s_i$ is the minimal-norm solution of $\bar\partial s_i=k_w\,d\bar z_i$. H\"ormander's estimate gives
\begin{equation}
 \|s_i\|_m^2\le\frac1m\int_\Omega(H_u^{-1})_{i\bar i}|k_w|^2e^{-mu}d\lambda.
\end{equation}
To relate this estimate to $b_m$, choose an orthonormal basis
with $\sigma_1=k_w$. The reproducing property implies
$\sigma_\ell(w)=0$ for every $\ell\ge2$.
Differentiating $b_m=m^{-1}\log\sum_\ell|\sigma_\ell|^2$ at
$w$, and direct computation gives 
\begin{equation}
 m(b_m)_{i\bar j}(w)=
 \frac{\sum_{\ell\ge2}\partial_i\sigma_\ell(w)\overline{\partial_j\sigma_\ell(w)}}{K_m(w,w)}.
\end{equation}
Every $f\in\mathscr H_m$ with $f(w)=0$ expands in the
remaining basis functions $\sigma_\ell$ with $\ell\ge 2$. Cauchy--Schwarz inequality therefore gives
\[
|\partial_i f(w)|^2
\le m(b_m)_{i\bar i}(w)K_m(w,w)\,\|f\|_m^2.
\]
In unitary coordinates diagonalizing $H_{b_m}(w)$, apply this to $f=(z_i-w_i)k_w$. Since $\Omega$ is bounded, $f\in\mathscr H_m$, and
\[
 1\le m(b_m)_{i\bar i}(w)\|s_i\|_m^2.
\]
Thus $H_{b_m}(w)>0$. Let $J$ index its $q$ largest reciprocal eigenvalues. Then
\[
 P_{q,I}(H_{b_m}(w))
 \le m\sum_{i\in J}\|s_i\|_m^2
 \le\int_\Omega\sum_{i\in J}(H_u^{-1})_{i\bar i}|k_w|^2e^{-mu}d\lambda\le a.
\]
For general $u$, take $u_\delta=u*\varrho_\delta\downarrow u$
on the fixed domain $\Omega$, and denote the corresponding
kernel by $K_{m,\delta}$.
Since $e^{-mu_\delta}\uparrow e^{-mu}$, the extremal formula
shows that $K_{m,\delta}(w,w)$ decreases as
$\delta\downarrow0$, and
\[
K_{m,\delta}(w,w)\ge K_m(w,w).
\]
For fixed $w$, the normalized kernels
\[
\frac{K_{m,\delta}(z,w)}{\sqrt{K_{m,\delta}(w,w)}}
\]
form a normal family, by local upper bounds for $u_\delta$
and the submean inequality. Any subsequential limit $f$
satisfies, by Fatou's lemma,
\[
\|f\|_m\le1,\qquad
|f(w)|^2=\lim_{\delta\downarrow0}K_{m,\delta}(w,w).
\]
The extremal formula for $K_m$ gives the opposite inequality
for this limit. Hence $K_{m,\delta}(w,w)\downarrow K_m(w,w)$.
The bounds
\[
|K_{m,\delta}(z,w)|^2
\le K_{m,\delta}(z,z)K_{m,\delta}(w,w)
\]
give local normality of the two-variable kernels.
Their limits are determined by the diagonal, so polarization and
Cauchy estimates yield $K_{m,\delta}\to K_m$ in
$C^\infty_{\rm loc}(\Omega\times\Omega)$.
On $\{K_m(z,z)>0\}$, the logarithmic Hessians converge.
Closedness of $\{H>0:P_{q,I}(H)\le a\}$ proves the desired inequality.

Finally, the density is absorbed into the weight:
$e^{-mu_\delta}\rho=e^{-m(u_\delta-m^{-1}\log\rho)}$.
Since $H_{u_\delta}\ge I/a$ and $H_{\log\rho}$ is bounded on
$\overline\Omega$, there is $C$, independent of $m,\delta$, with
\[
H_{u_\delta}-\frac1mH_{\log\rho}
\ge(1-C/m)H_{u_\delta}>0
\]
for $m>C$.
Inversion bounds its partial inverse trace by $a/(1-C/m)$.
Apply the smooth result to $u_\delta-m^{-1}\log\rho$, then let
$\delta\downarrow0$ by the preceding kernel-convergence argument.
\end{proof}

To compare local approximations on overlaps, we use the standard
cutoff and $\bar\partial$-correction method for Bergman kernels.
Here the curvature grows with $m$, and the proof records
the resulting relative error uniformly near the base locus.
\begin{lemma}\label{lem:localization}
Let $W\Subset V\Subset\mathbb C^n$ be bounded pseudoconvex domains and $W'\Subset W$. Suppose $u$ is plurisubharmonic on a neighborhood of $\overline V$, with $dd^cu\ge\delta\,dd^c|z|^2$, $\delta>0$. Use the same density $\rho\,d\lambda$ for both kernels, where
$\rho>0$ is smooth on a neighborhood of $\overline V$. Then, for sufficiently large $m$,
\begin{equation}\label{eq:kernel-comparison}
K_{V,m}(w,w)\le K_{W,m}(w,w)
 \le(1+C/\sqrt m)^2K_{V,m}(w,w),\qquad w\in W'.
\end{equation}
The constant $C$ is uniform,  depending only on $W'\Subset W$.
\end{lemma}
\begin{proof}
Restriction decreases the Hilbert norm, hence
$K_{V,m}\le K_{W,m}$.
For the reverse bound, we extend the normalized kernel from
$W$ to a holomorphic function on $V$, preserving its value
at $w$ and increasing its norm by at most $C/\sqrt m$.
Fix $w\in W'$ with $K_{W,m}(w,w)>0$, and let $k_w$ be its
normalized kernel, so $\|k_w\|_{W,m}=1$.
Choose $0\le\eta\le1$, $\eta\in C_c^\infty(W)$, with
$\eta=1$ near $w$, such that
\[
d(w,\operatorname{supp}\bar\partial\eta)\ge r>0,\qquad
\|\bar\partial\eta\|_\infty\le C.
\]
These constants can be chosen uniformly for $w\in W'$.
For the weight
\[
v=mu-\log\rho+n\log|z-w|^2
\]
we have $H_v\ge(m\delta-C)I\ge m\delta I/2$ away from $w$.
The form $k_w\bar\partial\eta$, extended by zero, is
$\bar\partial$-closed on $V$, and
\[
\int_V|k_w\bar\partial\eta|^2e^{-v}d\lambda
\le C\int_W|k_w|^2e^{-mu}\rho\,d\lambda=C.
\]
Thus H\"ormander's estimate\eqref{eq:hormander} gives
$\bar\partial s=k_w\bar\partial\eta$ with
\begin{equation}\label{eq:kernel-correction}
\int_V|s|^2e^{-mu}|z-w|^{-2n}\rho\,d\lambda\le C/m.
\end{equation}

To justify the singular weight, use
\[
v_\nu=m(u*\varrho_\nu)-\log\rho
      +n\log(|z-w|^2+\nu^2)\downarrow v.
\]
The curvature and right-hand-side estimates above are uniform in
$\nu$. Let $s_\nu$ be the resulting solutions. After passage to
a subsequence, $s_\nu\rightharpoonup s$ in $L^2_{\rm loc}(V)$
and $\bar\partial s=k_w\bar\partial\eta$.
Weak lower semicontinuity on an exhaustion of $V$ gives, for
each fixed $\nu_0>0$,
\[
\int_V|s|^2e^{-v_{\nu_0}}d\lambda
\le\liminf_{\nu\downarrow0}
 \int_V|s_\nu|^2e^{-v_{\nu_0}}d\lambda
\le C/m.
\]
Indeed, $e^{-v_{\nu_0}}\le e^{-v_\nu}$ for $\nu<\nu_0$.
Monotone convergence as $\nu_0\downarrow0$ gives
\eqref{eq:kernel-correction}.

The function $s$ is holomorphic near $w$.
Local upper bounds for $u$, together with $\rho>0$, imply
\[
\int_{|z-w|<r}|s|^2|z-w|^{-2n}d\lambda<\infty,
\]
so $s(w)=0$. Otherwise its constant term would produce
$\int_0^r t^{-1}dt=\infty$.
Since $V$ is bounded, \eqref{eq:kernel-correction}
also gives $\|s\|_{V,m}\le C/\sqrt m$.
Therefore
\[
\begin{gathered}
F=\eta k_w-s\in\mathscr H_m(V),\\
F(w)=\sqrt{K_{W,m}(w,w)},\qquad
\|F\|_{V,m}\le1+C/\sqrt m,
\end{gathered}
\]
and the extremal characterization implies
\[
K_{V,m}(w,w)
\ge\frac{|F(w)|^2}{\|F\|_{V,m}^2}
\ge\frac{K_{W,m}(w,w)}{(1+C/\sqrt m)^2}.
\]
If $K_{W,m}(w,w)=0$, the desired inequalities follow from
$0\le K_{V,m}\le K_{W,m}$.
All constants depend only on the curvature lower bound, the density,
and the indicated cutoff bounds, and not on $K_{W,m}(w,w)$.
\end{proof}

The common singularity type is determined by Nadel's multiplier
ideal theorem \cite{Nadel}, in the local generation form of
\cite[Proposition~5.7]{D00}. The following is its application
to our weighted kernels. Note that a smooth positive density does not
change local integrability.
\begin{lemma}\label{lem:common-ideal}
In the setting of Lemma~\ref{lem:localization}, take $m$ sufficiently large that $mu-\log\rho$ has a strictly positive curvature lower bound. The ideal generated by an orthonormal basis of $\mathscr H_m(V)$, restricted to any relatively compact smaller set, equals
\begin{equation}\label{eq:common-ideal}
\mathcal I(mu)_w=\{g\in\mathcal O_{V,w}:|g|^2e^{-mu}
                         \text{ is locally integrable near }w\}.
\end{equation}
It is coherent. When two local potentials differ by a smooth pluriharmonic function, these sheaves agree on the overlap.
\end{lemma}
\begin{proof}
Set $\psi=mu-\log\rho$. Since $\rho$ is smooth and strictly
positive on a neighborhood of $\overline V$, there is a constant
$C$ such that
\[
dd^c\psi\ge(m\delta-C)\,dd^c|z|^2>0
\]
for all sufficiently large $m$. Moreover,
\[
e^{-\psi}=e^{-mu}\rho,\qquad
\mathscr H_m(V)=\mathscr H(V,\psi),\qquad
\mathcal I(\psi)=\mathcal I(mu).
\]
The last equality follows because $\rho$ is locally bounded
above and below by positive constants. Applying the local
$L^2$-generation theorem
\eqref{eq:ideal-generators}
proves both coherence and generation by finitely many basis
functions on every relatively compact smaller set.

If two local potentials $u_i,u_j$ differ by a smooth
pluriharmonic function, then
\[
e^{-mu_i}=e^{-m(u_i-u_j)}e^{-mu_j}.
\]
The first factor is locally bounded above and below by positive
constants. Hence the local integrability conditions coincide,
and $\mathcal I(mu_i)=\mathcal I(mu_j)$ on the overlap.
\end{proof}

After principalization \cite[Theorem~1.10]{BM97}, finite
generation identifies a common divisor of the basis functions.
We justify division of the entire Hilbert-valued series below.
This gives a smooth remainder on the modification and a bounded
logarithmic remainder on the base.
\begin{lemma}\label{lem:smooth-remainder}
Let $(\sigma_\ell)$ be a Hilbert basis as above, generating the
coherent ideal $\mathcal I$. Let $\pi:Y\to V$ be a
principalization with
$\mathcal I\cdot\mathcal O_Y=\mathcal O_Y(-E)$.
If $s$ is a local generator of $\mathcal O_Y(-E)$, then
\begin{equation}\label{eq:smooth-remainder}
\sum_\ell|\sigma_\ell\circ\pi|^2=|s|^2h,
\end{equation}
where $h$ is smooth and strictly positive. Thus the  Bergman kernel $b_m$, has a smooth resolved remainder.
\end{lemma}
\begin{proof}
The basis functions define a holomorphic map
$z\mapsto(\sigma_\ell(z))_{\ell\ge1}$ with values in $\ell^2$.
Indeed, the components are holomorphic and
\[
\sup_{z\in K}\sum_{\ell\ge1}|\sigma_\ell(z)|^2
=\sup_{z\in K}K_m(z,z)<\infty
\qquad(K\Subset V).
\]
Local boundedness and the Cauchy formula give holomorphicity
with values in $\ell^2$.

On a principalization chart, write
$s=\gamma z_1^{q_1}\cdots z_k^{q_k}$, with $\gamma$ a
nowhere-vanishing holomorphic function.
Each $\sigma_\ell\circ\pi$ is divisible by $s$.
For a factor $z_1^{q_1}$, the first $q_1$ Taylor
coefficients vanish in every component, hence also as vectors
in $\ell^2$. Division of the vector-valued Taylor series by
$z_1^{q_1}$ therefore preserves holomorphicity. Repeating this
for all factors gives
\[
y\longmapsto
\left(\frac{\sigma_\ell\circ\pi}{s}(y)\right)_{\ell\ge1}
\quad\text{holomorphic with values in }\ell^2.
\]
Consequently
\[
h=\sum_{\ell\ge1}\left|\frac{\sigma_\ell\circ\pi}{s}\right|^2
\]
is smooth and satisfies
\eqref{eq:smooth-remainder}.

Fix $y\in Y$. Choose finitely many basis elements generating
the local weighted integrability ideal near $\pi(y)$, as in
\eqref{eq:ideal-generators}.
The local form of principalization stated above says that
their quotients by $s$ have no common zero. Thus
\[
h(y)\ge\sum_{\ell=1}^N
\left|\frac{\sigma_\ell\circ\pi}{s}(y)\right|^2>0.
\]
The quotient of $h$ by this finite sum is continuous and
independent of the choice of $s$. Properness of $\pi$
makes it bounded on the inverse image of a sufficiently small
compact neighborhood. Multiplying by $|s|^2$ and descending
to that neighborhood gives
\[
\sum_{\ell=1}^N|\sigma_\ell(z)|^2
\le K_m(z,z)\le C\sum_{\ell=1}^N|\sigma_\ell(z)|^2.
\]
It follows that
\[
\log K_m(z,z)
=\log\sum_{\ell=1}^N|\sigma_\ell(z)|^2+O(1),
\qquad
\log(K_m\circ\pi)=\log|s|^2+\log h.
\]
The latter remainder is smooth, as asserted.
\end{proof}

\subsection{Global regularization} 
We now carry out Demailly's local-to-global construction
\cite[Section~3]{D92}, using the regularized maximum of
\cite[Chapter~I, Lemma~5.18]{D12}.
Lemmas~\ref{lem:bergman-trace} and
\ref{lem:localization} provide the trace control
and overlap estimates; Lemma~\ref{lem:smooth-remainder}
provides the resolved smoothness.
\begin{theorem}\label{thm:bergman}
For every $T\in\mathcal C_a$ and $\varepsilon>0$, there exists a positive current $T_\varepsilon\in\alpha$, with analytic singularity type and smooth remainder after principalization, such that
\begin{equation}\label{eq:bergman-resolution}
T_\varepsilon\in\mathcal C_{a+\varepsilon},\qquad
\pi_\varepsilon^*T_\varepsilon
 =\sigma_\varepsilon+[D_\varepsilon],\qquad
\sigma_\varepsilon\ge\frac{\pi_\varepsilon^*\omega}{a+\varepsilon}.
\end{equation}
Here $\pi_\varepsilon$ is the principalization, which is a finite composition of blowups with smooth centers, $D_\varepsilon$ is effective real, and $\sigma_\varepsilon$ is smooth and semipositive. The currents can be selected with $T_\varepsilon\to T$ weakly as $\varepsilon\downarrow0$.
\end{theorem}

\begin{proof}
Write $T=\kappa+dd^c\phi$, with $\kappa\in\alpha$ smooth, and
fix a smooth positive volume form $dV$.
For $\eta>0$, choose a finite cover
\[
U_i'\Subset U_i\Subset V_i\Subset\widetilde V_i,\qquad
\bigcup_iU_i'=X,
\]
by nested coordinate balls, and constant positive forms $B_i$ with
\begin{equation}\label{eq:frozen-background}
B_i\le\omega\le(1+\eta)B_i
\quad\text{on }\widetilde V_i.
\end{equation}
Let $dd^cq_i=\kappa$ on $\widetilde V_i$ and set
$u_i=\phi+q_i$. Then
\[
dd^cu_i=T,\qquad T\ge\omega/a,\qquad
\tr(T_{\rm ac}^{-1}B_i)\le a.
\]
By the discussion after Definition~\ref{def:trace-family}, the convolutions
of $u_i$ satisfy the same frozen matrix inequality.

Let $K_{i,m}$ be the kernel on $V_i$ for the weight
$e^{-mu_i}dV$, and set
\[
v_{i,m}=\frac1m\log K_{i,m}(z,z)-q_i.
\]
Lemma~\ref{lem:bergman-trace} gives, away from its singular locus,
\begin{equation}\label{eq:local-trace}
\tr_{\kappa+dd^cv_{i,m}}\omega
\le\frac{(1+\eta)a}{1-C/m}.
\end{equation}
Here and below $C$ may depend on the fixed finite cover and
may increase from one estimate to the next.
All limits in $m$ are taken after choosing $\eta$ and this cover.

We next compare the potentials on overlaps.
For $w\in U_i\cap U_j$, choose a common ball
$w\in W\Subset V_i\cap V_j$, with uniform interior and cutoff
bounds. Such bounds exist because
$\overline U_i\Subset V_i$ and the cover is finite.
On the coordinate ball $W$, one can write $u_i-u_j=\operatorname{Re}h_{ij}$, with
$h_{ij}$ holomorphic, since $u_i-u_j$ is pluriharmonic.
Multiplication by $e^{mh_{ij}/2}$ is an isometry between the two
weighted spaces on $W$, since both use $dV$.
Their kernels therefore satisfy
\[
K_{W,i,m}=e^{m\operatorname{Re}h_{ij}}K_{W,j,m},\qquad
\frac1m\log K_{W,i,m}-q_i
=\frac1m\log K_{W,j,m}-q_j.
\]
Applying Lemma~\ref{lem:localization} in each chart
and taking logarithms yields
\begin{equation}\label{eq:overlap}
|v_{i,m}-v_{j,m}|
\le\frac2m\log(1+C/\sqrt m)=O(m^{-3/2})
\end{equation}
on $U_i\cap U_j$ off the singular locus.
By Lemma~\ref{lem:common-ideal}, the local weighted
integrability conditions agree:
\[
\mathcal I(mu_i)=\mathcal I(mu_j)\quad\text{on overlaps}.
\]
They define a nonzero coherent ideal $\mathcal I_m$ on $X$,
locally generated by the corresponding basis functions. Its zero
set $Z_m$ is characterized in each chart by
\[
Z_m\cap U_i=\{x\in U_i:K_{i,m}(x,x)=0\}.
\]

Choose $-1\le\vartheta_i\le1$ smooth on $U_i$, with
$\vartheta_i=1$ on $U_i'$, $\vartheta_i=-1$ near
$\partial U_i$, and $-C\omega\le dd^c\vartheta_i\le C\omega$.
Put
\[
\tau_m=m^{-1/2},\qquad
\widetilde v_{i,m}=v_{i,m}+\tau_m\vartheta_i.
\]
For large $m$, \eqref{eq:local-trace}
gives $\kappa+dd^cv_{i,m}\ge\omega/(2(1+\eta)a)$.
Consequently, after increasing $C$,
\[
\begin{aligned}
\kappa+dd^c\widetilde v_{i,m}
&\ge\kappa+dd^cv_{i,m}-C\tau_m\omega\\
&\ge(1-C\tau_m)(\kappa+dd^cv_{i,m}).
\end{aligned}
\]
Combining inversion with
\eqref{eq:local-trace}, and
increasing $C$ once more, gives
\begin{equation}\label{eq:perturbed-trace}
\tr_{\kappa+dd^c\widetilde v_{i,m}}\omega
\le A_m:=\frac{(1+\eta)a}{1-C\tau_m}
\end{equation}
for all sufficiently large $m$.

Let $M_r$ be the standard regularized maximum, defined using a
product mollifier as in \cite[I.5.18]{D12}. Its properties are
\[
\begin{gathered}
\max_i t_i\le M_r(t)\le\max_i t_i+r,\qquad
M_r(t+s\mathbf1)=M_r(t)+s,\\
\partial_iM_r\ge0,\qquad
\sum_i\partial_iM_r=1,\qquad D^2M_r\ge0.
\end{gathered}
\]
Moreover, an entry $t_i$ can be omitted whenever
$t_i+2r<\max_{k\ne i}t_k$.
For $x$ near $\partial U_i$, choose $j$ with $x\in U_j'$.
Then \eqref{eq:overlap} gives
\[
\widetilde v_{j,m}(x)-\widetilde v_{i,m}(x)
\ge2\tau_m-\frac2m\log(1+C/\sqrt m)>\tau_m/2
\]
for large $m$.
Hence the $i$-th entry is omitted near its boundary, and
\begin{equation}\label{eq:glued-potential}
v_m(x)=M_{\tau_m/4}
 \bigl((\widetilde v_{i,m}(x))_{\,i:\,x\in U_i}\bigr),
\qquad x\in X\setminus Z_m,
\end{equation}
is a well-defined smooth function.

On a neighborhood with a fixed set of contributing charts,
write $M_i,M_{ij}$ for the first and second derivatives of
this regularized maximum. The chain rule gives
\[
\kappa+dd^cv_m
=\sum_i M_i(\kappa+dd^c\widetilde v_{i,m})
 +i\sum_{i,j}M_{ij}
       \partial\widetilde v_{i,m}\wedge
       \bar\partial\widetilde v_{j,m}.
\]
The last form is semipositive, and $M_i\ge0$, $\sum_iM_i=1$.
Convexity and monotonicity of the inverse-trace function imply
\[
\tr_{\kappa+dd^cv_m}\omega
\le\sum_i M_i
 \tr_{\kappa+dd^c\widetilde v_{i,m}}\omega
\le A_m
\quad\text{on }X\setminus Z_m.
\]

Choose a principalization $\pi_m:Y_m\to X$ such that
\[
\mathcal I_m\cdot\mathcal O_{Y_m}=\mathcal O_{Y_m}(-E_m),
\]
where $E_m$ is an effective integral divisor.
If $s$ is a local generator of $\mathcal O_{Y_m}(-E_m)$,
Lemma~\ref{lem:smooth-remainder} gives
\[
K_{i,m}\circ\pi_m=|s|^2h_{i,m},\qquad
g_{i,m}:=\frac1m\log h_{i,m}-q_i\circ\pi_m
             +\tau_m\vartheta_i\circ\pi_m\in C^\infty.
\]
Thus
\[
\widetilde v_{i,m}\circ\pi_m
=\frac1m\log|s|^2+g_{i,m}.
\]
The boundary inequalities for differences of the $g_{i,m}$
extend across $\{s=0\}$ by continuity. Translation invariance
of $M_{\tau_m/4}$ therefore yields locally
\[
v_m\circ\pi_m
=\frac1m\log|s|^2+g_m,\qquad
g_m=M_{\tau_m/4}(g_{i,m})\in C^\infty.
\]
By \eqref{eq:poincare-lelong},
\begin{equation}\label{eq:resolved-current}
\begin{gathered}
\pi_m^*(\kappa+dd^cv_m)=\sigma_m+[D_m],\\
D_m=\frac{2\pi}{m}E_m,\qquad
\sigma_m=\pi_m^*\kappa+dd^cg_m.
\end{gathered}
\end{equation}
The local expressions for $\sigma_m$ agree and define a smooth
closed form on $Y_m$.

Each $K_{i,m}$ is continuous and vanishes on $Z_m$.
The bounds for the regularized maximum show that $v_m\to-\infty$
as a point approaches $Z_m$.
Extend $v_m$ by $-\infty$ there.
The functions $q_i+v_m$ are plurisubharmonic off $Z_m$ and
locally bounded above, so their upper-semicontinuous extensions
are plurisubharmonic. Hence
\[
T_m:=\kappa+dd^cv_m\ge0,\qquad T_m\in\mathcal C_{A_m}.
\]
Lemma~\ref{lem:smooth-remainder} and the bounded
overlap differences give the stated analytic singularity type.
On the complement of the exceptional and base divisors,
\[
\sigma_m\ge\frac{\pi_m^*\omega}{A_m}.
\]
Both forms are smooth, so this inequality holds on all of $Y_m$.

Finally, for the fixed cover,
\[
A_m\longrightarrow(1+\eta)a,\qquad
|v_m-v_{i,m}|\le C\tau_m
\quad\text{on }U_i\setminus Z_m.
\]
The local Bergman approximation gives
$\|v_{i,m}-\phi\|_{L^1(U_i,dV)}\to0$, and therefore
\[
\|v_m-\phi\|_{L^1(X,dV)}
\le\sum_i\|v_{i,m}-\phi\|_{L^1(U_i,dV)}
 +C\tau_m\operatorname{Vol}_{dV}(X)
\longrightarrow0.
\]
Given $\varepsilon>0$, first choose $\eta$ with
$(1+\eta)a<a+\varepsilon$, then $m$ so large that
\[
A_m\le a+\varepsilon,\qquad
\|v_m-\phi\|_{L^1(X,dV)}<\varepsilon.
\]
Set $T_\varepsilon=T_m$, $\pi_\varepsilon=\pi_m$.
These choices give \eqref{eq:bergman-resolution}
and $T_\varepsilon\rightharpoonup T$ as $\varepsilon\downarrow0$.
\end{proof}

The same construction applies with $q=n-1$ in
Lemma~\ref{lem:bergman-trace}. We record this consequence
of Theorem~\ref{thm:bergman} for the strict
subsolution in Section~\ref{sec:concentration}.
\begin{lemma}\label{lem:partial-trace}
Let $Q\in\alpha$ be a K\"ahler current on a compact K\"ahler $n$-fold, $n\ge2$, with $P_\omega(Q_{\rm ac})\le a_0$ almost everywhere. For every $\epsilon>0$ and smooth representative $\kappa\in\alpha$, there are a proper closed analytic set $Z$ and $B=\kappa+dd^c\psi$ such that
\[
 B\ge\frac{\omega}{a_0+\epsilon},\qquad
 B\in C^\infty(X\setminus Z),\qquad
 P_\omega(B)\le a_0+\epsilon\quad\text{on }X\setminus Z.
\]
The potential $\psi\to-\infty$ along $Z$, has logarithmic singularity type in a coherent ideal, and has a smooth logarithmic remainder after principalization.
\end{lemma}
\begin{proof}
Apply the construction of Theorem~\ref{thm:bergman}, using Lemma~\ref{lem:bergman-trace} with $q=n-1$. Every step uses only convexity, monotonicity, homogeneity, and the lower bound implied by $P_\omega\le a_0$. Thus the same global potential $v_m$ satisfies
\[
 P_\omega(\kappa+dd^cv_m)\le\frac{(1+\eta)a_0}{1-Cm^{-1/2}}
 \quad\text{on }X\setminus Z_m.
\]
Choose $\eta>0$, then one $m$, so that this bound is at most $a_0+\epsilon$. Set $\psi=v_m$ and $Z=Z_m$. The principalization and extension arguments in that theorem give the remaining assertions.
\end{proof}

\section{The analytic and birational thresholds}\label{sec:threshold}
In this section, we establish the equivalence between $a_*$ from the analytic threshold and $\zeta$ from the birational threshold. One side is proved by a Lamari type duality \cite[Lemma 3.3]{L99}, and the other side is proved using the Bergman approximation discussed in the previous section. We also obtain K\"ahler residual models
on which the numerical inequalities for $J$-stability become asymptotically nonnegative.

\subsection{Birational tests}

We prove some basic estimates for the birational minimal slope $\zeta$ in this subsection. 

The first observation is that Hironaka's domination theorem \cite[Corollary~2]{H75} and a
small exceptional-divisor perturbation allow us to use K\"ahler
residual classes in Datar--Mete--Song's definition of minimal slope $\zeta$
\cite[Definition~1.3]{DMS26}. The slopes are preserved in the limit.

\begin{lemma}
    \label{lem:kahler-domination}
The minimal slope $\zeta$ is unchanged if its modifications are required to be smooth compact K\"ahler manifolds, and the residual class $L=\pi^*\alpha-[D]$ can be chosen to be K\"ahler. 
\end{lemma}
\begin{proof}
    We use the following form of Hironaka's resolution of singularities \cite[Corollary 2]{H75}: if $\pi:Y\to X$ is a modification between smooth compact complex manifolds, there are a modification $p:Z\to Y$ and a composite $\pi\circ p:Z\to X$ of blow-ups along smooth centers. For a composition of blowups there is an effective exceptional real divisor $E$ such that
\begin{equation}\label{eq:kahler-infimum}
K=p^*\pi^*\alpha-[E]\quad\hbox{is K\"ahler}.
\end{equation}
For $L=\pi^*\alpha-[D]$ big and nef set
\begin{equation}\label{eq:kahler-perturbation}
L_\varepsilon=(1-\varepsilon)p^*L+\varepsilon K
=p^*\pi^*\alpha-\big((1-\varepsilon)p^*D+\varepsilon E\big),\qquad 0<\varepsilon<1.
\end{equation}
Its remainder is effective and its class is K\"ahler, since it is a positive multiple of a K\"ahler class plus a nef class. Continuity and $L^n>0$ give $c_{L_\varepsilon}\to c_L$. Taking infima proves the assertion.
\end{proof}
\begin{lemma}\label{lem:positive-slope}
$0<\zeta\leq c$.
\end{lemma}
\begin{proof}

The identity test proves the upper bound. Choose a K\"ahler form $\kappa\in\alpha$ and $\delta>0$ such that $\omega\geq\delta\kappa$. For any test, nefness and effectiveness give

\[
L^{n-1}\pi^*\beta
\geq\delta L^{n-1}\pi^*\alpha
=\delta L^n+\delta L^{n-1}[D]
\geq\delta L^n.
\]
The intersection inequalities follow by approximating the nef class by K\"ahler classes. Hence every test slope is at least $n\delta$.

\end{proof}

A bounded slope also prevents the residual volume from tending
to zero. This is a direct application of the
Khovanskii--Teissier inequality in the form of
\cite[Theorem~2.1]{LX16}.
\begin{lemma}\label{lem:volume-bound}
If an admissible test $L$ has slope at most $M$, then
\[
L^n\geq(n/M)^n\beta^n.
\]
In particular, a minimizing sequence has a uniform positive lower bound for its top volume.
\end{lemma}
\begin{proof}
We use the mixed Khovanskii--Teissier inequality for nef and big classes \cite[Theorem 2.1]{LX16}: if $B_1,\ldots,B_n$ are big and nef real $(1,1)$-classes on a compact K\"ahler $n$-fold, then
\[
B_1\cdots B_n\geq\prod_{j=1}^n(B_j^n)^{1/n},
\]
with equality precisely when all $B_j$ are proportional. 

Apply it to $n-1$ copies of $L$ and one copy of $\pi^*\beta$, all big and nef on the K\"ahler model from the previous lemma. Then
\[
M\geq n\frac{L^{n-1}\pi^*\beta}{L^n}
\geq n\left(\frac{\beta^n}{L^n}\right)^{1/n}.
\]
This proves the lemma. 
\end{proof}

Combining the preceding mixed-intersection inequality with
$L^n\le\alpha^n$ gives a lower bound expressed entirely
in the original classes.
\begin{lemma}
    $\zeta\geq n(\beta^n/\alpha^n)^{1/n}$.
\end{lemma}
\begin{proof}
    $\pi^*\alpha$ and $L$ are nef and their difference is the effective divisor $D$, so
\[
\alpha^n-L^n
=D\cdot\sum_{k=0}^{n-1}(\pi^*\alpha)^{n-1-k}L^k\geq0.
\]
The lower bound follows from the mixed Khovanskii--Teissier inequality.
\end{proof}

\subsection{The analytic threshold}
In this section, we discuss the analytic threshold which characterizes the solvability of equation $\tr_{T_{\rm ac}}\omega=a$.  The analytic threshold is more natural for solving the $J$-equation.  Recall the definitions \begin{equation}\label{eq:threshold}
 \mathcal C_a=\{T\in\alpha:T\ge0,\ dT=0,\ T_{\rm ac}>0,
                          \ \tr_{T_{\rm ac}}\omega\le a\text{ a.e.}\},
 \qquad a_*=\inf\{a>0:\mathcal C_a\ne\varnothing\}.
\end{equation}

We first show that $a_*$ is attained for $\cC_{a_*}\neq \emptyset$.
\begin{lemma}\label{lem:attainment}
 $0<a_*<\infty$, and $\mathcal C_{a_*}\ne\emptyset$. Every $T\in\mathcal C_a$ satisfies
\begin{equation}\label{eq:weak-subsolution}
T\geq\omega/a,\qquad
a\langle T^p\rangle-p\omega\wedge\langle T^{p-1}\rangle\geq0
\quad(1\leq p\leq n).
\end{equation}
\end{lemma}
This proves attainment of an analytic weak-subsolution threshold.
\begin{proof}
For $\kappa\in\alpha$ K\"ahler,
\[
 0<\frac{\beta\alpha^{n-1}}{\alpha^n}\le a_*
 \le\max_X\tr_\kappa\omega<\infty.
\]
Indeed, $T\in\mathcal C_a$ implies $T\ge\omega/a$; pair with $\kappa^{n-1}$, one gets $\alpha^n\geq \alpha^{n-1}\beta/a$. This implies the lower bound for $a_*$. The upper bound is trivial.  Choose $a_j\downarrow a_*$, $T_j\in\mathcal C_{a_j}$. Their masses are $\alpha^n$, so a subsequence converges weakly to a positive current $T\in\alpha$. Equation~\eqref{eq:convolution}, at each fixed radius, gives
\[
 \tr_{T*\varrho_r}B\le a_*
 \quad\text{for every constant }0<B\le\omega.
\]
Differentiation of measures and shrinking frozen backgrounds yield $T\in\mathcal C_{a_*}$. Lemma~\ref{lem:np-cone}, followed by the same freezing argument, gives the non-pluripolar inequalities for every $T\in\mathcal C_a$.
\end{proof}

The following Cauchy-Schwarz inequality will be used many times. 
\begin{lemma}[Trace Cauchy--Schwarz inequality]\label{lem:trace-cs}
For positive definite Hermitian matrices $A,B$ of the same size,
\begin{equation}\label{eq:trace-cs}
\bigl(\tr B^{1/2}\bigr)^2
\le (\tr A)\,\tr(A^{-1}B).
\end{equation}
Equality holds if and only if $A=tB^{1/2}$ for some $t>0$.
\end{lemma}
\begin{proof}
Choose unitary basis such that $A=\text{diag}(a_1,...,a_n)$. Write $B^\frac{1}{2}$ as $(c_{i\bar j})$. Then Cauchy-Schwarz inequality gives \[
\bigl(\sum_ic_{i\bar i}\bigl)^2\leq \bigl(\sum_ia_i\bigl)\bigl(\sum_i\frac{|c_{i\bar i}|^2}{a_i}\big)\leq \bigl(\sum_ia_i\bigl)\bigl(\sum_i\frac{\sum_j|c_{i\bar j}|^2}{a_i}\big)=\tr A\tr(A^{-1}B).
\]
\end{proof}
Usually we apply the above inequality in the following form\[
(\tr\sqrt{\chi^{-1}\omega})^2\leq \tr_{\chi}\Omega\tr_{\Omega}\omega,
\] for any positive Hermitian form $\Omega$. 

The following theorem is a trace-constrained version of Lamari duality \cite[Lemme~3.3]{L99}; see \cite[Lemma~2.1]{Tosatti16} for an English account. The method of using Hahn-Banach separation was due to Sullivan \cite{Sullivan76}. 
\begin{theorem}[Lamari-type characterization]\label{thm:duality}
For a K\"ahler class $\alpha$ and fixed smooth K\"ahler form $\omega$ on a compact K\"ahler manifold of dimension $n\geq2$, for $a>0$, one has $\mathcal C_a\ne\varnothing$ if and only if
\[
 \int_X\bigl(\tr\sqrt{\chi^{-1}\omega}\bigr)^2\chi^n
 \le na\int_X\alpha\wedge\chi^{n-1}
\]
for every Gauduchon metric $\chi$, that is, every smooth Hermitian metric satisfying $dd^c(\chi^{n-1})=0$. 

Equivalently,
\begin{equation}\label{eq:duality}
a_*=
\sup_{\substack{\chi\ {\rm smooth\ Hermitian\  metric}\\dd^c(\chi^{n-1})=0}}
\frac{\displaystyle\int_X
 \big(\tr\sqrt{\chi^{-1}\omega}\big)^2\chi^n}
{\displaystyle n\int_X\alpha\wedge\chi^{n-1}}.
\end{equation}
\end{theorem}

\begin{proof}
Lemma~\ref{lem:trace-cs}, positivity of singular coefficients, and $dd^c\chi^{n-1}=0$ give
\[
 T\in\mathcal C_a\quad\Longrightarrow\quad
 \int_X\kappa\wedge\chi^{n-1}
 =\int_XT\wedge\chi^{n-1}
 \ge\frac1{na}\int_X(\tr\sqrt{\chi^{-1}\omega})^2\chi^n,
\]where we used that $nT_{\rm ac}\wedge\chi^{n-1}=\tr_{\chi}T_{\rm ac}\chi^n$.
We denote the right-hand side of \eqref{eq:duality} as $b_*$. 
Thus $0<b_*\le a_*$.

Fix $a\ge b_*$. Let $E=\mathcal A^{n-1,n-1}_{\mathbb R}(X)$, with its $C^\infty$ Fr\'echet topology, and give the space $E'$ of real $(1,1)$-currents its weak-* topology. Set
\[
 \mathscr C_a=\{U\in E':U\ge0,\ \tr_{U_{\rm ac}}\omega\le a\text{ a.e.}\},
 \qquad
 \mathscr F=\{\kappa+dd^cu:u\in\mathcal D'(X,\mathbb R)\}.
\] Here $\mathcal{D}'(X,\R)$ is the space of distributions. 
The currents in $\mathscr C_a$ need not to be closed. This nonempty convex set is weakly closed by \eqref{eq:convolution}. The $dd^c$-lemma identifies $\mathscr F$ with the closed currents in class $\alpha$. It is weakly closed, since membership is equivalent to
\[
 dS=0,\qquad \langle S-\kappa,\eta\rangle=0
 \quad\text{for every smooth closed real }(2n-2)\text{-form }\eta.
\]
Only the $(n-1,n-1)$-component enters the pairing.

The difference $\mathscr C_a-\mathscr F$ is weakly closed. Indeed, if a net $U_\nu-S_\nu\to R$, with $U_\nu\in\mathscr C_a$, $S_\nu\in\mathscr F$, then
\[
 \langle U_\nu,\kappa^{n-1}\rangle
 =\langle U_\nu-S_\nu,\kappa^{n-1}\rangle+\alpha^n.
\]
Positivity therefore bounds the masses of $U_\nu$ on a tail. Weak compactness gives a subnet $U_\nu\to U\in\mathscr C_a$, and $S_\nu\to U-R\in\mathscr F$. Hence $R\in\mathscr C_a-\mathscr F$.

If $\mathscr C_a\cap\mathscr F=\varnothing$, Hahn--Banach separation gives a weak-* continuous functional $\ell$ and $c_0>0$ such that $\ell(U-S)\ge c_0$. Such a functional is evaluation on a smooth form: continuity gives finitely many $\eta_j\in E$ with
\[
 |\ell(R)|\le C\max_j|\langle R,\eta_j\rangle|.
\]
It factors through $R\mapsto(\langle R,\eta_j\rangle)_j$, so $\ell(R)=\langle R,\sigma\rangle$ for a finite real linear combination $\sigma$ of the $\eta_j$. Thus
\[
 \int_X(U-S)\wedge\sigma\ge c_0
 \quad(U\in\mathscr C_a,\ S\in\mathscr F).
\]
Since each $S=\kappa+\ddc u$ in the distribution sense, $S_t:=\kappa+t\ddc u\in \mathscr{F}$ for all $t\in \R$. The separating inequality gives $\int_X\ddc u\wedge\sigma=0$. It follows $\ddc\sigma=0.$
Similarly, $U+tQ\in \mathscr{C}_a$ for any positive current $Q$ and $t>0$. The separation inequality with $t\to \infty$ gives $\int_XQ\wedge\sigma\geq 0$. It follows that the smooth form $\sigma\geq 0$. 

For sufficiently small $\epsilon>0$,
$\sigma_\epsilon=\sigma+\epsilon\kappa^{n-1}$ is smooth, strictly positive and $dd^c$-closed, and
\[
\int_X(U-\kappa)\wedge\sigma_\epsilon\geq c_0/2
\quad(U\in\mathscr C_a).
\] Now we apply the observation of Michelsohn \cite[pp.~279--280]{M82} that 
every smooth strictly positive $(n-1,n-1)$-form $\sigma_\epsilon$ has the form $\sigma_\epsilon=\chi^{n-1}$. Thus $\chi$ is Gauduchon here.
By \eqref{eq:trace-cs}, the infimum of $\int_XU\wedge\chi^{n-1}$ over $\mathscr C_a$ equals
\[
\frac1{na}\int_X
\big(\tr\sqrt{\chi^{-1}\omega}\big)^2\chi^n.
\]
There is no closure constraint on $U$, so the smooth pointwise minimizing form from \eqref{eq:trace-cs} attains this infimum. Separation therefore gives
\[
\int_X\kappa\wedge\chi^{n-1}+c_0/2
\leq\frac1{na}\int_X
\big(\tr\sqrt{\chi^{-1}\omega}\big)^2\chi^n
\leq\int_X\kappa\wedge\chi^{n-1},
\]
where the last inequality is $a\geq b_*$. This contradiction proves
$\mathscr C_a\cap\mathscr F\ne\varnothing$.
Every current in this intersection is in $\mathcal C_a$, and is automatically K\"ahler because it dominates $\omega/a$.
Hence $a_*\leq b_*$, proving \eqref{eq:duality}.

\end{proof}

With this Lamari-type characterization of $\cC_{a}$ being nonempty, one can show that the birational minimal slope is bounded below by the analytic slope $a_*$. 
\begin{proposition}\label{prop:slope-comparison}
For every Gauduchon metric $\chi$ on $X$, every K\"ahler modification $\pi:Y\to X$, and every effective real divisor $D$ such that $L=\pi^*\alpha-[D]$ is big and nef, set $c_L:=nL^{n-1}\pi^*\beta/L^n$. Then 
\begin{equation}\label{eq:slope-comparison}
\int_X(\tr\sqrt{\chi^{-1}\omega})^2\chi^n\leq n c_L\int_Y L\wedge\pi^*(\chi^{n-1})
\leq n c_L\int_X\alpha\wedge\chi^{n-1}.
\end{equation}
\end{proposition}

\begin{proof}

Write \[
\coc=\int_X (\tr\sqrt{\chi^{-1}\omega})^2\chi^n, \qquad   F=\pi^*[(\tr\sqrt{\chi^{-1}\omega})^2\chi^n]/h^n\geq 0,\] where $h$ is a fixed smooth K\"ahler form on $Y$. For $\epsilon>0$, the class $L_\epsilon:=L+\epsilon[h]$ is K\"ahler. We set \[
c_\epsilon:=\coc+\epsilon\int_Yh^n, \qquad c_{L,\epsilon}=nL_\epsilon^{n-1}\pi^*\beta/L_{\epsilon}^n.
\]
By Yau's theorem \cite{Y78}, there is a smooth K\"ahler form $\Omega_\epsilon\in L_\epsilon$, such that
\begin{equation}\label{eq:comparison-ma}
\Omega_\epsilon^n
=\frac{L_\epsilon^n}{c_\epsilon}(F+\epsilon)h^n.
\end{equation}

Off the exceptional set, the Cauchy-Schwarz inequality gives\[
\pi^*(\tr\sqrt{\chi^{-1}\omega})^2\leq \tr_{\pi^*\chi}\Omega_\epsilon\,\tr_{\Omega_\epsilon}\pi^*\omega
\]
So 
\begin{equation}\label{eq:comparison-cs}
\begin{aligned}
n\int_YL_\epsilon\wedge\pi^*\chi^{n-1}
&=\int_Y\tr_{\pi^*\chi}\Omega_\epsilon \pi^*\chi^n\geq\int_Y\frac{(\pi^*\tr\sqrt{\chi^{-1}\omega})^2}{\tr_{\Omega_\epsilon}\pi^*\omega}\pi^*\chi^n\geq\frac{\coc^2}{\displaystyle\int_Y\tr_{\Omega_\epsilon}\pi^*\omega Fh^n}.
\end{aligned}
\end{equation}
 Equation \eqref{eq:comparison-ma} gives
\[
\begin{aligned}
\int_Y\tr_{\Omega_\epsilon}\pi^*\omega Fh^n
&\leq\frac{c_\epsilon}{L_\epsilon^n}
       \int_Y\tr_{\Omega_\epsilon}\pi^*\omega\Omega_\epsilon^n=\frac{c_\epsilon}{L_\epsilon^n}
       n\int_Y\Omega_\epsilon^{n-1}\wedge \pi^*\omega
=c_\epsilon c_{L,\epsilon}.
\end{aligned}
\]
Combining this with \eqref{eq:comparison-cs} gives
\[
\coc^2\leq nc_\epsilon c_{L,\epsilon}
\int_YL_\epsilon\wedge\pi^*\chi^{n-1}.
\]
The intersection expressions are continuous in $\epsilon$, $L^n>0$, and $c_\epsilon\to \coc>0$. Letting $\epsilon\to0$ proves the first inequality in \eqref{eq:slope-comparison}.

For the second inequality use the positive $dd^c$-closed form $\pi^*\chi^{n-1}$:
\[
\int_YL\wedge\pi^*\chi^{n-1}
=\int_X\alpha\wedge\chi^{n-1}
-\int_D\pi^*\chi^{n-1}
\leq\int_X\alpha\wedge\chi^{n-1}.
\]
The integral over a real effective divisor is the corresponding nonnegative real linear combination of integrals over prime divisors. Since $c_L>0$, the claimed inequality follows.

\end{proof}

\begin{corollary}\label{cor:critical-family}
For every pair of K\"ahler classes $\alpha,\beta$ on a compact K\"ahler $n$-fold, $n\geq2$, and every K\"ahler form $\omega\in\beta$, there is a K\"ahler current $T\in\alpha$ with
\[
\tr\big((T_{\rm ac})^{-1}\omega\big)\leq\zeta
\quad\text{almost everywhere},
\]
and
\begin{equation}\label{eq:critical-family}
T\geq\frac{\omega}{\zeta},\qquad
\zeta\langle T^p\rangle-p\omega\wedge\langle T^{p-1}\rangle\geq0
\quad(1\leq p\leq n).
\end{equation}
In particular, $a_*\le\zeta$.
\end{corollary}
\begin{proof}

For each Gauduchon metric, \eqref{eq:slope-comparison} bounds its ratio in \eqref{eq:duality} by the slope $c_L$ of every admissible K\"ahler modification. Lemma~\ref{lem:kahler-domination} allows the infimum over all original modifications. Therefore every ratio in \eqref{eq:duality} is at most $\zeta$, and $a_*\leq\zeta$. Lemma~\ref{lem:attainment} gives $T\in\mathcal C_{a_*}\subseteq\mathcal C_\zeta$, and its conclusions give \eqref{eq:critical-family}.
\end{proof}

\subsection{Residual classes at the minimal slope}

For a K\"ahler residual class $L=\pi^*\alpha-[D]$, write
\begin{equation}\label{eq:residual-threshold}
 \Gamma(L,\beta)=\inf_{\substack{V\subsetneq Y\text{ irreducible}\\1\le p=\dim V<n}}
 \frac{(c_LL^p-p\pi^*\beta L^{p-1})[V]}{(n-p)L^p[V]},
 \qquad \Gamma_-(L,\beta)=\max\{0,-\Gamma(L,\beta)\}.
\end{equation}
We use $\inf\varnothing=+\infty$. Thus $\Gamma_-=0$ means that all these numerical inequalities are nonnegative, so called $J$-semistable. $\Gamma$ was called the stability threshold, and was introduced by Sjöström Dyrefelt in the study of optimal lower bounds of $J$-functional \cite{SD20}. 

To show the opposite inequality $\zeta\leq a_*$, we use the principalization given by Bergman approximations. One can perturb the smooth semipositive remainder to be K\"ahler, make it a admissible  modification, and keep the control on the inverse trace. 
\begin{lemma}\label{lem:residual-models}
Suppose $\pi:Y\to X$ is a composition of blowups with smooth centers, and
\[
R\in\mathcal C_a, \quad \pi^*R=\sigma+[D],\quad
\sigma\text{ smooth},\quad D\ge0.
\]
There are effective real divisors $D_t$ and K\"ahler forms $h_t\in L_t: =\pi^*\alpha-[D_t]$, $0<t<1$, such that
\begin{equation}\label{eq:residual-trace}
\tr_{h_t}\pi^*\omega\le a_t:=a/(1-t).
\end{equation}
In particular, 
\begin{equation}\label{eq:residual-numerical}
c_{L_t}=n\frac{\pi^*\beta L_t^{n-1}}{L_t^n}\le a_t,\qquad
\int_V(a_tL_t^p-p\pi^*\beta L_t^{p-1})\ge0
\end{equation}
for every proper positive-dimensional irreducible analytic $V\subset Y$, $p=\dim V$.
\end{lemma}
\begin{proof}
Set $b=\pi^*\omega$.
On the complement of the exceptional divisor and $\operatorname{Supp}D$,
the form $\sigma$ corresponds to $R$, and hence
\[
\tr_\sigma b\le a,\qquad \sigma\ge b/a.
\]
The second inequality extends to $Y$ by continuity.
 Choose an effective
exceptional real divisor $E$ and a K\"ahler form
$\kappa\in\pi^*\alpha-[E]$. Define
\begin{equation}\label{eq:residual-forms}
h_t=(1-t)\sigma+t\kappa,\qquad D_t=(1-t)D+tE.
\end{equation}
Then
\[
h_t>0,\qquad D_t\ge0,\qquad
[h_t]=\pi^*\alpha-[D_t].
\]
Off the exceptional and base divisors, matrix inversion gives
\[
h_t\ge(1-t)\sigma
\quad\Longrightarrow\quad
\tr_{h_t}b
\le\frac{\tr_\sigma b}{1-t}
\le\frac a{1-t}=a_t.
\]
Since $h_t>0$ everywhere, continuity proves
\eqref{eq:residual-trace} on $Y$.
Integration yields
\[
n\pi^*\beta L_t^{n-1}
=\int_Y\tr_{h_t}b\,h_t^n
\le a_tL_t^n.
\]

For a complex $p$-plane $P\subset T_yY$, choose an
$h_t$-orthonormal basis adapted to $P$.
The matrix $(b_{k\bar\ell})$ is semipositive, so
\[
\tr_{h_t|_P}(b|_P)
=\sum_{k=1}^p b_{k\bar k}
\le\sum_{k=1}^n b_{k\bar k}
=\tr_{h_t}b\le a_t.
\]
Taking $P=T_yV_{\rm reg}$ gives
\[
\begin{aligned}
\int_V(a_tL_t^p-p\pi^*\beta L_t^{p-1})
&=\int_{V_{\rm reg}}(a_th_t^p-pb\wedge h_t^{p-1})\\
&=\int_{V_{\rm reg}}
 \bigl(a_t-\tr_{h_t|_{TV_{\rm reg}}}
                   (b|_{TV_{\rm reg}})\bigr)h_t^p
\ge0.
\end{aligned}
\]
This proves \eqref{eq:residual-numerical}.
\end{proof}

\begin{corollary}\label{cor:threshold-models}
\begin{equation}\label{eq:slope-threshold}
\mathcal C_a\ne\varnothing\quad\Longrightarrow\quad\zeta\le a.
\end{equation}
Moreover, there is a sequence of \emph{K\"ahler} residual models $L_j=\pi_j^*\alpha-[D_j]$ with
\begin{equation}\label{eq:selected-models}
c_{L_j}\longrightarrow\zeta,\qquad
\Gamma_-(L_j,\beta)\longrightarrow0.
\end{equation}
\end{corollary}

 Combining \eqref{eq:slope-threshold} with the previous inequality $a_*\le\zeta$ gives $a_*=\zeta$.
 
\begin{proof}
Let $T\in\mathcal C_a$, and choose
$\varepsilon_j\downarrow0$, $t_j\downarrow0$.
Theorem~\ref{thm:bergman} and 
Lemma~\ref{lem:residual-models} give K\"ahler
residual models $L_j=\pi_j^*\alpha-[D_j]$ satisfying
\[
\tr_{h_j}\pi_j^*\omega\le
a_j:=\frac{a+\varepsilon_j}{1-t_j}\longrightarrow a,\qquad
\zeta\le c_{L_j}\le a_j.
\]
Passing to the limit proves
\eqref{eq:slope-threshold}.

Now choose $T\in\mathcal C_\zeta$, whose existence is supplied
by Corollary~\ref{cor:critical-family}.
The same construction gives
$\zeta\le c_{L_j}\le a_j\to\zeta$.
For every proper irreducible $p$-dimensional subvariety $V\subset Y_j$,
\eqref{eq:residual-numerical} implies
\[
\begin{aligned}
\frac{\int_V(p\pi_j^*\beta\,L_j^{p-1}-c_{L_j}L_j^p)}
     {(n-p)\int_VL_j^p}
&\le\frac{a_j-c_{L_j}}{n-p}\le a_j-c_{L_j}.
\end{aligned}
\]
Here $\int_VL_j^p>0$, $n-p\ge1$, and $a_j-c_{L_j}\ge0$.
Taking the supremum and then its nonnegative part yields
\[
0\le\Gamma_-(L_j,\beta)
\le a_j-c_{L_j}\le a_j-\zeta\longrightarrow0.
\]
Thus \eqref{eq:selected-models}
holds. Finally, $\mathcal C_a\ne\varnothing\Rightarrow a\ge\zeta$
gives $a_*\ge\zeta$, while
Corollary~\ref{cor:critical-family} gives the
reverse inequality.
\end{proof}

\section{Rigidity at the critical slope}\label{sec:rigidity}
In this section, we show that every current at the critical threshold satisfies
the trace equality almost everywhere, and that there is only one
such current. The proof constructs balanced metrics from smooth
$J$-solutions on the residual models and uses them in the dual
inequality of Section~\ref{sec:threshold}. Strict convexity then gives equality of the currents, including
their singular parts.
Throughout this section, $n\ge3$. For a residual test, set
\[
 q_L=c_LL^{n-1}-(n-1)\pi^*\beta L^{n-2},\qquad
 c_{L,D}=q_L[D].
\]

\subsection{An intersection estimate and balanced metrics}
We first vary the residual class toward $\pi^*\alpha$.
The defining infimum of Datar--Mete--Song \cite{DMS26} bounds
the slope along this segment. A one-sided Taylor estimate
then controls the divisorial pairing, with constants independent
of the modification.
\begin{lemma}\label{lem:orthogonality}
For any $M>0$, there is a constant $C>0$, independent of the modification model, such that every admissible test $(\pi,D)$ with $c_L\le M$ and $c_L-\zeta\le C/8$ satisfies
\begin{equation}\label{eq:divisor-pairing}
\cld=\int_D(c_LL^{n-1}-(n-1)\pi^*\beta L^{n-2})\le \frac{\alpha^n}{n}\sqrt{2C(c_L-\zeta)}.
\end{equation}
Thus every minimizing sequence of big-and-nef residual models has $\max\{\cld,0\}\to 0$.
\end{lemma}
\begin{proof}
   By Lemma~\ref{lem:volume-bound}, $L^n\geq (n/M)^n\beta^n$. Consider the path
\begin{equation}\label{eq:orthogonality}
L_s=L+sD=(1-s)L+s\pi^*\alpha,\qquad 0\le s\le1.
\end{equation}
The whole path is big and nef, with effective remainder $(1-s)D$. When $L$ is K\"ahler it stays K\"ahler for $s<1$. In particular,
\begin{equation}\label{eq:slope-variation}
c(s)=n\frac{\pi^*\beta L_s^{n-1}}{L_s^n}\ge\zeta.
\end{equation}
Since $\pi^*\alpha-L=D$ is effective,  $L^{p}\pi^*\alpha^{n-p}\leq\alpha^n$, and $\pi^*\beta L^{p}\pi^*\alpha^{n-1-p}\leq \beta\alpha^{n-1}$. 
For $v(s)=L_s^n$ and $w(s)=n\pi^*\beta L_s^{n-1}$, expansion gives 
\[
(n/M)^n\beta^n=L^n\le v\leq \alpha^n,
\]
\[
0\leq \int_DnL_s^{n-1}=v'(s)= nDL_s^{n-1}=n(L-\pi^*\alpha)L_s^{n-1}\leq nLL_{s}^{n-1}\leq n\alpha^n,
\]\[
|v''(s)|=|n(n-1)D^2L_s^{n-2}|=|n(n-1)(L-\pi^*\alpha)^2L^{n-2}|\leq n(n-1)\alpha^n.
\]
Similarly, 
\[
0<w\le nb,\quad 0\le w'\le n(n-1)b,\quad
|w''|\le4n(n-1)(n-2)b.
\]
The quotient $c=w/v$ thus has $|c''|\le C$ throughout $[0,1]$.
Direct differentiation gives
\begin{equation}\label{eq:divisor-bound}
c'(0)=-\frac{n \cld}{L^n}.
\end{equation} Taylor expansion of $c(s)$ gives 
\[
c(s)=c_L+c'(0)s+O(s^2).
\]
Combining \eqref{eq:slope-variation} and \eqref{eq:divisor-bound}, one has \[
n\cld/L^n\leq (c_L-\zeta)/s+Cs.
\]
For $0<c_L-\zeta\le C/8$, choose $s=\sqrt{2(c_L-\zeta)/C}$, giving the desired inequality.
\end{proof}

We next solve a perturbed smooth $J$-equation on each model
using Song's numerical criterion \cite{S20}.
Michelsohn's positive-root construction
\cite[pp.~279--280]{M82} turns its closed cone form
into a balanced metric. The estimates below show that these
metrics asymptotically realize the equality in the Cauchy-Schwarz inequality in Lemma~\ref{lem:trace-cs} and Theorem~\ref{thm:duality}.
\begin{lemma}\label{lem:balanced-metrics}
Suppose K\"ahler modifications $\pi_j:Y_j\to X$, with $L_j=\pi_j^*\alpha-[D_j]$ being K\"ahler,  satisfy
\[
c_{L_j}\to\zeta,\qquad \Gamma_-(L_j,\beta)\to0.
\]
Then there are smooth balanced metrics $\chi_j$ on $Y_j$ such that,
\begin{equation}\label{eq:balanced-pairings}
\begin{aligned}
n\pi_j^*\alpha[\chi_j^{n-1}]&=c_{L_j}L_j^n+o(1),\\
\int_{Y_j}
 \bigl(\tr\sqrt{\chi_j^{-1}\pi_j^*\omega}\bigr)^2\chi_j^n
&=c_{L_j}^2L_j^n+o(1).
\end{aligned}
\end{equation}
Moreover, 
\begin{equation}\label{eq:balanced-mass}
n\int_{Y_j}\pi_j^*\omega\wedge \chi_j^{n-1}\le C,\qquad
\chi_j^n\ge c_0(\pi_j^*\omega)^n,
\end{equation}
where $C,c_0>0$ are independent of the modifications.
Here $\chi_j^{-1}\pi_j^*\omega$ is the nonnegative,
$\chi_j$-self-adjoint endomorphism associated with
$\pi_j^*\omega$; its nonnegative square root is defined also on
the exceptional locus.

\end{lemma}
\begin{proof}
We construct $\chi_j$ from a smooth $J$-equation on each model.
Choose $\varepsilon_j>\Gamma_-(L_j,\beta)$ with $\varepsilon_j\to0$, choose
a K\"ahler form $\theta_j\in L_j$, and put
\[
\omega_j=\pi_j^*\omega+\varepsilon_j\theta_j,\qquad
c_j=c_{L_j}+n\varepsilon_j.
\]
Both $L_j$ and $[\omega_j]$ are K\"ahler. For every proper
irreducible $p$-dimensional subvariety $Z\subset Y_j$,
\[
\begin{aligned}
\int_Z(c_jL_j^p-p[\omega_j]L_j^{p-1})
&=\int_Z(c_{L_j}L_j^p-p\pi_j^*\beta L_j^{p-1})
  +(n-p)\varepsilon_j\int_ZL_j^p\\
&\ge(n-p)(\varepsilon_j-\Gamma_-(L_j,\beta))\int_ZL_j^p>0.
\end{aligned}
\]
The numerical criterion of Song \cite{S20} therefore gives a K\"ahler form
$h_j\in L_j$ satisfying $\tr_{h_j}\omega_j=c_j$.
Set
\begin{equation}\label{eq:balanced-form}
q_j=c_jh_j^{n-1}-(n-1)\omega_j\wedge h_j^{n-2}.
\end{equation}
This form is closed and strictly positive. Hence by the construction of Michelsohn
\cite[pp.~279--280]{M82} , there is a unique smooth positive
Hermitian form $\chi_j$ with $q_j=\chi_j^{n-1}$.
Thus $\chi_j$ is balanced, since $d(\chi_j^{n-1})=0$.

We first compute the cohomological pairing in
\eqref{eq:balanced-pairings}. By construction,
\begin{equation}\label{eq:balanced-class}
[q_j]=c_{L_j}L_j^{n-1}-(n-1)\pi^*\beta L_j^{n-2}+\varepsilon_jL_j^{n-1}, 
\end{equation}
Since $D_j$ is effective,
\[
0\le[q_j]D_j=c_{L_j,D_j}+\varepsilon_jL_j^{n-1}D_j,
\qquad
0\le L_j^{n-1}D_j=L_j^{n-1}(\pi_j^*\alpha-L_j)\le\pi_j^*\alpha\,L_j^{n-1}\le\alpha^n.
\]
Lemma~\ref{lem:orthogonality} gives
$(c_{L_j,D_j})_+\to0$, so
\begin{equation}\label{eq:divisor-limit}
-\varepsilon_j\alpha^n\le c_{L_j,D_j}\le(c_{L_j,D_j})_+,
\qquad c_{L_j,D_j}\longrightarrow0.
\end{equation}
Using $\pi_j^*\alpha=L_j+D_j$ and $nL_jq_{L_j}=c_{L_j}L_j^n$, we obtain
\[
n\pi_j^*\alpha[\chi_j^{n-1}]
=c_{L_j}L_j^n+nc_{L_j,D_j}
 +n\varepsilon_j\pi_j^*\alpha\,L_j^{n-1}
=c_jL_j^n+o(1).
\]

For the integral estimate, we work on the biholomorphic locus. For any real $(1,1)$-form $\eta$, \[
\begin{aligned}
    n\eta\wedge q_j=&c_j n\eta\wedge h_j^{n-1}-n(n-1)\omega_j\wedge\eta\wedge h_j^{n-2}\\
    =&c_j\tr_{h_j}\eta h_j^n-(\tr_{h_j}\omega_j\tr_{h_j}\eta-\langle \omega_j,\eta\rangle_{h_j})h_j^n\\
    =&\langle \omega_j,\eta\rangle_{h_j} h_j^n.
\end{aligned}
\]
Hence \[\langle\chi_j,\eta\rangle_{\chi_j}\chi_j^n=n\eta\wedge\chi_j^{n-1}=\langle \omega_j,\eta\rangle_{h_j}h_j^n \quad  \text{ for any real $(1,1)$-form $\eta$.}\]
It follows that 
\begin{equation}\label{eq:metric-identity}
\frac{\chi_j^n}{h_j^n}\chi_j^{-1}h_j=h_j^{-1}\omega_j.
\end{equation}
Taking the determinant gives \[
\left(\frac{\chi_j^n}{h_j^n}\right)^{n-1}=\frac{\omega_j^n}{h_j^n}.
\]
Consequently,
\[
\bigl(\tr\sqrt{\chi_j^{-1}\pi_j^*\omega}\bigr)^2\chi_j^n
=
\bigl(\tr\sqrt{h_j^{-1}\omega_j h_j^{-1}\pi^*_j\omega}\bigr)^2h_j^n.
\]

In an $h_j$-unitary frame put $A=h_j^{-1}\omega_j$ and $B=h_j^{-1}\pi_j^*\omega$. Then $A\ge B\ge0$, and
\[
 B^{1/2}AB^{1/2}\ge B^2,\qquad
 A^{1/2}BA^{1/2}\le A^2.
\]
The two matrices on the left have the same eigenvalues. Monotonicity of the square root gives
\[
 \begin{aligned}
 (\tr B)^2h_j^n
 &\le\bigl(\tr\sqrt{A^{1/2}BA^{1/2}}\bigr)^2h_j^n\\
 &=\bigl(\tr\sqrt{\chi_j^{-1}\pi_j^*\omega}\bigr)^2\chi_j^n
 \le(\tr A)^2h_j^n=c_j^2h_j^n.
 \end{aligned}
\]
Now
$\tr_{h_j}\pi_j^*\omega
=c_j-\varepsilon_j\tr_{h_j}\theta_j$.
Integrating on $Y_j$ gives
\begin{equation}\label{eq:trace-integral}
\begin{aligned}
0\le c_j^2L_j^n
 -\int_{Y_j}
  \bigl(\tr\sqrt{\chi_j^{-1}\pi_j^*\omega}\bigr)^2\chi_j^n
&\le2c_j\varepsilon_j
 \int_{Y_j}(\tr_{h_j}\theta_j)h_j^n\\
&=2nc_j\varepsilon_jL_j^n.
\end{aligned}
\end{equation}

Since $L_j^n\le\alpha^n$, the numbers $c_j$ are bounded, and
$c_j-c_{L_j}=n\varepsilon_j\to0$, this proves the second identity in
\eqref{eq:balanced-pairings}.

It remains to prove \eqref{eq:balanced-mass}.
Choose a fixed K\"ahler form $\kappa\in\alpha$ and $C>0$ with
$\omega\le C\kappa$. The first identity already proved gives
\[
n\int_{Y_j}\pi_j^*\omega\wedge \chi_j^{n-1}
\le Cn\pi_j^*\alpha[\chi_j^{n-1}]\le C,
\]
after increasing $C$.
The trace equation implies $h_j\ge\omega_j/c_j$, so the volume
identity in \eqref{eq:metric-identity} gives
\[
\chi_j^n
\ge c_j^{-n(n-2)/(n-1)}\omega_j^n
\ge c_0(\pi_j^*\omega)^n,
\]
with $c_0>0$ independent of $j$.
\end{proof}

\subsection{Uniqueness from equality in the trace bound}
Our goal is to show there is a unique K\"ahler current in the critical set $\cC_\zeta$. We first show the uniqueness holds under the assumption that all K\"ahler currents in this set satisfy the inverse trace identity for the absolutely continuous part. 
\begin{lemma}\label{lem:trace-uniqueness}
 Suppose every $T\in \cC_a$ satisfies
\begin{equation}\label{eq:saturated-trace}
\tr_{T_{\rm ac}}\omega=a
\quad\hbox{almost everywhere}.
\end{equation}
Then $\mathcal C_a$ consists of one current. 
\end{lemma}

\begin{proof}
Fix $T_1,T_2\in\mathcal C_a$, and write
$T_i=\kappa+dd^cu_i$, where $\kappa\in\alpha$ is a K\"ahler form.
The function $H\mapsto\tr(H^{-1}\omega)$ is strictly
convex on positive Hermitian matrices. Hence $(T_1+T_2)/2\in\mathcal C_a$.
By hypothesis, its inverse trace and those of $T_1,T_2$ all equal $a$
almost everywhere. Equality in the convexity inequality gives
\begin{equation}\label{eq:equal-densities}
(T_1)_{\rm ac}=(T_2)_{\rm ac}=:A
\quad\hbox{almost everywhere}.
\end{equation}
The same argument shows that every member of $\mathcal C_a$ has
absolutely continuous part $A$. We must show that their potentials
differ by constants.

Fix a smooth even convex function $\chi$ with
$\chi(t)=|t|$ for $|t|\ge1$, $|\chi'|\le1$, and
$\chi''>0$ on $(-1,1)$. Define the regularized maximum
\[
M(x,y):=\frac{x+y+\chi(x-y)}2.
\]
Then \[
\max\{x,y\}\le M(x,y)\le\max\{x,y\}+\frac12.
\]
For finite $t$, set
$M(-\infty,t)=M(t,-\infty)=t$, and set
$M(-\infty,-\infty)=-\infty$.
The displayed bound ensures convergence at common poles; when only
one argument tends to $-\infty$, $M$ eventually equals the other,
since $M(x,y)=\max\{x,y\}$ for $|x-y|\ge1$.

Put $h=u_1-u_2$, regarded as an $L^1$ function, and for each
$s\in\mathbb R$ let $w=M(u_1,u_2+s)$.
We claim that $\kappa+dd^cw\in\mathcal C_a$ and
\begin{equation}\label{eq:uniqueness-gradient}
(\kappa+dd^cw)_{\rm ac}
\ge A+\frac{\chi''(h-s)}2\,i\partial h\wedge\bar\partial h
\quad\hbox{almost everywhere}.
\end{equation}
The weak derivatives appearing here exist, as verified below.
The claim implies uniqueness: the left-hand side equals $A$,
so $dh=0$ almost everywhere on $\{|h-s|<1\}$.
Taking all $s\in\mathbb Q$ covers the full-measure set where $h$ is
finite. Thus $dh=0$ almost everywhere on $X$.
Since $h\in W^{1,1}(X)$ and $X$ is connected, $h$ is constant
almost everywhere, and $T_1=T_2$.

It remains to prove the claim. Work on a coordinate ball
$\Omega$ where $\kappa=dd^cr$, and put $U_i=u_i+r$.
Let $U_{i,\delta}=U_i*\varrho_\delta$ be their convolutions with a fixed
nonnegative smooth radial kernel, scaled to radius
$\delta$. On smaller balls define
\[
W_\delta=M(U_{1,\delta},U_{2,\delta}+s).
\]
Since $U_{i,\delta}\downarrow U_i$, monotonicity of $M$ and its
extension at $-\infty$ give
\[
W_\delta\downarrow M(U_1,U_2+s)=:w+r.
\]
Here the last equality follows from
$M(x+t,y+t)=M(x,y)+t$.
The first derivatives of $M$ lie in $[0,1]$, so for every
$K\Subset\Omega$,
\[
\|W_\delta-(w+r)\|_{L^1(K)}
\le\sum_{i=1}^2\|U_{i,\delta}-U_i\|_{L^1(K)}
\longrightarrow0.
\]
The bound by the ordinary maximum also shows $w+r\in L^1_{\rm loc}$.
Consequently $dd^cW_\delta\rightharpoonup\kappa+dd^cw$.

For smooth real functions $v_1,v_2$, the chain rule gives
\[
\begin{aligned}
dd^cM(v_1,v_2)
={}&\frac{1+\chi'(v_1-v_2)}2\,dd^cv_1
 +\frac{1-\chi'(v_1-v_2)}2\,dd^cv_2\\
&+\frac{\chi''(v_1-v_2)}2\,
 i\partial(v_1-v_2)\wedge\bar\partial(v_1-v_2).
\end{aligned}
\]
Thus its Hessian is a convex combination of the two Hessians plus
a nonnegative form. In particular $W_\delta$, and hence $w+r$,
is psh. Fix a constant positive form $B\le\omega$ on $\Omega$.
The convolution argument in Lemma~\ref{lem:attainment}
gives
$\tr_{dd^cU_{i,\delta}}B\le a$.
Convexity  gives $\tr_{dd^cW_\delta}B\le a$.
The argument of Lemma~\ref{lem:attainment} therefore yields
$\kappa+dd^cw\in\mathcal C_a$.

To justify \eqref{eq:uniqueness-gradient}, first note
that $U_i\in W^{1,1}_{\rm loc}$, see \cite[Theorem 1.48]{GZ17}.
 Thus $h\in W^{1,1}(X)$.
 Lebesgue
differentiation therefore gives, almost everywhere,
\[
\begin{aligned}
U_{1,\delta}-U_{2,\delta}
   &=h*\varrho_\delta\longrightarrow h,\\
\partial(U_{1,\delta}-U_{2,\delta})
   &=(\partial h)*\varrho_\delta\longrightarrow\partial h,\\
dd^cU_{i,\delta}&\longrightarrow A.
\end{aligned}
\]
Applying the smooth Hessian identity with
$v_1=U_{1,\delta}$, $v_2=U_{2,\delta}+s$, its right-hand side
therefore converges almost everywhere to
$A+\frac12\chi''(h-s)i\partial h\wedge\bar\partial h$.
For every smooth compactly supported positive $(n-1,n-1)$-form
$\psi$, Fatou's lemma gives
\[
\begin{aligned}
\int_\Omega
 \left(A+\frac{\chi''(h-s)}2\,i\partial h\wedge\bar\partial h\right)
 \wedge\psi
&\le\liminf_{\delta\downarrow0}
 \int_\Omega dd^cW_\delta\wedge\psi\\
&=\langle\kappa+dd^cw,\psi\rangle.
\end{aligned}
\]
This proves local integrability of the nonnegative gradient term and
an inequality of currents. Taking absolutely continuous parts gives
\eqref{eq:uniqueness-gradient}.
\end{proof}

Now we can apply Lemma~\ref{lem:balanced-metrics} to rule out any positive trace defect at the critical threshold.
\begin{theorem}\label{thm:rigidity}
For $n\ge3$,
\begin{equation}\label{eq:critical-equation}
a_*=\zeta,\qquad \mathcal C_\zeta=\{T\},\qquad
\tr((T)_{\rm ac}^{-1}\omega)=\zeta
\quad\hbox{almost everywhere}.
\end{equation}
\end{theorem}

\begin{proof}
Corollary~\ref{cor:threshold-models} gives $a_*=\zeta$ and the residual models required in Lemma~\ref{lem:balanced-metrics}. Lemma~\ref{lem:attainment} gives $\mathcal C_\zeta\ne\varnothing$.
Let $\chi_j$ be the metrics supplied by Lemma~\ref{lem:balanced-metrics}. On the biholomorphic locus, regard them as forms on $X$, and set
\[
 f_j=\bigl(\tr\sqrt{\chi_j^{-1}\omega}\bigr)^2\chi_j^n/\omega^n.
\]
Since the exceptional locus has zero measure, \eqref{eq:balanced-pairings} gives
\[
 \int_X f_j\omega^n=\int_{Y_j}
 \bigl(\tr\sqrt{\chi_j^{-1}\pi_j^*\omega}\bigr)^2\chi_j^n=c_{L_j}^2L_j^n+o(1).
\]

We shall use the uniform lower bound $f_j\ge m>0$.
Indeed, if $\lambda_1,\ldots,\lambda_n$ are the eigenvalues of
$\chi_j$ relative to $\omega$, the arithmetic--geometric mean
inequality and \eqref{eq:balanced-mass} give
\[
\begin{aligned}
f_j
&=\left(\sum_{k=1}^n\lambda_k^{-1/2}\right)^2
  \prod_{k=1}^n\lambda_k\ge n^2\left(\frac{\chi_j^n}{\omega^n}\right)^{(n-1)/n}
 \ge n^2c_0^{(n-1)/n}=:m>0
\end{aligned}
\]
almost everywhere on $X$, uniformly in $j$.

Fix $U\in\mathcal C_\zeta$, and put
$t=\tr(U_{\rm ac}^{-1}\omega)$, so that
$0<t\le\zeta$ almost everywhere.
Lemma~\ref{lem:trace-cs}, applied pointwise
with parameter $t(x)$, yields
\[
nU_{\rm ac}\wedge\chi_j^{n-1}\ge\frac{f_j}{t}\,\omega^n
\]
on the biholomorphic locus.
The local-potential pullback $\pi_j^*U$ is positive and closed
in $\pi_j^*\alpha$. On this locus its absolutely continuous
part corresponds to $U_{\rm ac}$; its singular coefficient
measures are positive. Integrating the preceding inequality
and bounding by the total mass on $Y_j$, we obtain
\begin{equation}\label{eq:pairing-bound}
\begin{aligned}
n\pi_j^*\alpha[\chi_j^{n-1}]
&=n\int_{Y_j}\pi_j^*U\wedge \chi_j^{n-1}\\
&\ge\int_X\frac{f_j}{t}\,\omega^n
 \ge\frac1\zeta\int_X f_j\omega^n.
\end{aligned}
\end{equation}
Here the first equality uses $\ddc \chi_j^{n-1}=0$.

Because $0<t\le\zeta$ and $f_j\ge m$,
\begin{equation}\label{eq:trace-defect}
\frac{f_j}{t}-\frac{f_j}{\zeta}
=\frac{f_j(\zeta-t)}{\zeta t}
\ge\frac{m}{\zeta^2}(\zeta-t)\ge0
\quad\hbox{almost everywhere}.
\end{equation}
Combining this estimate with
\eqref{eq:pairing-bound} gives
\begin{equation}\label{eq:defect-limit}
\begin{aligned}
0\le\frac{m}{\zeta^2}\int_X(\zeta-t)\omega^n
&\le n\pi_j^*\alpha[\chi_j^{n-1}]
 -\frac1\zeta\int_Xf_j\omega^n\\
&=\left(c_{L_j}-\frac{c_{L_j}^2}{\zeta}\right)L_j^n+o(1)
 \longrightarrow0.
\end{aligned}
\end{equation}
The convergence follows from $c_{L_j}\to\zeta$ and
$L_j^n\le\alpha^n$. The integral on the left is independent of $j$;
since its integrand is nonnegative, it follows that
\[
\tr(U_{\rm ac}^{-1}\omega)=t=\zeta
\quad\hbox{almost everywhere}.
\]
This holds for every $U\in\mathcal C_\zeta$.
The family is nonempty, so Lemma~\ref{lem:trace-uniqueness}
now gives $\mathcal C_\zeta=\{T\}$, including equality of the
singular parts. This proves the theorem.
    
\end{proof}

\section{The nonpluripolar product equation and uniqueness}\label{sec:scalar}
We use the uniqueness theorem to prove that the unique element in $\cC_\zeta$ satisfies the nonpluripolar-product $J$-equation. The idea is on each coordinate ball, replacing the currents by solution to the $J$-equation on the ball preserves the plurisubharmonicity and the inverse trace bound. Then uniqueness gives the equation. 

\subsection{Dirichlet replacement}
We need the following results of Guan-Li \cite{GL12} to construct solutions locally. 
\begin{theorem}\label{thm:dirichlet}
    Let $(M,\omega)$ be a compact Hermitian manifold of complex dimension $n>1$ with smooth boundary, let $\chi$ be a smooth real $(1,1)$-form, let $0<\psi\in C^\infty(\overline M)$, and let $\varphi\in C^\infty(\partial M)$. Put $\chi_v=\chi+\ddc v.$
Suppose there is $\underline v\in C^2(\overline M)$ with
\[
\chi_{\underline v}>0,\qquad
\chi_{\underline v}^n\geq
\psi\,\chi_{\underline v}^{n-1}\wedge\omega
\quad\hbox{on }\overline M,\qquad
\underline v=\varphi\quad\hbox{on }\partial M.
\]
Then the Dirichlet problem
\begin{equation}\label{eq:dirichlet}
\chi_v>0,\qquad
\chi_v^n=\psi\,\chi_v^{n-1}\wedge\omega,
\qquad v|_{\partial M}=\varphi
\end{equation}
has a unique admissible smooth solution, with its prescribed continuous boundary values. 
\end{theorem}

We only apply \eqref{eq:dirichlet} to the case where $M$ is a Euclidean coordinate ball $B_R$, $\chi=0$, and $\psi=n/a$. Extend the smooth boundary value $\varphi$ smoothly to $\overline B_R$. For sufficiently large $A$,
\[
\underline v=\widetilde\varphi+A(|z|^2-R^2)
\]
has $dd^c\underline v\geq(n/a)\omega$, hence
$\tr_{dd^c\underline v}\omega\leq a$. This is precisely the subsolution inequality in \eqref{eq:dirichlet}. Consequently every smooth boundary value on the ball admits a smooth interior solution of
\begin{equation}\label{eq:ball-equation}
a(dd^cv)^n=n\omega\wedge(dd^cv)^{n-1},
\qquad dd^cv>0.
\end{equation}

The boundary values of the potential of a current may be $-\infty$.
The following elementary approximation uses distance
sup-convolutions and summable smoothing errors to retain
monotonicity. It prepares the boundary data for
Theorem~\ref{thm:dirichlet}.
\begin{lemma}\label{lem:boundary-data}
Every upper semicontinuous $f:S^{2n-1}\to[-\infty,\infty)$, bounded above and not identically $-\infty$, is the pointwise decreasing limit of smooth functions $\varphi_j$ on the sphere.
\end{lemma}
\begin{proof}

For the distance $d$ of a fixed smooth Riemannian metric, put
\[
g_j(x)=\sup_y\{f(y)-j d(x,y)\}.
\]
Each $g_j$ is finite and $j$-Lipschitz, with
$f\le g_{j+1}\le g_j$.
Given $x$ and a real number $t>f(x)$, upper semicontinuity
gives $r>0$ such that $f(y)\le t$ when $d(x,y)<r$.
Consequently
\[
f(x)\le g_j(x)
\le\max\{t,\sup f-jr\},
\qquad
\limsup_{j\to\infty}g_j(x)\le t.
\]
Taking the infimum over $t>f(x)$ proves $g_j\downarrow f$,
also at points where $f=-\infty$.

By smooth approximation on the sphere, choose $h_j\in C^\infty$
with $\|h_j-g_j\|_\infty\le2^{-j}$, and set
$\varphi_j=h_j+3\cdot2^{-j}$. Then
\[
\begin{gathered}
g_j+2^{1-j}\le\varphi_j\le g_j+2^{2-j},\\
\varphi_{j+1}\le g_{j+1}+2^{1-j}
\le g_j+2^{1-j}\le\varphi_j.
\end{gathered}
\]
Hence $\varphi_j\downarrow f$.\qedhere
\end{proof}

The next statement is a consequence of
Bedford--Taylor convergence (Lemma~\ref{lem:bt-convergence})
and non-pluripolar locality.
\begin{lemma}\label{lem:scalar-limit}
Let $v_j$ be smooth psh functions on a domain $B\subset\mathbb C^n$, decreasing to a psh function $v\not\equiv-\infty$. Suppose $\omega$ is smooth K\"ahler, $a>0$, and
\begin{equation}\label{eq:decreasing-solutions}
a(dd^cv_j)^n=n\omega\wedge(dd^cv_j)^{n-1}.
\end{equation}
Assume the non-pluripolar products of $dd^cv$ are locally finite.  Then \begin{equation}\label{eq:limit-equation}
a\langle(dd^cv)^n\rangle
=n\omega\wedge\langle(dd^cv)^{n-1}\rangle.
\end{equation}
\end{lemma}
\begin{proof}

Work on a ball $B'\Subset B$. For $M\in\mathbb N$, put
\[
\begin{gathered}
w_{j,M}=\max(v_j,-M),\qquad w_M=\max(v,-M),\\
h_M=\min\{1,\max\{0,v+M\}\}
   =1+w_M-\max(v,-M+1).
\end{gathered}
\]
For each fixed $M$,
\[
w_{j,M}\downarrow w_M,\qquad
-M\le w_M\le w_{j,M}\le\max\{\sup_{B'}v_1,-M\}.
\]
The functions $1+w_M$ and $\max(v,-M+1)$ are locally
bounded psh. Apply weighted Bedford--Taylor monotone convergence
\cite[III, Theorem 3.7(a), p.~147]{D12} to each of these
fixed factors, with the fixed closed positive current
$\omega^{n-p}$. Subtracting the resulting limits gives
\[
h_M\omega^{n-p}\wedge(dd^cw_{j,M})^p
\rightharpoonup h_M\omega^{n-p}\wedge(dd^cw_M)^p,
\qquad p=n-1,n.
\]

Since
\[
\{h_M>0\}=\{v>-M\}\subset\{v_j>-M\},
\]
ordinary locality and the non-pluripolar truncation identity
\cite[Definition 1.1 and Proposition 1.4]{BEGZ10} give
\[
\begin{aligned}
h_M(dd^cw_{j,M})^p&=h_M(dd^cv_j)^p,\\
h_M(dd^cw_M)^p&=h_M\langle(dd^cv)^p\rangle,
\qquad p=n-1,n.
\end{aligned}
\]
Multiplying \eqref{eq:decreasing-solutions} by $h_M$
and letting $j\to\infty$ therefore yields
\[
a h_M\langle(dd^cv)^n\rangle
=n h_M\omega\wedge\langle(dd^cv)^{n-1}\rangle.
\]
Since $0\le h_M\uparrow\mathbf1_{\{v>-\infty\}}$, and both
non-pluripolar measures are locally finite and vanish on
$\{v=-\infty\}$, monotone convergence as $M\to\infty$
proves \eqref{eq:limit-equation}.\qedhere
\end{proof}

Now, we adapt the local replacement step in the
Perron's method(see
\cite[Sections~5--6]{BT76}) to the trace equation.
Here Guan--Li supplies the smooth solutions, while
Lemmas~\ref{lem:upper-tests} and
\ref{lem:scalar-limit} justify comparison and
passage to singular boundary data.
\begin{lemma}\label{lem:replacement}
Let $U$ be psh and $\omega$ a smooth K\"ahler form on a
neighborhood of the closed coordinate ball $\overline B_R$.
Assume $\tr_{U_{\rm ac}}\omega\leq a$ there.
There are solutions
$v_j\in C^\infty(B_R)\cap C^0(\overline B_R)$ of
\eqref{eq:ball-equation}, decreasing to a psh
function $v$, such that
\begin{equation}\label{eq:replacement}
v\geq U\quad\hbox{on }B_R,\qquad
\limsup_{B_R\ni z\to p}v(z)\leq U(p)
\quad(p\in\partial B_R).
\end{equation}
The limit is not identically $-\infty$, satisfies $\tr_{(\ddc v)_{\rm ac}}\omega\leq a$ on the ball, and the function
\begin{equation}\label{eq:replaced-equation}
W=\begin{cases}v,&z\in B_R,\\ U,&z\notin B_R\end{cases}
\end{equation}
with boundary value $U$ is psh on a neighborhood of $\overline B_R$ and satisfies $\tr_{(\ddc W)_{\rm ac}}\omega\leq a$ there.
\end{lemma}
\begin{proof}

The maximum principle and $U\not\equiv-\infty$ imply
$U|_{\partial B_R}\not\equiv-\infty$.
Choose smooth $\varphi_j\downarrow U|_{\partial B_R}$ by
Lemma~\ref{lem:boundary-data}, and solve
\eqref{eq:ball-equation} with
$v_j|_{\partial B_R}=\varphi_j$.
Lemma~\ref{lem:upper-tests} gives
\[
U\le v_{j+1}\le v_j.
\]
Hence, by \cite[Chapter I, theorem 5.4]{D12},
\[
v_j\downarrow v\in\operatorname{PSH}(B_R),\qquad
U\le v,\qquad v_j\longrightarrow v\text{ in }L^1_{\rm loc}(B_R).
\]
The weak closedness proved in
Lemma~\ref{lem:attainment} yields
\[
dd^cv\ge\omega/a,\qquad
\tr_{(dd^cv)_{\rm ac}}\omega\le a
\quad\text{almost everywhere on }B_R.
\]
For $p\in\partial B_R$ and every fixed $j$, continuity of
$v_j$ up to the boundary gives
\[
\limsup_{B_R\ni z\to p}v(z)
\le\lim_{B_R\ni z\to p}v_j(z)=\varphi_j(p).
\]
Taking the infimum over $j$ proves
\eqref{eq:replacement}.

Define $W$ by \eqref{eq:replaced-equation}.
The boundary inequality makes $W$ upper semicontinuous.
Moreover, $U\le W$ and $W$ is locally bounded above, so
$W\in L^1_{\rm loc}$.
It is psh away from $\partial B_R$.
At $p\in\partial B_R$, for every unit vector
$\xi\in\mathbb C^n$ and sufficiently small $r>0$,
\[
\begin{aligned}
W(p)=U(p)
&\le\frac1{2\pi}\int_0^{2\pi}U(p+re^{it}\xi)\,dt\\
&\le\frac1{2\pi}\int_0^{2\pi}W(p+re^{it}\xi)\,dt.
\end{aligned}
\]
This inequality is automatic when $U(p)=-\infty$.
The submean characterization on complex lines therefore
proves $dd^cW\ge0$.

Off the sphere, locality gives
\[
(dd^cW)_{\rm ac}=
\begin{cases}
(dd^cv)_{\rm ac}&\text{on }B_R,\\
(dd^cU)_{\rm ac}&\text{outside }\overline B_R.
\end{cases}
\]
Since $\partial B_R$ has zero smooth volume, these identities
and the trace bounds for $U,v$ imply
\[
\tr_{(dd^cW)_{\rm ac}}\omega\le a
\quad\text{almost everywhere}.
\]
Together with $dd^cW\ge0$, this proves the lemma.
\end{proof}
The preceding replacement preserves $\mathcal C_a$.
If this family is a singleton, the current must coincide
with its replacement on every coordinate ball. This is
the local replacement principle underlying the
Perron method \cite{BT76}, applied here using the rigidity
already established.

\begin{theorem}\label{thm:scalar}
Let $X$ be a connected compact K\"ahler manifold of dimension
$n\ge2$, let $\alpha$ be a K\"ahler class, and let $\omega$
be a K\"ahler form. For $a>0$, suppose that the family $\mathcal C_a$ satisfies
\[
\mathcal C_a=\{T\}.
\]
Then $T\ge\omega/a$ and
\begin{equation}\label{eq:scalar-solution}
a\langle T^n\rangle
=n\omega\wedge\langle T^{n-1}\rangle.
\end{equation}
\end{theorem}
\begin{proof}
Lemma~\ref{lem:attainment} gives
$T\ge\omega/a$. Choose a K\"ahler form $\kappa\in\alpha$
and write $T=\kappa+dd^cu$.
On a coordinate neighborhood of a closed ball
$\overline B_R\subsetneq X$, choose $r$ with $dd^cr=\kappa$
and put $U=u+r$. Then $dd^cU=T$.

Lemma~\ref{lem:replacement} gives smooth
solutions $v_j\downarrow v\ge U$ and permits the definition
\[
\widetilde u=
\begin{cases}
v-r&\text{on }B_R,\\
u&\text{on }X\setminus B_R.
\end{cases}
\]
Its positivity and trace conclusions give
\[
\kappa+dd^c\widetilde u
\in\mathcal C_a=\{T\},
\qquad
dd^c(\widetilde u-u)=0.
\]
Thus $\widetilde u-u$ is distributionally $\kappa$-harmonic.
Elliptic regularity and compactness make it constant.
Since it vanishes on $X\setminus\overline B_R$, it is zero;
therefore
\[
v_j\downarrow v=U\quad\text{on }B_R
\]
as plurisubharmonic representatives.

The non-pluripolar products of $T$ are locally finite by
\cite[Proposition 1.6]{BEGZ10}. Applying
Lemma~\ref{lem:scalar-limit} yields
\[
a\langle(dd^cU)^n\rangle
=n\omega\wedge\langle(dd^cU)^{n-1}\rangle
\quad\text{on }B_R.
\]
Such balls cover $X$, proving
\eqref{eq:scalar-solution}.\qedhere
\end{proof}

With the discussion above, we can also give another proof of Datar-Mete-Song's minimal slope conjecture, which was originally proved by Fu \cite{F26}.
Recall that the pair $(\alpha,\beta)$ is \emph{$J$-semistable} if
\[
 (c\alpha^p-p\beta\alpha^{p-1})[V]\ge0
 \quad\text{for every irreducible }V\subsetneq X,\quad 1\le p=\dim V<n.
\]
\begin{corollary}\label{cor:semistable}
Let $X$ be a connected compact K\"ahler manifold of dimension
$n\ge2$, and let $\alpha,\beta$ be K\"ahler classes.
If $(\alpha,\beta)$ is numerically $J$-semistable, then
\[
\zeta
=c=n\frac{\alpha^{n-1}\beta}{\alpha^n}.
\]
\end{corollary}

\begin{proof}
The identity test gives $\zeta\le c$. Fix K\"ahler forms
$\kappa\in\alpha$, $\omega\in\beta$, and choose $C>0$
with $\kappa\le C\omega$. For $\epsilon>0$, set
$\omega_\epsilon=\omega+\epsilon\kappa$. For every proper
positive-dimensional reduced irreducible analytic subvariety
$V\subset X$, with $p=\dim V$, semistability gives
\[
\begin{aligned}
&\bigl((c+n\epsilon)\alpha^p
-p(\beta+\epsilon\alpha)\alpha^{p-1}\bigr)[V]\\
&\qquad=(c\alpha^p-p\beta\alpha^{p-1})[V]
+(n-p)\epsilon\alpha^p[V]>0.
\end{aligned}
\]
Since
\[
n\frac{\alpha^{n-1}(\beta+\epsilon\alpha)}{\alpha^n}
=c+n\epsilon,
\]
the smooth numerical criterion \cite[Corollary 1.2]{S20}
provides a K\"ahler form $h_\epsilon\in\alpha$ satisfying
\[
\tr_{h_\epsilon}\omega_\epsilon=c+n\epsilon.
\]

By Corollary~\ref{cor:critical-family},
there is a K\"ahler current $T\in\alpha$ with
\[
\tr_{T_{\rm ac}}\omega\le\zeta
\quad\text{almost everywhere}.
\]
Write $T=h_\epsilon+dd^cu_\epsilon$, where $u_\epsilon$
is the upper-semicontinuous quasi-psh representative normalized
by $\sup_Xu_\epsilon=0$, and choose a maximum point
$x_\epsilon$. Locally near $x_\epsilon$, write
$h_\epsilon=dd^cr_\epsilon$, with $r_\epsilon$ smooth.
Then $u_\epsilon+r_\epsilon$ is a psh potential of $T$, and
\[
u_\epsilon+r_\epsilon\le r_\epsilon,\qquad
(u_\epsilon+r_\epsilon)(x_\epsilon)=r_\epsilon(x_\epsilon).
\]
Thus $r_\epsilon$ is a smooth upper test, and
Lemma~\ref{lem:upper-tests} gives
$\tr_{h_\epsilon}\omega(x_\epsilon)\le\zeta$.
Since $\omega_\epsilon\le(1+C\epsilon)\omega$,
\[
\begin{aligned}
c+n\epsilon
&=\tr_{h_\epsilon}\omega_\epsilon(x_\epsilon)\\
&\le(1+C\epsilon)
\tr_{h_\epsilon}\omega(x_\epsilon)
\le(1+C\epsilon)\zeta.
\end{aligned}
\]
Letting $\epsilon\downarrow0$ gives $c\le\zeta$, hence
$\zeta=c$.
\end{proof}

\section{An analytic strict subsolution}\label{sec:concentration}

Throughout this section, $n\ge2$, $\kappa\in\alpha$ and
$\omega\in\beta$ are the fixed K\"ahler forms, and $a>0$.
We prove that $\mathcal C_a\ne\varnothing$ implies the existence of a
K\"ahler current in $\alpha$, smooth outside an analytic set, for which
the partial trace is strictly less than $a$. The construction follows
G.~Chen's mass-concentration method \cite[Section~3]{C21},
using the diagonal concentration of Demailly--P\u{a}un \cite{DP04}.
The additional work here is to make the estimates uniform over the modifications.
Bergman regularization then gives the barrier used in
Section~\ref{sec:flow}. 
The convolution characterization
\eqref{eq:convolution} will be applied without mention: for a positive closed current
$R$ and a smooth positive form $\omega$,
\[
 P_\omega(R_{\rm ac})\le a\quad 
 a.e. \quad\Longleftrightarrow\quad
 P_{\omega_0}(R*\varrho_r)\le a
\]
on every coordinate ball, for every constant positive form $\omega_0\le \omega$
and every convolution whose support stays in that ball. 

We start with the regularization. The required forms on the modifications follow from
Theorem~\ref{thm:bergman} and
Lemma~\ref{lem:residual-models}.
Their intersection bounds keep the later concentration
constants uniform.
\begin{lemma}\label{lem:kahler-models}
Suppose $\mathcal C_a\ne\varnothing$. There are modifications
$\pi_j:Y_j\to X$, effective real divisors $D_j$ on the  compact
K\"ahler manifolds $Y_j$, and K\"ahler forms
\begin{equation}\label{eq:kahler-models}
 h_j\in L_j:=\pi_j^*\alpha-[D_j],\qquad
 \Omega_j:=\pi_j^*\omega,\qquad
 \tr_{h_j}\Omega_j\le A_j,\qquad A_j\longrightarrow a.
\end{equation}
For these models,
\begin{equation}\label{eq:mixed-bounds}
 0\le L_j^r(\pi_j^*\beta)^{n-r}
       \le\alpha^r\beta^{n-r}\quad(0\le r\le n),
 \qquad
 L_j^n\ge(n/A_j)^n\int_X\omega^n.
\end{equation}
\end{lemma}

\begin{proof}
Fix $T\in\mathcal C_a$. Apply
Theorem~\ref{thm:bergman}, followed by
Lemma~\ref{lem:residual-models}, with
$\varepsilon_j,t_j\downarrow0$. The resulting forms satisfy
\eqref{eq:kahler-models} with
\[
 A_j=\frac{a+\varepsilon_j}{1-t_j}.
\]

For $1\le r\le n$, the identity
\[
 (\pi_j^*\alpha)^r-L_j^r
 =[D_j]\sum_{\ell=0}^{r-1}
          (\pi_j^*\alpha)^{r-1-\ell}L_j^\ell
\]
has nonnegative intersection with $(\pi_j^*\beta)^{n-r}$: each
factor other than $[D_j]$ is nef, and $D_j$ is effective.
This proves the first assertion in
\eqref{eq:mixed-bounds}, including $r=0$.
Off the exceptional locus, the arithmetic--geometric mean inequality
and $\tr_{h_j}\Omega_j\le A_j$ give
$h_j^n\ge(n/A_j)^n\Omega_j^n$. Integration proves the second.
\end{proof}

Now, we state the geometric concentration step of
Demailly--P\u{a}un \cite[Lemma~2.1(iii)--(iv)]{DP04} for the readers' convenience.
The statement below specializes their construction to a smooth
submanifold and records uniformity for compact subsets of it.
\begin{lemma}[Demailly--P\u{a}un mass concentration]
\label{lem:dp-concentration}
Let $(M,\Theta)$ be a compact K\"ahler manifold of dimension $N$,
and let $Z\subset M$ be a nonempty closed complex submanifold of
codimension $p\ge1$.
Fix a smooth tubular identification with the normal bundle of $Z$,
using the metric induced by $\Theta$.
There are constants $\tau,c_*,s_0,b>0$ and functions
$\psi_s\in C^\infty(M,\mathbb R)$, $0<s<s_0$, such that
\begin{equation}\label{eq:dp-forms}
 \Theta_s:=\Theta+\tau dd^c\psi_s\in[\Theta],
 \qquad \Theta_s\ge\tfrac12\Theta.
\end{equation}
For a compact set $K\subset Z$, let $\mathcal T_s(K)$ be the closed
normal tubular neighborhood of radius $c_*s$ over $K$. Then
\begin{equation}\label{eq:dp-mass}
 \int_{\mathcal T_s(K)}\Theta_s^p\wedge\Theta^{N-p}
       \ge b\int_K\Theta^{N-p}.
\end{equation}
The constants are independent of $K$ and $s$.
\end{lemma}

Now we mimic the proof of G.~Chen \cite{C21} to show the existence of K\"ahler current in the shifted class, satisfying the cone condition in the local convolution sense. The additional point is to make the diagonal mass independent of the modification, which is necessary for passing to the limit.
\begin{proposition}\label{prop:concentration}
Let $\pi_j:Y_j\to X$ be K\"ahler modifications, and let $D_j\ge0$ be effective real divisors.
Suppose there are K\"ahler forms
$h_j\in L_j=\pi_j^*\alpha-[D_j]$ satisfying
$\tr_{h_j}\pi_j^*\omega\le A_j\to a>0$.
There exist $e>0$ and a positive closed current
\begin{equation}\label{eq:shifted-current}
 S\in\alpha-e\beta,\qquad S\ge\omega/a,\qquad P_\omega(S)\le a
\end{equation}
in the weak local-convolution sense.
\end{proposition}

\begin{proof}
We use Chen's mass concentration
\cite[Theorem~1.18 and its proof in Section~3]{C21}.
 We give the details for that estimate before applying
Chen's fiber-integration argument.

\smallskip\noindent
\emph{1. Positive background forms.}
Put $V_\beta=\int_X\omega^n$, $\Omega_j=\pi_j^*\omega$, and
$\epsilon_j=1/j$. Set
\begin{equation}\label{eq:backgrounds}
 g_j:=\Omega_j+\epsilon_jh_j>0,\qquad
 \tr_{h_j}g_j\le a_j:=A_j+n\epsilon_j\longrightarrow a.
\end{equation}
After discarding finitely many indices, assume $a/2\le a_j\le2a$.
The intersection argument in
\eqref{eq:mixed-bounds} applies to these models.
Since $[g_j]=\pi_j^*\beta+\epsilon_jL_j$ and $\epsilon_j\le1$,
\begin{equation}\label{eq:positive-mixed}
 L_j^r[g_j]^{n-r}
 \le\sum_{\ell=0}^{n-r}\binom{n-r}{\ell}
                 \alpha^{r+\ell}\beta^{n-r-\ell}
 \quad(0\le r\le n).
\end{equation}
In particular, $V_\beta\le\int_{Y_j}g_j^n\le C_0$ for a constant
$C_0$ independent of $j$.

\smallskip\noindent
\emph{2. Chen's equation on the product.}
Let $p_1,p_2:Y_j\times Y_j\to Y_j$ be the projections.
Subscripts $1,2$ denote pullback from the corresponding factor of
$X\times X$ or $Y_j\times Y_j$.
Apply Lemma~\ref{lem:dp-concentration} to
$(M,\Theta,Z)=(X\times X,\omega_1+\omega_2,\Delta_X)$, where $\Delta_X$ is the diagonal submanifold.
Write
\[
 \widehat\omega_s=\omega_1+\omega_2+\tau dd^c\psi_s
       \ge\tfrac12(\omega_1+\omega_2).
\]
The functions $\psi_s$, the coefficient $\tau$, and all constants
in that lemma are fixed independently of $j$. Set
\begin{equation}\label{eq:product-backgrounds}
 \begin{gathered}
 \Theta_j=g_{j,1}+g_{j,2},\qquad
 \widehat\Omega_j=\Omega_{j,1}+\Omega_{j,2},\\
 \widehat\Theta_{j,s}
 =\Theta_j+\tau dd^c(\pi_j\times\pi_j)^*\psi_s
 \ge\tfrac12\Theta_j\ge\tfrac12\widehat\Omega_j.
 \end{gathered}
\end{equation}
For $0<t\le1$, define the class and the positive constant
\[
 \begin{aligned}
 \mathcal L_{j,t}
   &=(1+t)p_1^*L_j+a_j^{-1}p_2^*[g_j],\\
 f_{j,t}
   &=\frac{a_j(1+t)^nL_j^n
           -n(1+t)^{n-1}[g_j]L_j^{n-1}}{\int_{Y_j}g_j^n}>0.
 \end{aligned}
\]
Indeed, $n[g_j]L_j^{n-1}\le a_jL_j^n$ by
\eqref{eq:backgrounds}. The K\"ahler form
$(1+t)h_{j,1}+a_j^{-1}g_{j,2}\in\mathcal L_{j,t}$ satisfies
\[
 \tr_{(1+t)h_{j,1}+a_j^{-1}g_{j,2}}\Theta_j
 \le\frac{a_j}{1+t}+na_j<(n+1)a_j.
\]
Fix
\[
 q=\frac1{8n\bigl(2(n+1)a\bigr)^{2n-1}}>0
\]
and define
\[
 F_{j,t,s}
 =a_j^{-n}f_{j,t}
       +q\left(\frac{\widehat\Theta_{j,s}^{2n}}{\Theta_j^{2n}}-1\right).
\]
The hypotheses of \cite[Theorem~1.14]{C21}, in dimension $2n$
and at level $(n+1)a_j$, follow from the preceding strict trace
inequality and
\begin{equation}\label{eq:chen-hypotheses}
 \begin{gathered}
 F_{j,t,s}>-q>-\frac1{4n}\bigl((n+1)a_j\bigr)^{-(2n-1)},\\
 \int_{Y_j\times Y_j}F_{j,t,s}\Theta_j^{2n}
 =\binom{2n}{n}a_j^{-n}f_{j,t}
                           \left(\int_{Y_j}g_j^n\right)^2\\
 =(n+1)a_j\mathcal L_{j,t}^{2n}
       -2n[\Theta_j]\mathcal L_{j,t}^{2n-1}>0.
 \end{gathered}
\end{equation}
Here $[\widehat\Theta_{j,s}]=[\Theta_j]$, which gives the first
integral identity. We therefore obtain K\"ahler forms
$W_{j,t,s}\in\mathcal L_{j,t}$ satisfying
\begin{equation}\label{eq:product-equation}
 \tr_{W_{j,t,s}}\Theta_j
   +F_{j,t,s}\frac{\Theta_j^{2n}}{W_{j,t,s}^{2n}}=(n+1)a_j,
 \qquad
 P_{\Theta_j}(W_{j,t,s})<(n+1)a_j.
\end{equation}
In particular, $W_{j,t,s}\ge((n+1)a_j)^{-1}\Theta_j$.
Since $2n((n+1)a_j)^{1-2n}>q$, the equation yields
\[
 \begin{aligned}
 (n+1)a_jW_{j,t,s}^{2n}
 &=2n\Theta_j\wedge W_{j,t,s}^{2n-1}
       +(a_j^{-n}f_{j,t}-q)\Theta_j^{2n}
       +q\widehat\Theta_{j,s}^{2n}\\
 &\ge q\widehat\Theta_{j,s}^{2n}.
 \end{aligned}
\]
Thus we have the global lower bound
\begin{equation}\label{eq:product-determinant}
 W_{j,t,s}^{2n}\ge
       \frac{q}{2(n+1)a}\widehat\Theta_{j,s}^{2n}.
\end{equation}

\smallskip\noindent
\emph{3. A uniform coefficient on the diagonal.}
Let $\operatorname{Exc}\pi_j$ denote the exceptional locus.
Its image has zero $\omega^n$-measure, so choose
\[
 K_j\Subset X\setminus\pi_j(\operatorname{Exc}\pi_j),
 \qquad \int_{K_j}\omega^n\ge V_\beta/2.
\]
Identify $K_j$ with its image in $\Delta_X$.
For fixed $j$ and small $s$, the tube $\mathcal T_s(K_j)$ from
Lemma~\ref{lem:dp-concentration} lies in the
biholomorphic region of $\pi_j\times\pi_j$; denote its lift by
$\mathcal T_{j,s}$.
Since
\[
 \widehat\Theta_{j,s}
 =(\pi_j\times\pi_j)^*\widehat\omega_s
       +\epsilon_j(h_{j,1}+h_{j,2}),
\]
the measure $d\mu_{j,s}=\widehat\Theta_{j,s}^n\wedge\widehat\Omega_j^n$
has the lower bound
\begin{equation}\label{eq:tube-mass}
 \begin{aligned}
 \mu_{j,s}(\mathcal T_{j,s})
 &\ge\int_{\mathcal T_s(K_j)}
                 \widehat\omega_s^n\wedge(\omega_1+\omega_2)^n\\
 &\ge b\,2^n\int_{K_j}\omega^n
       \ge d_0:=b\,2^{n-1}V_\beta>0.
 \end{aligned}
\end{equation}
Here $b$ is the fixed constant in
\eqref{eq:dp-mass}, and
$(\omega_1+\omega_2)|_{\Delta_X}=2\omega$.
Only the required smallness of $s$ depends on $j$.

Let $\lambda_1\le\cdots\le\lambda_{2n}$ be the eigenvalues of
$W_{j,t,s}$ relative to $\widehat\Theta_{j,s}$.
The mixed-volume formula and
\eqref{eq:positive-mixed} give a uniform constant
$M>0$ such that
\begin{equation}\label{eq:large-eigenvalues}
 \begin{aligned}
 \int (\lambda_{n+1}\cdots\lambda_{2n})
                         \widehat\Theta_{j,s}^{2n}
 &\le\binom{2n}{n}\int W_{j,t,s}^n\wedge\widehat\Theta_{j,s}^n
 \le M,\\
 \int W_{j,t,s}^n\wedge\widehat\Theta_{j,s}^n
 &=\left(\int_{Y_j}g_j^n\right)
   \sum_{r=0}^n\binom nr^2(1+t)^r a_j^{-(n-r)}
                  L_j^r[g_j]^{n-r}.
 \end{aligned}
\end{equation}
All unmarked integrals here are over $Y_j\times Y_j$.
By \eqref{eq:product-backgrounds},
$d\mu_{j,s}\le2^n\widehat\Theta_{j,s}^{2n}$.
With $R=2^{n+1}M/d_0$, we have
\[
 \mu_{j,s}\bigl(\{\lambda_{n+1}\cdots\lambda_{2n}>R\}\bigr)
       \le\frac{2^nM}{R}=\frac{d_0}{2}.
\]
\par\needspace{4\baselineskip}
The subset of $\mathcal T_{j,s}$ where
$\lambda_{n+1}\cdots\lambda_{2n}\le R$ has $\mu_{j,s}$-mass at
least $d_0/2$. On this subset,
\eqref{eq:product-determinant} gives
\[
 \lambda_1\cdots\lambda_n\ge\frac{q}{2(n+1)aR},
 \qquad
 W_{j,t,s}^n\ge(\lambda_1\cdots\lambda_n)\widehat\Theta_{j,s}^n.
\]
Therefore
\begin{equation}\label{eq:diagonal-bound}
 \int_{\mathcal T_{j,s}}W_{j,t,s}^n\wedge\widehat\Omega_j^n
       \ge d_1:=\frac{q\,d_0}{4(n+1)aR}>0.
\end{equation}

For fixed $j$, the cohomology classes bound the masses of
$W_{j,t,s}^n$ and $W_{j,t,s}^{n-1}$ against powers of
$\Theta_j$. Choose $t_k,s_k\downarrow0$ and a common subsequence
such that, with $W_{j,k}=W_{j,t_k,s_k}$,
\[
 W_{j,k}^n\rightharpoonup\Lambda_j,
 \qquad W_{j,k}^{n-1}\rightharpoonup\Lambda'_j.
\]
Both limits are positive and closed. For every fixed small $r>0$,
the compact tube $\mathcal T_{j,r}$ eventually contains
$\mathcal T_{j,s_k}$. The Portmanteau inequality gives
\[
 d_1\le\limsup_{k\to\infty}
       \int_{\mathcal T_{j,r}}W_{j,k}^n\wedge\widehat\Omega_j^n
 \le(\Lambda_j\wedge\widehat\Omega_j^n)(\mathcal T_{j,r}).
\]
Letting $r\downarrow0$ yields
\[
 (\Lambda_j\wedge\widehat\Omega_j^n)
       \bigl(\Delta_{Y_j}\cap p_1^{-1}(\pi_j^{-1}(K_j))\bigr)
       \ge d_1.
\]
The restriction of a positive closed current to an analytic set is
positive and closed by the Skoda--El Mir extension theorem
\cite[Chapter~III, Theorem~2.3]{D12}.
Since $\Delta_{Y_j}$ is connected and smooth of codimension $n$,
the support theorems
\cite[Chapter~III, Theorems~2.10 and~2.13]{D12} therefore give
\begin{equation}\label{eq:diagonal-limits}
 \mathbf1_{\Delta_{Y_j}}\Lambda_j=\theta_j[\Delta_{Y_j}],
 \qquad \mathbf1_{\Delta_{Y_j}}\Lambda'_j=0.
\end{equation}
Here the second equality uses
$\operatorname{codim}\Delta_{Y_j}=n>n-1$.
Since $\widehat\Omega_j|_{\Delta_{Y_j}}=2\Omega_j$,
\begin{equation}\label{eq:uniform-coefficient}
 \begin{gathered}
 \theta_j\,2^n\int_{K_j}\omega^n\ge d_1,
 \qquad \theta_j\ge\theta_*:=\frac{d_1}{2^nV_\beta}>0,\\
 \Lambda_j\ge\theta_*[\Delta_{Y_j}]
       \quad\text{for every }j.
 \end{gathered}
\end{equation}

\smallskip\noindent
\emph{4. Chen's fiber integration and the limit on $X$.}
Define
\begin{equation}\label{eq:fiber-integral}
 \gamma_j=\frac{a_j^{n-1}}{n\int_{Y_j}g_j^n},\qquad
 Q_{j,k}=\gamma_j(p_1)_*(W_{j,k}^n\wedge g_{j,2}),\qquad
 Q_j=\gamma_j(p_1)_*(\Lambda_j\wedge g_{j,2}).
\end{equation}
Expanding the product class gives
\[
 [Q_{j,k}]=(1+t_k)L_j,
 \qquad Q_{j,k}\rightharpoonup Q_j\in L_j,
 \qquad Q_j\ge\gamma_j\theta_j g_j.
\]
We now apply the fiber-integration and eigenvalue-truncation argument
in \cite[proof of Theorem~1.18, Section~3]{C21}, substituting
\[
 (Y_j,g_j,L_j,a_j)
 \quad\text{for Chen's }(M,\chi,[\omega_0],c).
\]
That argument uses the product class of $W_{j,k}$, the inequality
$P_{\Theta_j}(W_{j,k})<(n+1)a_j$, and
\eqref{eq:diagonal-limits}; its conclusion, with the
normalization \eqref{eq:fiber-integral}, is
\begin{equation}\label{eq:chen-subtraction}
 0<e\le\tfrac12\gamma_j\theta_j
 \quad\Longrightarrow\quad
 Q_j-eg_j\ge0,\qquad P_{g_j}(Q_j-eg_j)\le a_j
 \quad\text{weakly}.
\end{equation}
The coefficient removed in the cited proof is exactly
$\gamma_j\theta_j/2$; a smaller coefficient is allowed by
monotonicity. The eigenvalue truncation in that proof is essential:
it uses $\mathbf1_{\Delta_{Y_j}}\Lambda'_j=0$ to preserve the
partial-trace bound after removing the diagonal contribution.
Thus \eqref{eq:chen-subtraction} is an application of
that argument, rather than an inference from $Q_j\ge\gamma_j\theta_jg_j$.
Possible additional components over the exceptional image contribute
positive currents and do not affect this argument.

By \eqref{eq:positive-mixed} and
\eqref{eq:uniform-coefficient}, we may choose
\[
 0<e:=\frac{(a/2)^{n-1}\theta_*}{4nC_0}
       \le\tfrac12\gamma_j\theta_j\qquad\text{for every }j.
\]
Put
\[
 R_j=Q_j-eg_j,\qquad
 U_j=(\pi_j)_*\bigl(R_j+[D_j]+e\epsilon_jh_j\bigr).
\]
These currents are positive and closed, and
\[
 [U_j]
 =(\pi_j)_*\bigl(L_j-e[g_j]+[D_j]+e\epsilon_jL_j\bigr)
 =\alpha-e\beta.
\]
On the complement of the exceptional and divisorial images,
monotonicity gives
$P_\omega((U_j)_{\rm ac})\le a_j$.
Those images have zero Lebesgue measure, while the singular coefficient
measures of $U_j$ are positive. Hence
\[
 U_j\ge\omega/a_j,\qquad
 P_{g_0}(U_j*\varrho_r)\le a_j
\]
for every constant $0<g_0\le\omega$ on the convolution chart,
by convexity and monotonicity. The masses
$\int_XU_j\wedge\kappa^{n-1}=(\alpha-e\beta)\alpha^{n-1}$
are fixed. Passing to a weakly convergent subsequence and then to each
fixed convolution gives \eqref{eq:shifted-current}.
\end{proof}

Adding back the class removed by mass concentration makes
the cone condition strict. We then apply
Lemma~\ref{lem:partial-trace}
to obtain analytic singularity type while retaining a
positive strictness margin
\begin{theorem}
\label{thm:strict-subsolution}
Suppose $\mathcal C_a\ne\varnothing$ for $a>0$.
There are a proper closed analytic set $Z\subset X$, a quasi-psh
function $\psi$, and $\eta>0$ such that
\begin{equation}\label{eq:strict-subsolution}
 \begin{gathered}
 B_{\rm an}=\kappa+dd^c\psi\in\alpha
       \text{ is a K\"ahler current},\qquad
 \psi\in C^\infty(X\setminus Z),\\
 P_\omega(B_{\rm an})\le a-\eta\quad\text{on }X\setminus Z.
 \end{gathered}
\end{equation}
The potential $\psi$ has logarithmic singularity type along a coherent
ideal, tends to $-\infty$ along $Z$, and has a smooth remainder
after principalization. In particular, the conclusion holds for
$a=\zeta$.
\end{theorem}

\begin{proof}
Lemma~\ref{lem:kahler-models} and
Proposition~\ref{prop:concentration}
give $e>0$ and $S\in\alpha-e\beta$ satisfying
\eqref{eq:shifted-current}. Set
\[
 B=S+e\omega\in\alpha,\qquad B\ge(1/a+e)\omega.
\]
At almost every point, let $x_1,\ldots,x_{n-1}$ be the largest
inverse eigenvalues of $S_{\rm ac}$ relative to $\omega$.
Then $\sum_{i=1}^{n-1}x_i\le a$. Since
$x\mapsto x/(1+ex)$ is increasing and concave,
\begin{equation}\label{eq:strict-margin}
 P_\omega(B_{\rm ac})
 =\sum_{i=1}^{n-1}\frac{x_i}{1+ex_i}
 \le\frac{a}{1+ea/(n-1)}=:b<a.
\end{equation}
Choose $0<\varepsilon<a-b$. Apply
Lemma~\ref{lem:partial-trace} to $B$, with
$a_0=b$, to obtain $B_{\rm an}=\kappa+dd^c\psi$, smooth outside
$Z$, with
\[
 B_{\rm an}\ge\frac{\omega}{b+\varepsilon},\qquad
 P_\omega(B_{\rm an})\le b+\varepsilon=a-\eta,
 \qquad \eta:=a-b-\varepsilon>0.
\]
The same lemma gives the stated singularity properties.
Finally, $\mathcal C_\zeta\ne\varnothing$ by
Corollary~\ref{cor:critical-family}.
\end{proof}

\section{Convergence of the J-flow}\label{sec:flow}
In this section we show that the $J$-flow, starting from an arbitrary smooth K\"ahler metric, will converge globally as a current and smoothly outside a proper subvariety. We fix  K\"ahler forms $\kappa\in\alpha$ and $\omega\in\beta$. Consider the $J$-flow with initial metric $\kappa+\ddc\varphi_0>0$, \begin{equation}
    \begin{cases}
        \partial_t\varphi=c-\tr_\varphi\omega,\\
        \varphi(0)=\varphi_0.
    \end{cases}
\end{equation} 
Here $c=\cab=n\alpha^{n-1}\beta/\alpha^n$ is the topological constant determined by $(\alpha,\beta).$
Write $\dot\varphi=\partial_t\varphi$ and
$\ddot\varphi=\partial_t^2\varphi$. In local holomorphic coordinates,
\[
 \omega=i g_{i\bar j}\,dz^i\wedge d\bar z^j,\qquad
 \kappa_\varphi:=\kappa+\ddc\varphi=i\chi_{i\bar j}\,dz^i\wedge d\bar z^j.
\]
With $(\chi^{i\bar j})$ the inverse of $\chi$, the linearized
operator is
\begin{equation}\label{eq:flow-linearization}
 \widetilde\Delta f=h^{i\bar j}f_{i\bar j},\qquad
 h^{i\bar j}=\chi^{i\bar\ell}\chi^{k\bar j}g_{k\bar\ell}.
\end{equation}

Let
\[
 \mathcal C_\zeta=\{T\},\qquad T=\kappa+\ddc\varphi_\infty,\qquad
 \int_X\varphi_\infty\,\omega^n=0,
 \qquad b(t)=\frac1{\beta^n}\int_X\varphi(t)\,\omega^n.
\]
No regularity of $\varphi_\infty$ is assumed here. We summarize some basic facts for the $J$-flow here. The long-time existence was proved by Chen~\cite{C04}, and the others are simple consequences of maximum principle. 
\begin{lemma}\label{lem:flow-facts}
The flow exists for long time, and along the flow,
\begin{enumerate}
    \item $(\partial_t-\widetilde\Delta)\dot\varphi=0$.
    \item $0<\min_X\tr_{{\varphi_0}}\omega
   \le\tr_{\varphi}\omega
   \le\max_X\tr_{{\varphi_0}}\omega,$ so $
  \kappa_\varphi
   \ge\frac{\omega}{\max_X\tr_{{\varphi_0}}\omega}.$
   \end{enumerate}
\end{lemma}

The upper-test criterion of Lemma~\ref{lem:upper-tests} permits comparison
with the possibly singular potential $\varphi_\infty$. We learn this trick from \cite{FLSW14}.
\begin{lemma}\label{lem:flow-comparison}
There is $C>0$ such that
\begin{equation}\label{eq:flow-comparison}
 \varphi(t,x)\ge\varphi_\infty(x)+(c-\zeta)t-C
 \qquad(t\ge0,\ x\in X).
\end{equation}
\end{lemma}
\begin{proof}
Choose $C$ with $\varphi_\infty-C\le\varphi_0$. For $\varepsilon>0$ and
finite $\tau$, set
\[
 H(x,t)=\varphi_\infty(x)+(c-\zeta-\varepsilon)t-C-\varphi(t,x)
 \quad\text{on }X\times[0,\tau].
\]
If $\max H>0$, it is attained at $(x_0,t_0)$ with $t_0>0$ and
$\varphi_\infty(x_0)>-\infty$. The function
\[
 x\longmapsto\varphi(t_0,x)+\varphi_\infty(x_0)-\varphi(t_0,x_0)
\]
touches  $\varphi_\infty$ at $x_0$ from above. Lemma~\ref{lem:upper-tests},
applied after adding a local potential of $\kappa$, and
$\partial_tH(x_0,t_0)\ge0$ imply
\[
 \zeta+\varepsilon
 \le\tr_{\kappa+\ddc\varphi(t_0)}\omega(x_0)\le\zeta.
\]
Thus $H\le0$. Let $\varepsilon\downarrow0$ and then $\tau\to\infty$.
\end{proof}

\subsection{Global weak convergence of $J$-flow }
In this subsection, we show the weak convergence of the $J$-flow. The uniqueness of weak solution to the $J$-equation at minimal slope plays an important role here. 

We start with a simple observation about $L^1$-closeness of $\cC_a$:
\begin{lemma}\label{lem:trace-closure}
Let $\chi_j\in\alpha$ be smooth K\"ahler forms with a common positive
lower bound. Suppose
\[
 \chi_j\rightharpoonup S,\qquad
 \tr_{\chi_j}\omega\le f_j,\qquad
 f_j\longrightarrow a>0\quad\text{in }L^1(X).
\]
Then $S\in\mathcal C_a$.
\end{lemma}
\begin{proof}
Fix a coordinate ball, a constant form $0<B\le\omega$, and a convolution
radius $r>0$. Convexity gives
\begin{equation}\label{eq:trace-closure}
 \tr_{\chi_j*\varrho_r}B
 \le(\tr_{\chi_j}B)*\varrho_r\le f_j*\varrho_r.
\end{equation}
On smaller balls, $\chi_j*\varrho_r\to S*\varrho_r$ smoothly and
$f_j*\varrho_r\to a$ uniformly. The common lower bound therefore gives
$\tr_{S*\varrho_r}B\le a$. Letting $r\downarrow0$ at Lebesgue points,
and choosing constant backgrounds tending to $\omega$ at each such
point, yields $\tr_{S_{\rm ac}}\omega\le a$ almost everywhere.
Positivity, closedness, and the cohomology class pass to the weak limit.
\end{proof}

We recall the standard compactness of normalized
quasi-plurisubharmonic potentials. The proof combines the
Green formula with local $L^1$-compactness
\cite[Chapter~I, Proposition~4.21]{D12}.
\begin{lemma}
\label{lem:normalized-potentials}
There is a constant $C$, depending only on $X,\kappa,\omega$, such
that every smooth $u$ with
$\kappa+dd^cu\ge0$ and $\sup_Xu=0$ satisfies
\begin{equation}\label{eq:potential-integral}
 \int_X|u|\,\omega^n\le C.
\end{equation}
Every sequence of such potentials has a subsequence converging in
$L^1(X)$ to a $\kappa$-psh function with supremum zero.
\end{lemma}
Now we can show the limit of the asymptotic slope. 
\begin{lemma}\label{lem:drift}
\begin{equation}\label{eq:drift}
 \frac{\varphi(t)}t\longrightarrow c-\zeta\quad\text{in }L^1(X),
 \qquad
 \frac{\sup_X\varphi(t)}t\longrightarrow c-\zeta.
\end{equation}
The first convergence is locally uniform wherever $\varphi_\infty$ is
locally bounded below.
\end{lemma}
\begin{proof}
Lemma~\ref{lem:normalized-potentials} and the definition of $b(t)$ give
\begin{equation}\label{eq:flow-normalization}
 \sup_X(\varphi-b)\le C,\qquad
 \frac1{\beta^n}\int_X|\varphi-b|\,\omega^n
 =\frac2{\beta^n}\int_X(\varphi-b)_+\,\omega^n\le2C.
\end{equation}
Integrating \eqref{eq:flow-comparison} gives
$\liminf b(t)/t\ge c-\zeta$. By lemma~\ref{lem:flow-facts}
$|\dot\varphi|\le C$, so $b(t)/t$ is bounded for $t\ge1$.

Suppose $t_j\to\infty$ and $b(t_j)/t_j\to\ell$. Then
$\varphi(t_j)/t_j\to\ell$ in $L^1(X)$. Set
\[
 \overline\chi_j=\frac1{t_j}\int_0^{t_j}
                   \kappa_{\varphi(t)}\,dt\in\alpha.
\]
Convexity and the flow equation give
\begin{equation}\label{eq:average-trace}
 \tr_{\overline\chi_j}\omega
 \le\frac1{t_j}\int_0^{t_j}
          \tr_{\kappa_{\varphi(s)}}\omega\,ds
 =c-\frac{\varphi(t_j)-\varphi_0}{t_j}
 \longrightarrow c-\ell\quad\text{in }L^1(X).
\end{equation}
The forms $\overline\chi_j$ have a uniform lower bound by Lemma~\ref{lem:flow-facts} and
$\int_X\overline\chi_j\wedge\kappa^{n-1}=\alpha^n$. A subsequence
therefore converges weakly to a current in $\mathcal C_{c-\ell}$ by
Lemma~\ref{lem:trace-closure}. Here
$c-\ell\ge\min_X\tr_{\kappa_{\varphi_0}}\omega>0$.
Since $a_*=\zeta$, we obtain $\ell\le c-\zeta$. Thus
$b(t)/t\to c-\zeta$, and \eqref{eq:flow-normalization} proves
\eqref{eq:drift}.

If $\inf_K\varphi_\infty>-\infty$ on a compact set $K$, then
\[
 c-\zeta+\frac{\inf_K\varphi_\infty-C}{t}
 \le\inf_K\frac{\varphi(t)}t
 \le\sup_K\frac{\varphi(t)}t
 \le\frac{\sup_X\varphi(t)}t\longrightarrow c-\zeta.
\]
This proves the local uniform assertion.
\end{proof}

The following differential inequality converts convergence of temporal
averages into convergence at every time.
\begin{lemma}\label{lem:time}
\begin{equation}\label{eq:time}
 \partial_t\bigl(\tr_{\varphi}\omega\bigr)
 \ge-\frac1{2t}\tr_{\varphi}\omega
 \qquad(t>0).
\end{equation}
Consequently $\sqrt t\,\tr_{\varphi(t)}\omega(x)$ is
nondecreasing in $t$ for each $x\in X$.
\end{lemma}
\begin{proof}
At a point, choose $\omega$-normal coordinates in which
$(\chi_{i\bar j})=\operatorname{diag}(\lambda_1,\ldots,\lambda_n)$.
Differentiating $\ddot\varphi=\widetilde\Delta\dot\varphi$ gives
\[
 (\partial_t-\widetilde\Delta)\ddot\varphi
 =-2\sum_{i,j}\frac{|\dot\varphi_{i\bar j}|^2}{\lambda_i^2\lambda_j}.
\]
By Cauchy--Schwarz inequality,
\[
 (\ddot\varphi)^2
 =\left(\sum_i\frac{\dot\varphi_{i\bar i}}{\lambda_i^2}\right)^2
 \le\left(\sum_i\lambda_i^{-1}\right)
          \sum_i\frac{|\dot\varphi_{i\bar i}|^2}{\lambda_i^3}
 \le(c-\dot\varphi)\sum_{i,j}
          \frac{|\dot\varphi_{i\bar j}|^2}{\lambda_i^2\lambda_j}.
\]
For $H=2t\ddot\varphi+\dot\varphi-c$, it follows that
\[
 (\partial_t-\widetilde\Delta)H
 \le2\ddot\varphi-\frac{4t(\ddot\varphi)^2}{c-\dot\varphi}
 =-\frac{2\ddot\varphi}{c-\dot\varphi}H,\qquad
 H(0)=-\tr_{\varphi_0}\omega<0.
\]
On every finite cylinder, the coefficients are smooth and bounded and
$\widetilde\Delta$ is uniformly elliptic. The maximum principle gives
$H\le0$, equivalently
$-\ddot\varphi\ge-(c-\dot\varphi)/(2t)$. This is \eqref{eq:time}.
\end{proof}
\begin{lemma}\label{lem:trace-limit}
Let $a>0$. If $\varphi(t)/t\to c-a$ in $L^1(X)$, then
\[
 \tr_{\varphi(t)}\omega\longrightarrow a
 \quad\text{in }L^1(X).
\]
If the potential convergence is locally uniform on an open set, so is
the trace convergence.
\end{lemma}
\begin{proof}
Fix $R>1$ and $t>0$. Lemma~\ref{lem:time} gives, pointwise on $X$,
\[
 \begin{aligned}
 \tr_{\varphi(s)}\omega
 &\ge\sqrt{t/s}\,\tr_{\varphi(t)}\omega
       &&(t\le s\le Rt),\\
 \tr_{\varphi(s)}\omega
 &\le\sqrt{t/s}\,\tr_{\varphi(t)}\omega
       &&(t/R\le s\le t).
 \end{aligned}
\]
Integrating these inequalities yields
\[
 \begin{aligned}
 \int_t^{Rt}\tr_{\varphi(s)}\omega\,ds
 &\ge2t(\sqrt R-1)\tr_{\varphi(t)}\omega,\\
 \int_{t/R}^t\tr_{\varphi(s)}\omega\,ds
 &\le2t(1-R^{-1/2})\tr_{\varphi(t)}\omega.
 \end{aligned}
\]
By the flow equation, the two integrals equal
$c(R-1)t-\varphi(Rt)+\varphi(t)$ and
$c(1-R^{-1})t-\varphi(t)+\varphi(t/R)$, respectively.
Dividing by the positive coefficients of the trace therefore gives
\begin{equation}\label{eq:time-averages}
 \begin{aligned}
 \tr_{\varphi(t)}\omega
 &\le\frac{\sqrt R+1}{2}
   \left(c-\frac{\varphi(Rt)-\varphi(t)}{(R-1)t}\right),\\
 \tr_{\varphi(t)}\omega
 &\ge\frac{\sqrt R+1}{2\sqrt R}
   \left(c-\frac{\varphi(t)-\varphi(t/R)}{(1-R^{-1})t}\right).
 \end{aligned}
\end{equation}
Both right-hand functions in the parenthesis tend to $a$ in $L^1(X)$. Hence
\[
 \begin{aligned}
 \limsup_{t\to\infty}\int_X
   \bigl(\tr_{\varphi(t)}\omega-a\bigr)_+\,\omega^n
 &\le\frac{\sqrt R-1}{2}\,a\beta^n,\\
 \limsup_{t\to\infty}\int_X
   \bigl(a-\tr_{\varphi(t)}\omega\bigr)_+\,\omega^n
 &\le\frac{\sqrt R-1}{2\sqrt R}\,a\beta^n.
 \end{aligned}
\]
Let $R\downarrow1$. If the potential convergence is uniform on a
compact set, the same argument applies to the suprema of the two
positive parts on that set.
\end{proof}

\begin{proposition}\label{prop:flow-weak}
\begin{equation}\label{eq:flow-weak}
 \tr_{\varphi(t)}\omega\longrightarrow\zeta
       \quad\text{in }L^1(X),\qquad
 \kappa+\ddc\varphi(t)\rightharpoonup T.
\end{equation}
The trace convergence is locally uniform wherever $\varphi_\infty$
is locally bounded below. Moreover,
\begin{equation}\label{eq:flow-level}
 \begin{aligned}
 b'(t)&=c-\frac1{\beta^n}
            \int_X\tr_{\varphi(t)}\omega\,\omega^n
          \longrightarrow c-\zeta,\\
 \partial_t(\varphi-b)
   &=c-b'(t)-\tr_{\varphi}\omega.
 \end{aligned}
\end{equation}
\end{proposition}
\begin{proof}
Lemmas~\ref{lem:drift} and~\ref{lem:trace-limit} give the trace
convergence. Every sequence $t_j\to\infty$
has a subsequence on which $\kappa+\ddc\varphi(t_j)$ converges weakly.
Lemma~\ref{lem:trace-closure}, with
$f_j=\tr_{\varphi(t_j)}\omega$, puts the limit in
$\mathcal C_\zeta=\{T\}$. Thus the entire flow converges.
Differentiating the integral defining $b(t)$ proves
\eqref{eq:flow-level}.
\end{proof}
In the semistable case, the perturbed solutions used in
Corollary~\ref{cor:semistable} give global comparison functions.
Lemma~\ref{lem:trace-limit} then yields uniform convergence on $X$. We remark that when $\zeta<c$, one cannot expect uniform  convergence for the trace, since the average of the trace with respect to $\kappa_\varphi^n$ is always equal to $c$. 
\begin{corollary}\label{cor:uniform-trace}
Suppose that $(\alpha,\beta)$ is $J$-semistable. Then
\begin{equation}\label{eq:uniform-trace}
 \left\|\frac{\varphi(t)}t\right\|_{L^\infty(X)}\longrightarrow0,
 \qquad
 \|\dot\varphi(t)\|_{L^\infty(X)}
 =\|\tr_{\kappa+\ddc\varphi(t)}\omega-c\|_{L^\infty(X)}
 \longrightarrow0.
\end{equation}
\end{corollary}
\begin{proof}
Choose $C>0$ with $\kappa\le C\omega$. For every $\varepsilon>0$,
the proof of Corollary~\ref{cor:semistable}, using
\cite[Corollary~1.2]{S20}, provides a smooth K\"ahler form
\[
 \kappa_\varepsilon=\kappa+\ddc u_\varepsilon\in\alpha,
 \qquad
 \tr_{\kappa_\varepsilon}(\omega+\varepsilon\kappa)=c+n\varepsilon.
\]
Since $\omega\le\omega+\varepsilon\kappa\le(1+C\varepsilon)\omega$,
\[
 \frac{c+n\varepsilon}{1+C\varepsilon}
 \le\tr_{\kappa_\varepsilon}\omega\le c+n\varepsilon,
 \qquad
 \delta_\varepsilon
 :=\|\tr_{\kappa_\varepsilon}\omega-c\|_{L^\infty(X)}
 \longrightarrow0.
\]
Set $C_\varepsilon=\|\varphi_0-u_\varepsilon\|_{L^\infty(X)}$.
The functions
\[
 u_\varepsilon-\delta_\varepsilon t-C_\varepsilon,
 \qquad
 u_\varepsilon+\delta_\varepsilon t+C_\varepsilon
\]
are respectively a subsolution and a supersolution of
\eqref{eq:flow}, because their spatial metrics equal $\kappa_\varepsilon$
and
\[
 -\delta_\varepsilon
 \le c-\tr_{\kappa_\varepsilon}\omega\le\delta_\varepsilon.
\]
They enclose $\varphi_0$ at $t=0$. The parabolic comparison principle
on each finite cylinder gives
\[
 |\varphi(t,x)-u_\varepsilon(x)|
 \le C_\varepsilon+\delta_\varepsilon t.
\]
Thus, for each fixed $\varepsilon>0$,
\[
 \limsup_{t\to\infty}
 \left\|\frac{\varphi(t)}t\right\|_{L^\infty(X)}
 \le\delta_\varepsilon.
\]
Let $\varepsilon\downarrow0$ and apply
Lemma~\ref{lem:trace-limit} on $X$, with $a=c$.
No bound on $\|u_\varepsilon\|_{L^\infty(X)}$ uniform in
$\varepsilon$ is required.
\end{proof}
\subsection{Uniform estimates off a subvariety}
We use the strict subsolution from Section~\ref{sec:concentration} to
bound the potential and the metric away from its pole set. Fix
\[
 B=\kappa+\ddc\psi\ge b_0\omega,\qquad
 \sup_X\psi=0,\qquad Z=\{\psi=-\infty\},
\]
where $\psi$ has analytic singularity type and is smooth on
$X\setminus Z$, and
\begin{equation}\label{eq:flow-barrier}
 P_\omega(B)\le\zeta-\eta\quad\text{on }X\setminus Z,
 \qquad b_0>0,\quad 0<\eta<\zeta.
\end{equation}
Such a $B$ exists by Theorem~\ref{thm:strict-subsolution}.

Recall the following calculations of Weinkove, and Song-Weinkove \cite{W06,SW08}, \begin{lemma}\label{lem:flow-trace}
There is $C_\omega\ge0$, depending only on $\omega$ and $n$, such that
\begin{equation}\label{eq:flow-trace}
 (\widetilde\Delta-\partial_t)
       \log\tr_\omega(\kappa_\varphi)
 \ge-C_\omega\tr_h\omega.
\end{equation}
If $\lambda_1,\ldots,\lambda_n$ are the eigenvalues of
$\kappa+\ddc\varphi$ relative to $\omega$, then
$\tr_h\omega=\sum_i\lambda_i^{-2}$.
\end{lemma}

Now we can use the singular strict subsolution to get the weighted $C^2$-estimate. 
\begin{lemma}\label{lem:flow-weight} There are $t_0,N,C>0$ such that, for every
$\tau\ge t_0$, setting
\[
 m_\tau=\inf_{(X\setminus Z)\times[t_0,\tau]}
                    \bigl(\varphi(x,t)-b(t)-\psi(x)\bigr),
\]
one has
\begin{equation}\label{eq:flow-weight}
 \tr_\omega(\kappa_\varphi(t))
 \le C\exp\!\left[N\bigl(\varphi(t)-b(t)-\psi-m_\tau\bigr)\right]
 \quad\text{on }X\setminus Z,\quad t_0\le t\le\tau.
\end{equation}
The constants are independent of $\tau$.
\end{lemma}
\begin{proof}
Choose $N$ sufficiently large that
\[
 \widehat B=B-\frac{C_\omega}{N}\omega\ge\frac{b_0}{2}\omega,
 \qquad P_\omega(\widehat B)\le\zeta-\frac\eta2.
\]
Choose $t_0$ with $c-b'(t)\ge\zeta-\eta/4$ for $t\ge t_0$. This is possible by Proposition~\ref{prop:flow-weak}.
The flow equation gives
\[
 \begin{aligned}
 \widetilde\Delta(\varphi-b-\psi)
   &=\tr_{\varphi}\omega-\tr_h B,\\
 (\widetilde\Delta-\partial_t)(\varphi-b-\psi)
   &=2\tr_{\varphi}\omega-(c-b')-\tr_h B.
 \end{aligned}
\]
Consider
\[
 H=\log\tr_\omega(\kappa+\ddc\varphi)-N(\varphi-b-\psi).
\]
On each finite cylinder, $H\to-\infty$ uniformly as $x\to Z$.
If its maximum occurs at a time greater than $t_0$, then
\eqref{eq:flow-trace} gives, in an $\omega$-unitary frame diagonalizing
$\kappa+\ddc\varphi$,
\begin{equation}\label{eq:flow-maximum}
 0\ge(\widetilde\Delta-\partial_t)H
 \ge N\left(c-b'
      +\sum_i\frac{\widehat B_{i\bar i}}{\lambda_i^2}
      -2\sum_i\lambda_i^{-1}\right).
\end{equation}
For the coordinate hyperplane $H_j=\{z_j=0\}$, Cauchy--Schwarz
gives
\begin{equation}\label{eq:diagonal-compression}
 \sum_{i\ne j}\frac1{\widehat B_{i\bar i}}
 \le\tr\bigl((\widehat B|_{H_j})^{-1}\bigr)
 \le P_\omega(\widehat B).
\end{equation}
Indeed, each diagonal entry of the inverse is at least the reciprocal
of the corresponding diagonal entry. Completing squares now yields
\[
 \begin{aligned}
 \sum_i\frac{\widehat B_{i\bar i}}{\lambda_i^2}
       -2\sum_i\lambda_i^{-1}
 &\ge-\sum_{i\ne j}\frac1{\widehat B_{i\bar i}}
       -\frac2{\lambda_j}\\
 &\ge-P_\omega(\widehat B)-\frac2{\lambda_j}.
 \end{aligned}
\]
Consequently $0\ge\eta/4-2/\lambda_j$ for every $j$, and
\[
 \tr_\omega(\kappa_\varphi)\le\frac{8n}{\eta},
 \qquad H\le\log(8n/\eta)-Nm_\tau
\]
at this maximum. If the maximum occurs at $t_0$, then
\[
 H+Nm_\tau
 \le\max_{X\setminus Z}H(t_0)
      +N\inf_{X\setminus Z}(\varphi(t_0)-b(t_0)-\psi),
\]
which is independent of $\tau$. Thus $H\le C-Nm_\tau$ in both cases,
proving \eqref{eq:flow-weight}.
\end{proof}

To bound the infimum $m_\tau$, we use the following form of ABP estimate \cite[Proposition~11]{S18}. 
\begin{lemma}
\label{lem:contact-volume}
Let $B_r\subset\mathbb R^d$ be a Euclidean ball, and let
$v$ be smooth near $\overline B_r$, with
\[
 v\ge0,\qquad v(0)=0,\qquad
 v|_{\partial B_r}\ge\varepsilon r^2,\qquad \varepsilon>0.
\]
For the measurable set
\[
 E=\left\{x\in B_r:
 v(x)\le\frac{\varepsilon r^2}{4},\
 |Dv(x)|\le\frac{\varepsilon r}{4},\
 D^2v(x)\ge0\right\},
\]
one has
\begin{equation}\label{eq:contact-volume}
 |B_{\varepsilon r/4}|
       \le\int_E\det_{\mathbb R}D^2v.
\end{equation}
Here $|\cdot|$ denotes Euclidean volume.
\end{lemma}

We apply Lemma~\ref{lem:contact-volume} near a minimum of
$\varphi-b-\psi$ to obtain a bound independent of time.
\begin{lemma}\label{lem:flow-potential}
There are constants $C,N>0$,
independent of $t$, such that
\begin{equation}\label{eq:flow-potential}
 \psi-C\le\varphi(t)-b(t)\le C,\qquad
 \tr_\omega(\kappa_{\varphi(t)})\le Ce^{-N\psi}
 \quad\text{on }X\setminus Z.
\end{equation}
\end{lemma}
\begin{proof}
By \eqref{eq:flow-normalization},
\begin{equation}\label{eq:flow-integral}
 \sup_X(\varphi-b)\le C_0,\qquad
 \frac1{\beta^n}\int_X|\varphi-b|\,\omega^n\le2C_0.
\end{equation}
For $\tau\ge t_0$, from Lemma~\ref{lem:flow-weight}, choose 
$(p,t_*)$ where $\varphi-b-\psi$ attains $m_\tau$.

Choose a coordinate ball $B_r$ centered at $p$, with $r>0$ uniform
over a finite atlas, and $\varepsilon>0$ such that
$\varepsilon\ddc|z|^2\le b_0\omega$.
Choose a smooth nondecreasing $\theta:[0,\infty)\to[0,\infty)$ with
\[
 \theta(s)=s\quad(0\le s\le\varepsilon r^2/2),\qquad
 \theta(s)\ge\varepsilon r^2/2\quad(s\ge\varepsilon r^2/2),
\]
and constant on $[\varepsilon r^2,\infty)$. Define
\[
 v(z)=\theta\bigl(\varphi(t_*,z)-b(t_*)-\psi(z)-m_\tau\bigr)
          +\varepsilon|z|^2.
\]
The first term is constant near $Z\cap\overline B_r$, so $v$ extends
smoothly across $Z$. Moreover,
\[
 v\ge0,\qquad v(0)=0,\qquad
 v|_{\partial B_r}\ge\varepsilon r^2.
\]
On the contact set $E$ of Lemma~\ref{lem:contact-volume},
\[
 E\cap Z=\varnothing,\qquad
 0\le\varphi(t_*)-b(t_*)-\psi-m_\tau\le\varepsilon r^2/4,
\]
and
\[
 0\le\ddc v
 =\kappa+\ddc\varphi(t_*)-B+\varepsilon\ddc|z|^2
 \le\kappa+\ddc\varphi(t_*).
\]
By \eqref{eq:flow-weight},
$\tr_\omega(\kappa+\ddc\varphi(t_*))\le Ce^{N\varepsilon r^2/4}$ on $E$.
Since $D^2v\ge0$ there,
\[
 \det_{\mathbb R}D^2v
 \le4^n\det_{\mathbb C}(v_{i\bar j})^2\le C.
\]
Equation~\eqref{eq:contact-volume} and comparability of coordinate and
$\omega$-volumes give $\int_E\omega^n\ge c_1>0$, uniformly in
$p,t_*,\tau$. As $\psi\le0$, we also have
$\varphi(t_*)-b(t_*)\le m_\tau+\varepsilon r^2/4$ on $E$. Therefore
\[
 c_1(-m_\tau-\varepsilon r^2/4)_+
 \le\int_E|\varphi(t_*)-b(t_*)|\,\omega^n
 \le2C_0\beta^n.
\]
Thus $m_\tau\ge-C$ for every $\tau$. Substitution in
\eqref{eq:flow-weight}, together with \eqref{eq:flow-integral}, proves
\eqref{eq:flow-potential} for $t\ge t_0$.
Enlarging $C$ covers the smooth flow on $[0,t_0]$.
\end{proof}

Now standard Evans-Krylov estimates give the higher order estimate. Together with the uniqueness of the limit, we get the following
\begin{theorem}
For every smooth $\varphi_0$ with $\kappa+\ddc\varphi_0>0$, the solution to \eqref{eq:flow}
exists for all $t\ge0$. The analytic set $Z$ in
Theorem~\ref{thm:main} can be chosen using only
$(X,\alpha,\beta,\kappa,\omega)$ so that
\begin{equation}
 \begin{aligned}
 \kappa+\ddc\varphi(t)&\rightharpoonup T
       &&\text{on }X,\\
 \kappa+\ddc\varphi(t)&\longrightarrow T
       &&\text{in }C^\infty_{\rm loc}(X\setminus Z).
 \end{aligned}
\end{equation}
For
\[
 b(t)=\frac1{\beta^n}\int_X\varphi(t)\,\omega^n,
\]
the potentials $\varphi(t)-b(t)$ converge in $L^1(X)$ and
$C^\infty_{\rm loc}(X\setminus Z)$ to the potential of $T$ with
zero $\omega^n$-mean. Moreover,
\[
 \frac{b(t)}t,\ b'(t)\longrightarrow c-\zeta,\qquad
 \tr_{\varphi(t)}\omega\longrightarrow\zeta
       \quad\text{in }L^p(X,\omega^n),\quad 1\le p<\infty.
\]
\end{theorem}
\begin{proof}
    Only the last $L^p$ convergence needs to be explained. It is because \[
 \|\tr_{\varphi(t)}\omega-\zeta\|_{L^p(\omega^n)}^p
 \le C^{p-1}
    \|\tr_{\varphi(t)}\omega-\zeta\|_{L^1(\omega^n)}
 \longrightarrow0\qquad(1\le p<\infty).
\]
\end{proof}
\section{The J-null locus in the semistable case}\label{sec:null}

In this section, assume $n\ge2$ and that $(\alpha,\beta)$ is
$J$-semistable. Thus $\zeta=c$ by Corollary~\ref{cor:semistable}.
We combine Fang--Ma's null-locus theorem with
Lemma~\ref{lem:partial-trace} to prescribe the singularity type
of a strict subsolution whose pole set is the numerical null
locus. This also specifies the smooth convergence locus in
Section~\ref{sec:flow}.

For a reduced irreducible $V\subsetneq X$ of dimension $1\le p<n$,
write
\begin{equation}\label{eq:null-defect}
 J_p(V)=(c\alpha^p-p\beta\alpha^{p-1})[V]\ge0,
 \qquad
 \Null_J(\alpha,\beta)=\bigcup_{J_{\dim V}(V)=0}V.
\end{equation}
We retain $T$ for the solution at the birational minimal slope and
use $B$ for strict subsolutions.

\subsection{Analytic characterization}

A closed positive current $B=\kappa+\ddc\psi\in\alpha$ is a
\emph{strict weak subsolution} if, for some $0<\eta<c$,
\begin{equation}\label{eq:null-strict}
 P_\omega(B_{\rm ac})\le c-\eta
 \quad\text{almost everywhere}.
\end{equation}
By \eqref{eq:convolution}, this is equivalent to the corresponding
local-convolution inequality. Write
\[
 \Pole(B)=\{\psi=-\infty\}.
\]

\begin{definition}\label{def:null-families}
Let $\mathcal R_\omega$ be the family of strict weak subsolutions
with logarithmic singularity type and smooth remainder after
principalization, in the sense of Section~\ref{sec:bergman}, which
are smooth outside their pole sets. Let $\mathcal B_\omega$ be
the subfamily whose potentials have logarithmic singularities with smooth remainder
on $X$ (this is so called the logarithmic singularity in \cite{CT15,L26a}: locally,
\[
 \psi=a\log\sum_j|f_j|^2+g,\qquad
 a>0,\quad f_j\in\cO,\quad g\in C^\infty.
\]
Define
\begin{equation}\label{eq:null-analytic-loci}
 E_{\rm res}(\omega)=\bigcap_{B\in\mathcal R_\omega}\Pole(B),
 \qquad
 E_{\rm base}(\omega)=\bigcap_{B\in\mathcal B_\omega}\Pole(B),
\end{equation}
with an empty intersection interpreted as $X$.
\end{definition}

For either family,
$\Pole(B)=\{x:\nu(B,x)>0\}$, since a bounded remainder does not
change Lelong numbers. A smooth remainder on $X$ pulls back to
a smooth remainder after principalization; hence
\[
 \mathcal B_\omega\subseteq\mathcal R_\omega,
 \qquad E_{\rm res}(\omega)\subseteq E_{\rm base}(\omega).
\]
These definitions distinguish two singularity conventions in the
analogue of the null-locus characterization of Collins--Tosatti
\cite{CT15}.

We summarize Fang--Ma's theorem here. Their forms $\omega,\chi$
correspond to $\kappa,\omega$ here. The assertions about
irreducible components are contained in
\cite[Proposition~5.2 and the proof of Theorem~5.1]{FM26}.

\begin{theorem}[{Fang--Ma \cite[Theorems~4.3 and~5.1]{FM26}}]
\label{thm:null-current}
The set $\Null_J(\alpha,\beta)$ in \eqref{eq:null-defect} is proper analytic and
has finitely many irreducible components, each of positive
dimension and $J$-null. There are $0<\tau<1$, $0<\eta<c$, and
a quasi-psh function $v$ such that
\[
 \begin{gathered}
 S=(1-\tau)\kappa+\ddc v\ge0,\qquad
 P_\omega(S_{\rm ac})\le c-\eta\quad\text{a.e.},\\
 \{v=-\infty\}=N,\qquad v\in C^\infty(X\setminus N).
 \end{gathered}
\]
Moreover, every strict weak subsolution $B\in\alpha$ that is
smooth outside its exact proper analytic pole set $Z$ satisfies
$N\subseteq Z$.
\end{theorem}

For the last assertion, Fang--Ma's normalization is recovered
by scaling. If $B=\kappa+\ddc\psi$ and
$P_\omega(B_{\rm ac})\le d<c$, choose $0<\sigma<1-d/c$. Then
\[
 (1-\sigma)B=(1-\sigma)\kappa+\ddc((1-\sigma)\psi),
 \qquad
 P_\omega((1-\sigma)B_{\rm ac})\le\frac d{1-\sigma}<c.
\]
After subtracting a constant from the potential, this is a
member of their defining family with the same pole set.
Theorem~\ref{thm:null-current} asserts smoothness of $v$ off
$\Null_J(\alpha,\beta)$, without asserting logarithmic singularity type.
The next theorem obtains this singularity type by applying Bergman kernel approximation in Section~\ref{sec:bergman}.

\begin{theorem}\label{thm:null-locus}
There exists $B\in\mathcal R_\omega$ with $\Pole(B)=\Null_J(\alpha,\beta)$.
Consequently,
\begin{equation}\label{eq:null-equality}
 E_{\rm res}(\omega)=\Null_J(\alpha,\beta).
\end{equation}
\end{theorem}
\begin{proof}
Take $S,\tau,\eta,v$ from Theorem~\ref{thm:null-current}.
Monotonicity of $P_\omega$ gives
\[
 B_0=S+\tau\kappa=\kappa+\ddc v\in\alpha,\qquad
 B_0\ge\frac{\omega}{c-\eta},\qquad
 P_\omega((B_0)_{\rm ac})\le c-\eta.
\]
For $0<\varepsilon<\eta$, Lemma~\ref{lem:partial-trace} gives
$B_\varepsilon\in\mathcal R_\omega$ with
$P_\omega((B_\varepsilon)_{\rm ac})\le c-\eta+\varepsilon$.

We check that the construction introduces no poles outside $\Null(\alpha,\beta)$.
Every local potential $u$ of $B_0$ is smooth there, so
$\mathcal I(mu)=\cO$ there for every $m>0$.
Lemma~\ref{lem:common-ideal} and the gluing in
Theorem~\ref{thm:bergman} therefore give
$\Pole(B_\varepsilon)\subseteq N$.
The minimality assertion of Theorem~\ref{thm:null-current},
applied to every member of $\mathcal R_\omega$, now yields
\[
 \Null_J(\alpha,\beta)\subseteq E_{\rm res}(\omega)
 \subseteq\Pole(B_\varepsilon)\subseteq \Null_J(\alpha,\beta).
\]
\end{proof}

\begin{corollary}\label{cor:null-removal}
Let $B_*\in\mathcal R_\omega$, and let $V$ be a
positive-dimensional irreducible component of $\Pole(B_*)$
with $J_{\dim V}(V)>0$. There is $B\in\mathcal R_\omega$
smooth near a general point of $V$.
\end{corollary}
\begin{proof}
Theorem~\ref{thm:null-current} gives $N\subseteq\Pole(B_*)$.
If $V\subseteq N$, then $V$ would be an irreducible component
of $N$, and hence $J$-null. Thus $V\nsubseteq N$.
The current in Theorem~\ref{thm:null-locus} is smooth on
$X\setminus N$.
\end{proof}

The local construction underlying this consequence is
\cite[Proposition~5.4 and the proof of Theorem~5.1]{FM26}.
It uses quantitative growth estimates and extension of a
strict subsolution near a subvariety. The prescribed
singularity type above is supplied separately by
Lemma~\ref{lem:partial-trace}.

The same realizing current gives the following refinement
of Theorem~\ref{thm:flow}.

\begin{corollary}\label{cor:null-flow}
In the $J$-semistable case, the analytic set in
Theorems~\ref{thm:main} and~\ref{thm:flow} can be chosen to be
$\Null_J(\alpha,\beta)$. In particular,
\[
 T\in C^\infty(X\setminus N),\qquad
 \kappa+\ddc\varphi(t)\longrightarrow T
 \quad\text{in }C^\infty_{\rm loc}(X\setminus N).
\]
\end{corollary}
\begin{proof}
By Corollary~\ref{cor:semistable}, $\zeta=c$.
Choose $B=\kappa+\ddc\psi$ from Theorem~\ref{thm:null-locus},
with $\sup_X\psi=0$. It satisfies
\eqref{eq:strict-subsolution} with $a=\zeta$ and $Z=N$, and
$\psi\to-\infty$ along $N$.
Apply Lemmas~\ref{lem:flow-weight} and~\ref{lem:flow-potential},
followed by the local regularity and convergence argument in
Section~\ref{sec:flow}.
\end{proof}

\begin{remark}
    Replacing $\mathcal R_\omega$ by $\mathcal B_\omega$ requires
a stronger regularization statement which requires an approximation with analytic singularity type and smooth remainder on original manifold, which is not clear so far. Thus
\eqref{eq:null-equality} does not identify the locus defined
using smooth logarithmic remainders on $X$.
\end{remark}

\printbibliography
\end{document}